\documentclass[9pt]{amsart}
\usepackage{amsmath}
\usepackage{amsxtra}
\usepackage{amscd}
\usepackage{amsthm}
\usepackage{amsfonts}
\usepackage{amssymb}
\usepackage{eucal}
\usepackage[all]{xy}
\usepackage{graphicx}
\usepackage[hidelinks]{hyperref}
\usepackage{tikz-cd}

\newtheorem{cor}[subsubsection]{Corollary}
\newtheorem{lem}[subsubsection]{Lemma}
\newtheorem{prop}[subsubsection]{Proposition}
\newtheorem{propconstr}[subsubsection]{Proposition-Construction}
\newtheorem{thmconstr}[subsubsection]{Theorem-Construction}

\newtheorem{conj}[subsubsection]{Conjecture}
\newtheorem{thm}[subsubsection]{Theorem}

\theoremstyle{definition}

\theoremstyle{remark}
\newtheorem{rem}[subsubsection]{Remark}

\newcommand{\thmref}[1]{Theorem~\ref{#1}}

\newcommand{\secref}[1]{Sect.~\ref{#1}}
\newcommand{\lemref}[1]{Lemma~\ref{#1}}
\newcommand{\propref}[1]{Proposition~\ref{#1}}
\newcommand{\corref}[1]{Corollary~\ref{#1}}
\newcommand{\conjref}[1]{Conjecture~\ref{#1}}

\numberwithin{equation}{section}

\newcommand{\nc}{\newcommand}
\nc{\renc}{\renewcommand}
\nc{\ssec}{\subsection}
\nc{\sssec}{\subsubsection}
\nc{\on}{\operatorname}

\nc\ol{\overline}
\nc\wt{\widetilde}
\nc\tboxtimes{\wt{\boxtimes}}
\nc\tstar{\wt{\star}}
\nc{\alp}{\alpha}

\nc{\ZZ}{{\mathbb Z}}
\nc{\NN}{{\mathbb N}}
\nc{\OO}{{\mathbb O}}
\renc{\SS}{{\mathbb S}}
\nc{\DD}{{\mathbb D}}
\nc{\GG}{{\mathbb G}}

\nc{\Fq}{{\mathbb F}_q}
\nc{\Fqb}{\ol{{\mathbb F}_q}}
\nc{\Ql}{\ol{{\mathbb Q}_\ell}}
\nc{\id}{\text{id}}
\nc\X{\mathcal X}

\nc{\Hom}{\on{Hom}}
\nc{\Lie}{\on{Lie}}
\nc{\Loc}{\on{Loc}}
\nc{\Pic}{\on{Pic}}
\nc{\Bun}{\on{Bun}}
\nc{\IC}{\on{IC}}
\nc{\Aut}{\on{Aut}}
\nc{\rk}{\on{rk}}
\nc{\Sh}{\on{Sh}}
\nc{\Perv}{\on{Perv}}
\nc{\pos}{{\on{pos}}}
\nc{\Conv}{\on{Conv}}
\nc{\Sph}{\on{Sph}}
\nc{\Sym}{\on{Sym}}
\nc{\BunBb}{\overline{\Bun}_B}
\nc{\BunNb}{\overline{\Bun}_N}
\nc{\BunTb}{\overline{\Bun}_T}
\nc{\BunBbm}{\overline{\Bun}_{B^-}}
\nc{\BunBbel}{\overline{\Bun}_{B,el}}
\nc{\BunBbmel}{\overline{\Bun}_{B^-,el}}
\nc{\Buno}{\overset{o}{\Bun}}
\nc{\BunPb}{{\overline{\Bun}_P}}
\nc{\BunBM}{\Bun_{B(M)}}
\nc{\BunBMb}{\overline{\Bun}_{B(M)}}
\nc{\BunPbw}{{\widetilde{\Bun}_P}}
\nc{\BunBP}{\widetilde{\Bun}_{B,P}}
\nc{\GUb}{\overline{G/U}}
\nc{\GUPb}{\overline{G/U(P)}}

\nc{\Hhom}{\underline{\on{Hom}}}
\nc\syminfty{\on{Sym}^{\infty}}
\nc\lal{\ol{\lambda}}
\nc\xl{\ol{x}}
\nc\thl{\ol{\theta}}
\nc\nul{\ol{\nu}}
\nc\mul{\ol{\mu}}
\nc\Sum\Sigma
\nc{\oX}{\overset{o}{X}{}}
\nc{\hl}{\overset{\leftarrow}h{}}
\nc{\hr}{\overset{\rightarrow}h{}}
\nc{\M}{{\mathcal M}}
\nc{\N}{{\mathcal N}}
\nc{\F}{{\mathcal F}}
\nc{\D}{{\mathcal D}}
\nc{\Q}{{\mathcal Q}}
\nc{\Y}{{\mathcal Y}}
\nc{\G}{{\mathcal G}}
\nc{\E}{{\mathcal E}}
\nc{\CalC}{{\mathcal C}}
\nc\Dh{\widehat{\D}}

\nc{\C}{{\mathcal C}}
\nc{\K}{{\mathcal K}}
\renewcommand{\H}{{\mathcal H}}

\nc{\T}{{\mathcal T}}
\nc{\V}{{\mathcal V}}
\renc{\P}{{\mathcal P}}
\nc{\A}{{\mathcal A}}
\nc{\B}{{\mathcal B}}
\nc{\U}{{\mathcal U}}

\nc{\Gr}{{\on{Gr}}}

\nc{\frn}{{\check{\mathfrak u}(P)}}

\nc{\fC}{\mathfrak C}
\nc{\p}{\mathfrak p}
\nc{\q}{\mathfrak q}
\nc\f{{\mathfrak f}}

\nc{\qo}{{\mathfrak q}}
\nc{\po}{{\mathfrak p}}
\nc{\s}{{\mathfrak s}}
\nc\w{\text{w}}

\renewcommand{\mod}{{\on{-mod}}}

\nc\mathi\iota
\nc\Spec{\on{Spec}}
\nc\Mod{\on{Mod}}
\nc{\tw}{\widetilde{\mathfrak t}}
\nc{\pw}{\widetilde{\mathfrak p}}
\nc{\qw}{\widetilde{\mathfrak q}}
\nc{\jw}{\widetilde j}

\nc{\grb}{\overline{\Gr}}
\nc{\I}{\mathcal I}

\nc{\lambdach}{{\check\lambda}}
\nc{\Lambdach}{{\check\Lambda}{}}
\nc{\much}{{\check\mu}}
\nc{\omegach}{{\check\omega}}
\nc{\nuch}{{\check\nu}}
\nc{\etach}{{\check\eta}}
\nc{\alphach}{{\check\alpha}}
\nc{\oblvtach}{{\check\oblvta}}
\nc{\rhoch}{{\check\rho}}
\nc{\ch}{{\check h}}

\nc{\Hb}{\overline{\H}}

\nc{\BA}{{\mathbb{A}}}
\nc{\BC}{{\mathbb{C}}}
\nc{\BE}{{\mathbb{E}}}
\nc{\BF}{{\mathbb{F}}}
\nc{\BG}{{\mathbb{G}}}
\nc{\BM}{{\mathbb{M}}}
\nc{\BO}{{\mathbb{O}}}
\nc{\BD}{{\mathbb{D}}}
\nc{\BL}{{\mathbb{L}}}
\nc{\Bl}{{\mathbb{l}}}
\nc{\BN}{{\mathbb{N}}}
\nc{\BP}{{\mathbb{P}}}
\nc{\BQ}{{\mathbb{Q}}}
\nc{\BR}{{\mathbb{R}}}
\nc{\BZ}{{\mathbb{Z}}}
\nc{\BS}{{\mathbb{S}}}
\nc{\BT}{{\mathbb{T}}}

\nc{\CA}{{\mathcal{A}}}
\nc{\CB}{{\mathcal{B}}}

\nc{\CE}{{\mathcal{E}}}
\nc{\CF}{{\mathcal{F}}}
\nc{\CH}{{\mathcal{H}}}

\nc{\CL}{{\mathcal{L}}}
\nc{\CC}{{\mathcal{C}}}
\nc{\CG}{{\mathcal{G}}}
\nc{\CM}{{\mathcal{M}}}
\nc{\CN}{{\mathcal{N}}}
\nc{\CK}{{\mathcal{K}}}
\nc{\CO}{{\mathcal{O}}}
\nc{\CP}{{\mathcal{P}}}
\nc{\CQ}{{\mathcal{Q}}}
\nc{\CR}{{\mathcal{R}}}
\nc{\CS}{{\mathcal{S}}}
\nc{\CT}{{\mathcal{T}}}
\nc{\CU}{{\mathcal{U}}}
\nc{\CV}{{\mathcal{V}}}
\nc{\CW}{{\mathcal{W}}}
\nc{\CX}{{\mathcal{X}}}
\nc{\CY}{{\mathcal{Y}}}
\nc{\CZ}{{\mathcal{Z}}}
\nc{\CI}{{\mathcal{I}}}
\nc{\cD}{{\mathcal{D}}}
\nc{\ocD}{\overset{\circ}{\mathcal{D}}}

\nc{\csM}{{\check{\mathcal A}}{}}
\nc{\oM}{{\overset{\circ}{\mathcal M}}{}}
\nc{\obM}{{\overset{\circ}{\mathbf M}}{}}
\nc{\oCA}{{\overset{\circ}{\mathcal A}}{}}
\nc{\obA}{{\overset{\circ}{\mathbf A}}{}}
\nc{\ooM}{{\overset{\circ}{M}}{}}
\nc{\osM}{{\overset{\circ}{\mathsf M}}{}}
\nc{\vM}{{\overset{\bullet}{\mathcal M}}{}}
\nc{\nM}{{\underset{\bullet}{\mathcal M}}{}}
\nc{\oD}{{\overset{\circ}{\mathcal D}}{}}
\nc{\obD}{{\overset{\circ}{\mathbf D}}{}}
\nc{\oA}{{\overset{\circ}{\mathbb A}}{}}
\nc{\op}{{\overset{\bullet}{\mathbf p}}{}}
\nc{\cp}{{\overset{\circ}{\mathbf p}}{}}
\nc{\oU}{{\overset{\bullet}{\mathcal U}}{}}
\nc{\oZ}{{\overset{\circ}{\mathcal Z}}{}}
\nc{\ofZ}{{\overset{\circ}{\mathfrak Z}}{}}
\nc{\oF}{{\overset{\circ}{\fF}}}

\nc{\fa}{{\mathfrak{a}}}
\nc{\fb}{{\mathfrak{b}}}
\nc{\fc}{{\mathfrak{c}}}
\nc{\fch}{{\mathfrak{ch}}}
\nc{\fd}{{\mathfrak{d}}}
\nc{\ff}{{\mathfrak{f}}}
\nc{\fg}{{\mathfrak{g}}}
\nc{\fgl}{{\mathfrak{gl}}}
\nc{\fh}{{\mathfrak{h}}}
\nc{\fj}{{\mathfrak{j}}}
\nc{\fl}{{\mathfrak{l}}}
\nc{\fm}{{\mathfrak{m}}}
\nc{\fn}{{\mathfrak{n}}}
\nc{\fu}{{\mathfrak{u}}}
\nc{\fp}{{\mathfrak{p}}}
\nc{\fr}{{\mathfrak{r}}}
\nc{\fs}{{\mathfrak{s}}}
\nc{\ft}{{\mathfrak{t}}}
\nc{\fT}{{\mathfrak{T}}}
\nc{\fz}{{\mathfrak{z}}}
\nc{\fsl}{{\mathfrak{sl}}}
\nc{\hsl}{{\widehat{\mathfrak{sl}}}}
\nc{\hgl}{{\widehat{\mathfrak{gl}}}}
\nc{\hg}{{\widehat{\mathfrak{g}}}}
\nc{\htt}{{\widehat{\mathfrak{t}}}}
\nc{\chg}{{\widehat{\mathfrak{g}}}{}^\vee}
\nc{\hn}{{\widehat{\mathfrak{n}}}}
\nc{\chn}{{\widehat{\mathfrak{n}}}{}^\vee}

\nc{\fA}{{\mathfrak{A}}}
\nc{\fB}{{\mathfrak{B}}}
\nc{\fD}{{\mathfrak{D}}}
\nc{\fE}{{\mathfrak{E}}}
\nc{\fF}{{\mathfrak{F}}}
\nc{\fG}{{\mathfrak{G}}}
\nc{\fK}{{\mathfrak{K}}}
\nc{\fL}{{\mathfrak{L}}}
\nc{\fM}{{\mathfrak{M}}}
\nc{\fN}{{\mathfrak{N}}}
\nc{\fP}{{\mathfrak{P}}}
\nc{\fU}{{\mathfrak{U}}}
\nc{\fV}{{\mathfrak{V}}}
\nc{\fZ}{{\mathfrak{Z}}}

\nc{\bb}{{\mathbf{b}}}
\nc{\bc}{{\mathbf{c}}}
\nc{\bd}{{\mathbf{d}}}
\nc{\bbf}{{\mathbf{f}}}
\nc{\be}{{\mathbf{e}}}
\nc{\bg}{{\mathbf{g}}}
\nc{\bi}{{\mathbf{i}}}
\nc{\bj}{{\mathbf{j}}}
\nc{\bn}{{\mathbf{n}}}
\nc{\bp}{{\mathbf{p}}}
\nc{\bq}{{\mathbf{q}}}
\nc{\bu}{{\mathbf{u}}}
\nc{\bv}{{\mathbf{v}}}
\nc{\bx}{{\mathbf{x}}}
\nc{\bs}{{\mathbf{s}}}
\nc{\by}{{\mathbf{y}}}
\nc{\bw}{{\mathbf{w}}}
\nc{\bA}{{\mathbf{A}}}
\nc{\bK}{{\mathbf{K}}}
\nc{\bB}{{\mathbf{B}}}
\nc{\bC}{{\mathbf{C}}}
\nc{\bG}{{\mathbf{G}}}
\nc{\bD}{{\mathbf{D}}}
\nc{\bH}{{\mathbf{He}}}
\nc{\bM}{{\mathbf{M}}}
\nc{\bN}{{\mathbf{N}}}
\nc{\bO}{{\mathbf{O}}}
\nc{\bV}{{\mathbf{V}}}
\nc{\bW}{{\mathbf{Wh}}}
\nc{\bX}{{\mathbf{X}}}
\nc{\bY}{{\mathbf{Y}}}
\nc{\bZ}{{\mathbf{Z}}}
\nc{\bS}{{\mathbf{S}}}
\nc{\bT}{{\mathbf{T}}}

\nc{\sA}{{\mathsf{A}}}
\nc{\sB}{{\mathsf{B}}}
\nc{\sC}{{\mathsf{C}}}
\nc{\sD}{{\mathsf{D}}}
\nc{\sF}{{\mathsf{F}}}
\nc{\sG}{{\mathsf{G}}}
\nc{\sH}{{\mathsf{H}}}
\nc{\sK}{{\mathsf{K}}}
\nc{\sL}{{\mathsf{L}}}
\nc{\sM}{{\mathsf{M}}}
\nc{\sO}{{\mathsf{O}}}
\nc{\sR}{{\mathsf{R}}}
\nc{\sU}{{\mathsf{U}}}
\nc{\sW}{{\mathsf{W}}}
\nc{\sQ}{{\mathsf{Q}}}
\nc{\sP}{{\mathsf{P}}}
\nc{\sY}{{\mathsf{Y}}}
\nc{\sZ}{{\mathsf{Z}}}
\nc{\sfp}{{\mathsf{p}}}
\nc{\sfq}{{\mathsf{q}}}
\nc{\sr}{{\mathsf{r}}}
\nc{\sk}{{\mathsf{k}}}
\nc{\su}{{\mathsf{u}}}
\nc{\sv}{{\mathsf{v}}}
\nc{\sg}{{\mathsf{g}}}
\nc{\sff}{{\mathsf{f}}}
\nc{\sfb}{{\mathsf{b}}}
\nc{\sfc}{{\mathsf{c}}}
\nc{\sd}{{\mathsf{d}}}

\nc{\BK}{{\bar{K}}}

\nc{\tA}{{\widetilde{\mathbf{A}}}}
\nc{\tB}{{\widetilde{\mathcal{B}}}}
\nc{\tg}{{\widetilde{\mathfrak{g}}}}
\nc{\tG}{{\widetilde{G}}}
\nc{\TM}{{\widetilde{\mathbb{M}}}{}}
\nc{\tO}{{\widetilde{\mathsf{O}}}{}}
\nc{\tU}{{\widetilde{\mathfrak{U}}}{}}
\nc{\TZ}{{\tilde{Z}}}
\nc{\tx}{{\tilde{x}}}
\nc{\tbv}{{\tilde{\bv}}}
\nc{\tfP}{{\widetilde{\mathfrak{P}}}{}}
\nc{\tz}{{\tilde{\zeta}}}
\nc{\tmu}{{\tilde{\mu}}}

\nc{\urho}{\underline{\rho}}
\nc{\uB}{\underline{B}}
\nc{\uC}{{\underline{\mathbb{C}}}}
\nc{\ui}{\underline{i}}
\nc{\uj}{\underline{j}}
\nc{\ofP}{{\overline{\mathfrak{P}}}}
\nc{\oB}{{\overline{\mathcal{B}}}}
\nc{\og}{{\overline{\mathfrak{g}}}}
\nc{\oI}{{\overline{I}}}

\nc{\eps}{\varepsilon}
\nc{\hrho}{{\hat{\rho}}}

\nc{\one}{{\mathbf{1}}}
\nc{\two}{{\mathbf{t}}}

\nc{\Rep}{{\mathop{\operatorname{\rm Rep}}}}
\nc{\Tot}{{\mathop{\operatorname{\rm Tot}}}}
\nc{\Ker}{{\mathop{\operatorname{\rm Ker}}}}
\nc{\Hilb}{{\mathop{\operatorname{\rm Hilb}}}}
\nc{\End}{{\mathop{\operatorname{\rm End}}}}
\nc{\Ext}{{\mathop{\operatorname{\rm Ext}}}}
\nc{\CHom}{{\mathop{\operatorname{{\mathcal{H}}\it om}}}}
\nc{\GL}{{\mathop{\operatorname{\rm GL}}}}
\nc{\gr}{{\mathop{\operatorname{\rm gr}}}}
\nc{\Id}{{\mathop{\operatorname{\rm Id}}}}
\nc{\de}{{\mathop{\operatorname{\rm def}}}}
\nc{\length}{{\mathop{\operatorname{\rm length}}}}
\nc{\supp}{{\mathop{\operatorname{\rm supp}}}}

\nc{\Cliff}{{\mathsf{Cliff}}}
\nc{\Fl}{\on{Fl}}
\nc{\Fib}{{\mathsf{Fib}}}
\nc{\Coh}{{\on{Coh}}}
\nc{\QCoh}{{\on{QCoh}}}
\nc{\IndCoh}{{\on{IndCoh}}}
\nc{\FCoh}{{\mathsf{FCoh}}}

\nc{\reg}{{\text{\rm reg}}}

\nc{\cplus}{{\mathbf{C}_+}}
\nc{\cminus}{{\mathbf{C}_-}}
\nc{\cthree}{{\mathbf{C}_*}}
\nc{\Qbar}{{\bar{Q}}}
\nc\Eis{{\on{Eis}}}
\nc\Eisb{\ol\Eis{}}
\nc\Eisr{\on{Eis}^{rat}{}}
\nc\wh{\widehat}
\nc{\Def}{\on{Def_{\check{\fb}}(E)}}
\nc{\barZ}{\overline{Z}{}}
\nc{\barbarZ}{\overline{\barZ}{}}
\nc{\barpi}{\overline\pi}
\nc{\barbarpi}{\overline\barpi}
\nc{\barpip}{\overline\pi{}^+}
\nc{\barpim}{\overline\pi{}^-}

\nc{\fq}{\mathfrak q}

\nc{\fqb}{\ol{\fq}{}}
\nc{\fpb}{\ol{\fp}{}}
\nc{\fpr}{{\fp^{rat}}{}}
\nc{\fqr}{{\fq^{rat}}{}}

\nc{\hattimes}{\wh\otimes}

\nc{\bh}{{{\mathbf h}}}
\nc{\bk}{{{\mathbf k}}}
\nc{\bOmega}{{\overline{\Omega(\check \fn)}}}

\nc{\seq}[1]{\stackrel{#1}{\sim}}

\nc{\cT}{{\check{T}}}
\nc{\cG}{{\check{G}}}
\nc{\cM}{{\check{M}}}
\nc{\cB}{{\check{B}}}
\nc{\cP}{{\check{P}}}

\nc{\ct}{{\check{\mathfrak t}}}
\nc{\cg}{{\check{\fg}}}
\nc{\cb}{{\check{\fb}}}
\nc{\cn}{{\check{\fn}}}

\nc{\cLambda}{{\check\Lambda}}

\nc{\cla}{{\check\lambda}}
\nc{\cmu}{{\check\mu}}
\nc{\cnu}{{\check\nu}}
\nc{\ceta}{{\check\eta}}

\nc{\DefbE}{{\on{Def}_{\cB}(E_\cT)}}

\nc{\imathb}{{\ol{\imath}}}
\nc{\rlr}{\overset{\longrightarrow}{\underset{\longrightarrow}\longleftarrow}}

\nc{\oBun}{\overset{\circ}\Bun}
\nc{\oSht}{\overset{\circ}\Sht}
\nc{\LocSys}{\on{LocSys}}
\nc{\BunBbb}{\ol{\ol{Bun}}_B}
\nc{\BunBr}{\Bun_B^{rat}}
\nc{\BunBrp}{\Bun_B^{rat,polar}}
\nc{\BunTrp}{\Bun_T^{rat,polar}}
\nc{\BunNr}{\Bun_N^{rat}}
\nc{\BunNre}{\Bun_N^{enh,rat}}
\nc{\BunTr}{\Bun_T^{rat}}
\nc{\Vect}{\on{Vect}}
\nc{\Whit}{\on{Whit}}
\nc{\CTb}{\ol{\on{CT}}}
\nc{\Ran}{\on{Ran}}
\nc{\CTr}{\on{CT}^{rat}{}}
\nc\jmathr{\jmath^{rat}{}}
\nc{\ux}{\underline{x}}
\nc{\clambda}{{\check\lambda}}
\nc{\calpha}{{\check\alpha}}
\nc{\ind}{{\mathbf{ind}}}
\nc{\oblv}{{\mathbf{oblv}}}
\nc{\coeff}{\on{W-coeff}}
\nc{\Poinc}{\on{Poinc}}
\nc{\Dmod}{\on{D-mod}}
\nc{\dr}{\on{dR}}
\nc{\oCZ}{\overset{\circ}\CZ}
\nc{\KL}{\on{KL}}
\nc{\triv}{{\mathbf{triv}}}

\nc{\dgSch}{\on{DGSch}}
\nc{\Sch}{\on{Sch}}
\nc{\affdgSch}{\on{DGSch}^{\on{aff}}}
\nc{\affSch}{\on{Sch}^{\on{aff}}}
\nc{\Sing}{\on{Sing}}
\nc{\inftygroup}{\infty\on{-Grpd}}
\renc{\dr}{{\on{dr}}}
\nc\Maps{\on{Maps}}
\nc\Res{\on{Res}}
\nc\bMaps{\mathbf{Maps}}
\nc{\ul}{\underline}
\nc{\bNP}{\mathbf{N(P)}}
\nc{\ofc}{\overset{\circ}\fch}
\nc{\ppart}{(\!(t)\!)}
\nc{\qqart}{[\![t]\!]}
\nc{\crit}{\on{crit}}
\nc{\DGCat}{\on{DGCat}}
\nc{\Shv}{\on{Shv}}
\nc{\bDelta}{\mathbf{\Delta}}
\nc{\genB}{{\overset{\on{gen}}\to B}}
\nc{\genP}{{\underset{\on{gen}}\longrightarrow P}}
\nc{\genN}{{\underset{\on{gen}}\longrightarrow N}}
\nc{\semiinf}{{\frac{\infty}{2}+\bullet}}
\nc{\mmod}{\on{-}\mathbf{mod}}
\nc{\commod}{\on{-}\mathbf{comod}}
\nc{\AdFr}{\on{Ad}_{\on{Frob}}}
\nc{\Frob}{\on{Frob}}
\nc{\Tr}{\on{Tr}}
\nc{\Sht}{\on{Sht}}
\nc{\sfe}{\mathsf{e}}
\nc{\tCat}{{2\on{-Cat}}}
\nc{\tIndCoh}{{2\on{-IndCoh}}}
\nc{\tQCoh}{{2\on{-QCoh}}}
\nc{\uShv}{\underline{\Shv}}
\nc{\qLisse}{\on{QLisse}}
\nc{\Nilp}{{\on{Nilp}}}
\nc{\sotimes}{\overset{!}\otimes} 
\nc{\AGCat}{\on{AGCat}}
\nc{\LS}{\on{LS}}
\nc{\Cat}{\on{-Cat}}
\nc{\bo}{\mathbf o}
\nc{\poinc}{\mathsf{poinc}}
\nc{\funct}{\mathsf{funct}}
\nc{\sFunct}{\mathsf{Funct}}
\nc{\BW}{\mathbb{W}}
\nc{\uHom}{\underline{\on{Hom}}}
\nc{\BH}{\mathbb{H}}

\newcommand{\WD}{\textrm{Weil-Drinf}}
\newcommand{\coarse}{\mathrm{coarse}}
\newcommand{\pss}{\textrm{Pro-ss}^\coarse}
\newcommand{\totext}[1]{\xrightarrow{#1}}

\begin{document}

\vskip1cm

\title[On the excursion algebra]{On the excursion algebra}

\author{Dennis Gaitsgory, Kevin Lin and Wyatt Reeves}  

\begin{abstract}
The \emph{excursion algebra} associated to a scheme $X$ over $\BF_q$ and a reductive group $\sG$ is the algebra of global functions
on the stack of arithmetic $\sG$-local systems on $X$. When $X$ is a curve, this algebra acts on the space of automorphic functions.
In this paper we establish some basic properties of this algebra.
\end{abstract}

\date{\today}

\maketitle

\tableofcontents

\section*{Introduction}

In his seminal work \cite{VLaf} (extended in \cite{Xue}),  V.~Lafforgue defined an action\footnote{See \secref{sss:exc action} for the precise assertion.} 
of the \emph{excursion algebra} on the space of automorphic functions. In this paper we study the excursion algebra abstractly, without connection
to automorphic functions (and we do it for an arbitrary smooth algebraic variety $X$ over $\BF_q$, not necessarily a curve); we denote it by $\on{Exc}(X,\sG)$, where
$\sG$ is the target reductive group. 

\medskip

Our main results include:

\begin{enumerate}

\item Each factor of $\on{Exc}(X,\sG)$ is reduced and normal; 

\medskip

\item In fact, $\on{Exc}(X,\sG)$ is a product of algebras of the form $\CO_{\sG'/\!/\on{Ad}_\sg(\sG')}$, where:

\begin{itemize}

\item $\sG'$ is a reductive subgroup of $\sG$;

\item $\sg\in \sG$ is an element that normalizes $\sG'$;

\item $\on{Ad}_\sg$ stands for the action by $\sg$-twisted conjugation.

\end{itemize}

\medskip

\item $\on{Exc}(X,\sG)$ is finitely generated over each local Hecke algebra $\CH_x(\sG)$ (see \thmref{t:Hecke}(a) for the precise assertion);

\medskip

\item For $\sG=GL_n$, the map from the global Hecke algebra $\CH_X(\sG)$ to $\on{Exc}(X,\sG)$ is surjective (see \thmref{t:Hecke}(b) for the precise assertion);

\medskip

\item The algebra $\on{Exc}(X,\sG)$ is independent of $\ell$ in the sense that there exists a canonically defined $\BQ$-algebra
$\on{Exc}(X,\sG)_\BQ$ such that for any $\ell\neq p$, we have a canonical isomorphism $$\ol\BQ_\ell\underset{\BQ}\otimes \on{Exc}(X,\sG)_\BQ\simeq
\on{Exc}(X,\sG),$$ and the map $\CH_X(\sG)\to \on{Exc}(X,\sG)$ is induced by a map  $\CH_{X,\BQ}(\sG)\to \on{Exc}(X,\sG)_\BQ$;

\medskip

\item In particular, for $\sG=GL_n$, the rational structure on $\on{Exc}(X,\sG)$ is uniquely characterized by the property of being compatible
with the one on $\CH_X(\sG)$. 

\end{enumerate} 

%

\ssec{The excursion algebra}

\sssec{}

Let $X_0$ be a (geometrically) connected scheme of finite type over $\BF_q$, and let $X$ denote its base change to $\ol\BF_q$. Let
$\sG$ be a reductive group over the field of coefficients $\sfe:=\ol\BQ_\ell$, $\ell\neq p:=\on{char}(\BF_q)$. 

\medskip

To this data, following \cite{AGKRRV}, we associate two objects in derived algebraic geometry over $\sfe$. One is
$\LS_\sG^{\on{restr}}(X)$, see \secref{sss:LS restr} for the definition. The prestack $\LS_\sG^{\on{restr}}(X)$ depends
functorially on the \'etale site of $X$; hence the geometric Frobenius acting on $X$ gives rise to an automorphism,
denoted, $\Frob$, of $\LS_\sG^{\on{restr}}(X)$. 

\medskip

We set
$$\LS_\sG^{\on{arithm}}(X):=(\LS_\sG^{\on{restr}}(X))^{\Frob}.$$

\medskip

One shows that $\LS_\sG^{\on{arithm}}(X)$ is a union of connected components, each of which is a quasi-compact
algebraic stack over $\sfe$ locally of finite type (in fact, the stack-quotient of an affine scheme by an action of $\sG$).
Moreover, if one restricts ramification along the divisor at infinity with respect to some compactification (which we will from
now on assume), $\LS_\sG^{\on{arithm}}(X)$ has finitely many connected components. 

\medskip

The excursion algebra $\on{Exc}(X,\sG)$ is defined as the algebra of global functions on $\LS_\sG^{\on{arithm}}(X)$, i.e., 
$$\on{Exc}(X,\sG):=\Gamma(\LS_\sG^{\on{arithm}}(X),\CO_{\LS_\sG^{\on{arithm}}(X)}).$$

A priori, $\on{Exc}(X,\sG)$ is a connective commutative DG algebra over $\sfe$. However, it follows from the
results of this paper that it is actually \emph{classical}, i.e., is concentrated in cohomological degree $0$. 

\medskip

The goal of this paper is to establish the properties of $\on{Exc}(X,\sG)$ listed in the preamble. 

\sssec{}

Here is some history of the definition of $\on{Exc}(X,\sG)$. For the rest of this subsection we will assume that $X_0$ is a curve; let $\ol{X}_0$
denote its compactification\footnote{We should say that the fact that we consider arbitrary schemes rather than curves 
is out of an academic interest--our main applications are in the case of curves.}. 

\medskip

The stack $\LS_\sG^{\on{arithm}}(X)$ should be thought of as the moduli space of \emph{continuous} representations 
of the Weil group $\on{Weil}(X)$ of $X$ into $\sG(\sfe)$ up to conjugation. 

\medskip

Let us temporarily replace $\on{Weil}(X)$ by an abstract group $\Gamma$, and consider the stack
$$\bMaps_{\on{groups}}(\Gamma,\sG)/\on{Ad}(\sG),$$
and set
$$\on{Exc}(\Gamma,\sG):=\Gamma(\bMaps_{\on{groups}}(\Gamma,\sG)/\on{Ad}(\sG),\CO_{\bMaps_{\on{groups}}(\Gamma,\sG)/\on{Ad}(\sG)}).$$

The algebra $\on{Exc}(\Gamma,\sG)$ was described explicitly by V.~Lafforgue in \cite[Sect. 11]{VLaf}. He exhibited it as a span 
of a set of specific elements that he called ``excursion operators", hence the name of the algebra (we recall the definition of these
elements in \secref{sss:exc op}). 

\medskip

The natural projection
$$\bMaps_{\on{groups}}(\Gamma,\sG)/\on{Ad}(\sG)\to \Spec(\on{Exc}(\Gamma,\sG))$$
gives rise to a bijection between:

\begin{itemize}

\item Conjugacy classes of homomorphisms $\Gamma\to \sG$, \emph{up to semi-simplification}, and 

\item $\sfe$-points of $\Spec(\on{Exc}(\Gamma,\sG))$.

\end{itemize}

In particular, the set of $\sfe$-points of $\Spec(\on{Exc}(\Gamma,\sG))$ is in bijection with the set of conjugacy classes 
of \emph{semi-simple} homomorphisms $\Gamma\to \sG$. We refer the reader to \cite[Sect 11]{VLaf} for the discussion
of the relation between $\on{Exc}(\Gamma,\sG)$ and \emph{quasi-characters}. 

\sssec{}

Let us take $\Gamma$ to be $\on{Weil}(X)$, viewed as an abstract group (i.e., we disregard its topology). Then the main construction 
in \cite{VLaf} (extended by C.~Xue in \cite{Xue}) defines an action of
$$\on{Exc}(X,\sG)^{\on{discr}}:=\on{Exc}(\on{Weil}(X),\sG)$$
on the space of automorphic functions with compact supports for the group $G:=\sG^\vee$ and the global field corresponding to $\ol{X}_0$
(with bounded ramification at points of $\ol{X}_0-X_0$). 

\medskip

We have a naturally defined closed embedding
$$^{\on{cl}}\!\LS_\sG^{\on{arithm}}(X)\to \bMaps_{\on{groups}}(\on{Weil}(X),\sG)/\on{Ad}(\sG),$$
which gives rise to a map 
$$\on{Exc}(X,\sG)^{\on{discr}}\to \on{Exc}(X,\sG),$$
which is surjective on $H^0(-)$.

\medskip

In fact, one can describe the corresponding quotient explicitly: the corresponding closed subscheme
$$\Spec(\on{Exc}(X,\sG))\subset \Spec(\on{Exc}(X,\sG)^{\on{discr}})$$
is the Zariski closure of the subset of $\sfe$-points that correspond to \emph{continuous}
semi-simple homomorphisms $\on{Weil}(X)\to \sG$, see \propref{p:coarse via discr} in this paper. 

\sssec{} \label{sss:exc action} 

Combined with the technology developed in \cite{AGKRRV}, one can show that the action of $\on{Exc}(X,\sG)^{\on{discr}}$
on automorphic functions actually factors via $\on{Exc}(X,\sG)$. 

\sssec{}

The primary impetus for this paper was to establish properties (1) and (2) mentioned in the preamble: 

\medskip

Property (1) is a key input for the project of proving a version of the Ramanujan-Petersson conjecture over function fields
(unramified version), see \cite{Ras}. 

\medskip

Property (2) serves as an input for the even more ambitious project of proving the Arthur conjectures, see also \cite{Ras}.  

\sssec{}

Properties (3) and (4) reflect some well-known phenomena on the automorphic side. 

\medskip

Properties (5) and (6) can be regarded as a tiny step towards the program of understanding the phenomenon of
``independence of $\ell$" for the Langlands correspondence over function fields. 

\ssec{The geometrically semi-simple locus}

The proofs of all the main results of this paper are based on an alternative description of $\on{Exc}(X,\sG)$ that 
we outline in this subsection.

\sssec{}

The prestack $\LS_\sG^{\on{restr}}(X)$ is a union of connected components $\LS_\sG^{\on{restr}}(X)_\alpha$;
two $\sG$-local systems belong to the same connected component if and only if they have isomorphic
semi-simplifications; each $\LS_\sG^{\on{restr}}(X)_\alpha$ contains a unique semi-simple local system $\sigma_\alpha$ 
(defined up to an isomorphism). 

\medskip

Denote
$$\LS_\sG^{\on{restr},0}(X):=\underset{\alpha}\sqcup \on{pt}/\sG_\alpha, \quad \sG_\alpha:=\on{Aut}(\sigma_\alpha);$$
this is the \emph{semi-simple} locus of $\LS_\sG^{\on{restr}}(X)$.

\sssec{}

Set
$$\LS_\sG^{\on{arithm},0}(X):=(\LS_\sG^{\on{restr},0}(X))^{\Frob}.$$

The main point of this paper is that the restriction map
\begin{equation} \label{e:retsriction Intro}
\on{Exc}(X,\sG):=\Gamma(\LS_\sG^{\on{arithm}}(X),\CO_{\LS_\sG^{\on{arithm}}(X)})\to
\Gamma(\LS_\sG^{\on{arithm},0}(X),\CO_{\LS_\sG^{\on{arithm},0}(X)})=:\on{Exc}(X,\sG)^0
\end{equation}
is an isomorphism. This is our \thmref{t:exc}.

\sssec{}

The results mentioned in the preamble all follow from the isomorphism \eqref{e:retsriction Intro}. 

\medskip

For example, points (1) and (2) follows from the fact for every $\alpha$ for which $\LS_\sG^{\on{restr}}(X)_\alpha$ is $\Frob$-invariant, we have
$$(\LS_\sG^{\on{restr}}(X)_\alpha)^{\Frob}\simeq \sG_\alpha/\on{Ad}_{\Frob}(\sG_\alpha),$$
as a stack.

\medskip

Point (3) follows from the fact that the map
$$\sG_\alpha/\!/\on{Ad}_{\Frob}(\sG_\alpha)\to \sG/\!/\on{Ad}(\sG),$$
induced by the inclusion $\sG_\alpha\to \sG$, is finite, which is a variant of a theorem from \cite{Vin}. 

\medskip

Point (5) is obtained by rewriting $\on{Exc}(X,\sG)^0$ via Drinfeld's rational model for the semi-simple envelope of the
Galois group of $X_0$, see \secref{ss:Drinf}.  

\sssec{}

Let us now explain the mechanism behind the isomorphism \eqref{e:retsriction Intro}. 

\medskip

Recall that the input into the definition of
$\LS_\sG^{\on{restr}}(X)$ is the symmetric monoidal category $\qLisse(X)$ of local systems on $X$. Let $\qLisse^{\on{relev}}(X)\subset \qLisse(X)$
be the full subcategory generated by local systems that admit a Weil structure. 

\medskip

Denote by $\LS_\sG^{\on{restr,relev}}(X)$ the corresponding substack of $\LS_\sG^{\on{restr}}(X)$
(in fact, $\LS_\sG^{\on{restr,relev}}(X)$ is the union of certain of the connected components of $\LS_\sG^{\on{restr}}(X)$). 

\medskip

It is easy to see that the inclusion
$$(\LS_\sG^{\on{restr,relev}}(X))^{\Frob}\to (\LS_\sG^{\on{restr}}(X))^{\Frob}=:\LS_\sG^{\on{arithm}}(X)$$
is an isomorphism.

\sssec{}

Let $\BZ^{\on{alg,wt}}$ be the pro-algebraic group, whose representations are finite-dimensional vector spaces
equipped with an automorphism, whose eigenvalues are $q$-Weil numbers.

\medskip

By definition, we have a homomorphism
$$\BZ\to \BZ^{\on{alg,wt}}$$
with Zariski-dense image. Denote by $\Frob\in \BZ^{\on{alg,wt}}$ the image of $1\in \BZ$. 

\medskip

In addition, we have a natural homomorphism $\BG_m\to \BZ^{\on{alg,wt}}$, which records the decomposition into weights. 

\sssec{} \label{sss:pre contr Intro}

The isomorphism \eqref{e:retsriction Intro} is based on the following two observations:

\begin{itemize}

\item The prestack $\LS_\sG^{\on{restr,relev}}(X)$ carries a natural action of $\BZ^{\on{alg,wt}}$, with $\Frob\in \BZ^{\on{alg,wt}}$
acting as the automorphism induced by the geometric Frobenius on $X$; 

\medskip

\item The resulting action of $\BG_m\subset \BZ^{\on{alg,wt}}$ on $\LS_\sG^{\on{restr,relev}}(X)$ extends to an action of the monoid $\BA^1$ (that still commutes with
$\BZ^{\on{alg,wt}}$), which contracts it to
$$\LS_\sG^{\on{restr,relev,0}}(X):=\LS_\sG^{\on{restr,relev}}(X)\cap \LS_\sG^{\on{restr,0}}(X).$$

\end{itemize}

\sssec{}

Let us show how these observations imply the isomorphism \eqref{e:retsriction Intro}. 

\medskip

First, the action of $\BZ^{\on{alg,wt}}$ on 
$\LS_\sG^{\on{restr,relev}}(X)$ induces one on 
$$(\LS_\sG^{\on{restr,relev}}(X))^{\Frob}\simeq \LS_\sG^{\on{arithm}}(X).$$

\medskip

Second, the $\BA^1$-action on $\LS_\sG^{\on{restr,relev}}(X)$ induces one on $\LS_\sG^{\on{arithm}}(X)$, which contracts it to
$\LS_\sG^{\on{arithm},0}(X)$. 

\medskip

In particular, we obtain that $\on{Exc}(X,\sG)$ carries the following structures:

\medskip

\begin{itemize}

\item It is a representation of $\BZ^{\on{alg,wt}}$;

\medskip

\item The induced representation of $\BG_m$ extends to a representation of the monoid $\BA^1$, and the latter contracts it to $\on{Exc}(X,\sG)^0$;

\end{itemize} 

\medskip

Note, however, that by the definition of $\LS_\sG^{\on{arithm}}(X)$, the action of $\Frob$ on it, and hence on $\on{Exc}(X,\sG)$, is trivial. 
Since $\Frob$ generates $\BZ^{\on{alg,wt}}$, we obtain that the action of all of $\BZ^{\on{alg,wt}}$ on $\on{Exc}(X,\sG)$ is actually trivial. 

\medskip

In particular, the action of $\BG_m$ on 
$\on{Exc}(X,\sG)$ is trivial. Hence, the contraction 
$$\on{Exc}(X,\sG)\to \on{Exc}(X,\sG)^0$$
is trivial, i.e., is an isomorphism. 

\medskip

We refer the reader to the proof of \thmref{t:exc}, where the above argument is spelled out in more detail.

\sssec{}

The idea of a contraction as in \secref{sss:pre contr Intro} traces back to the PhD thesis of the third author 
of this paper \cite{Re}, and a subsequent paper by the second and the third authors \cite{LR}. 

\medskip

In this next subsection we will explain where it comes from. 

\ssec{The contraction}

\sssec{}

The above structures on the prestack $\LS_\sG^{\on{restr,relev}}(X)$ are induced by the functoriality of the construction
from \cite[Sect. 1.8]{AGKRRV} by the corresponding structures on the (symmetric monoidal) category $\qLisse(X)$.

\medskip

The latter are induced by the corresponding structures on the ambient category $\Shv^{\on{relev}}(X)$ (defined in \secref{ss:relev}),
which we will describe below. 

\sssec{}

Let $\Shv^{\on{Weil,wt}}(X)$ be the category of mixed Weil sheaves with integral weights, see \secref{sss:weil shv wt}. 

\medskip

We have a natural action of
$\Rep(\BZ^{\on{alg,wt}})$ on $\Shv^{\on{Weil,wt}}(X)$ (tensor up by sheaves pulled back from the point-scheme). 
We show that
the naturally defined functor
$$\Vect\underset{\Rep(\BZ^{\on{alg,wt}})}\otimes \Shv^{\on{Weil,wt}}(X)\to \Shv^{\on{relev}}(X)$$
is an equivalence, see \corref{c:W acts on rel}. 

\medskip

By the principle of ``equivariantization/de-equivariantization" (see \eqref{e:eq de-eq}), the latter equivalence implies
that we have a natural action of $\BZ^{\on{alg,wt}}$ on $\Shv^{\on{relev}}(X)$.

\sssec{} \label{sss:filtr Intro}

The extension of the action of $\BG_m$ on $\Shv^{\on{relev}}(X)$ to an action of the monoid $\BA^1$ follows from the next principle:

\medskip

Let $\bC$ be a DG category with an action of $\BG_m$. Denote $\bD:=\bC^{\BG_m}$; this is a category acted on by $\Rep(\BG_m)$,
or which is the same by the group $\BZ$. 

\medskip

In \thmref{t:A1} we show that the datum of extension of the given $\BG_m$-action to a $\BA^1$-action is equivalent to the datum of
a $\BZ$-invariant \emph{filtration} on $\bD$ (see \secref{sss:filtr} for what we mean by this). 

\sssec{}

In the case of $\bC:=\Shv^{\on{relev}}(X)$, the corresponding category $(\Shv^{\on{relev}}(X))^{\BG_m}$ is the category
introduced in \cite{HL}, and which we denote by $\Shv^{\on{relev,gr}}(X)$. It is defined as
$$\Rep(\BG_m)\underset{\Rep(\BZ^{\on{alg,wt}})}\otimes\Shv^{\on{Weil,wt}}(X),$$
or which is the same
$$\Vect_\sfe\underset{\Rep(\BZ^{\on{alg,0}})}\otimes \Shv^{\on{Weil,wt}}(X),$$
where
$$\Rep(\BZ^{\on{alg,0}})\subset \Rep(\BZ^{\on{alg,wt}})$$
is the full subcategory corresponding to eigenvalues that are Weil numbers of weight $0$.

\medskip

Using the latter presentation, we define a filtration on $\Shv^{\on{relev,gr}}(X)$ to be induced by the weight filtration on 
$\Shv^{\on{Weil,wt}}(X)$.

\sssec{}

The fact that $\BA^1$ contracts $\LS_\sG^{\on{restr,relev}}(X)$ to $\LS_\sG^{\on{restr,relev,0}}(X)$ stems from the following
phenomenon:

\medskip

In the situation of \secref{sss:filtr Intro} the action of $\BA^1$ contracts $\bC$ to the category $\on{gr}_0(\bD)$. In our case,
$$\on{gr}_0(\Shv^{\on{relev,gr}}(X))\simeq \Vect_\sfe\underset{\Rep(\BZ^{\on{alg,0}})}\otimes \Shv^{\on{Weil,0}}(X)=:\Shv^{\on{relev,0}}(X),$$
where $\Shv^{\on{Weil,0}}(X)$ is the category of Weil sheaves of weight $0$. 

\medskip

In \propref{p:relev 0} we show that $\Shv^{\on{relev,0}}(X)$ is the \emph{semi-simplification} of the category $\Shv^{\on{relev}}(X)$, 
and hence the corresponding category $\qLisse^{\on{relev,0}}(X)$ is the semi-simplification of $\qLisse^{\on{relev}}(X)$.  
The prestack
associated to the (symmetric monoidal) category $\qLisse^{\on{relev,0}}(X)$ by the construction of \cite[Sect. 1.8]{AGKRRV} is 
exactly $\LS_\sG^{\on{restr,relev,0}}(X)$.

\ssec{Structure of the paper}

\sssec{}

The goal of \secref{s:A1} is to prove \thmref{t:A1}, which describes what kind of structure on a category 
$\bC$ makes it into a comodule over $\QCoh(\BA^1)$. As was mentioned already, this structure turns out to
be a filtration on the corresponding category $\bD:=\bC^{\BG_m}$. 

\medskip

Note that this is vaguely analogous to the fact that the datum of a filtration on a vector space $V$ 
is equivalent to realizing $V$ as the fiber at $1$ of a quasi-coherent sheaf on $\BA^1/\BG_m$. 

\sssec{}

The goal of \secref{s:Weil} is to carry out the main construction of this paper, which is formulated as
\thmref{t:contraction on category}: it says that the category $\Shv^{\on{relev}}(X)$ admits a canonical
action of $\BA^1$ that contracts it onto $\Shv^{\on{relev},0}(X)$.

\medskip

Prior to doing so, we introduce the main actors: 

\medskip

\begin{enumerate}

\item The groups
$$\BZ\to \BZ^{\on{alg}}\twoheadrightarrow \BZ^{\on{alg,wt}}\twoheadrightarrow \BZ^{\on{alg},0};$$

\medskip

\item The various versions of the category of Weil sheaves:
$$\Shv^{\on{Weil},0}(X)\subset \Shv^{\on{Weil,wt}}(X)\subset \Shv^{\on{Weil,loc.fin}}(X)\subset \Shv^{\on{Weil}}(X)$$
(so that for $X=\on{pt}$ we recover the categories of representations of the groups in the previous
item);

\medskip

\item The categories 
$$\Shv^{\on{relev}}(X), \,\, \Shv^{\on{relev,gr}}(X) \text{ and } \Shv^{\on{relev},0}(X),$$

\medskip

\item The relation between the categories in items (2) and (3) is via the procedure of equivariantization and de-equivariantization
with respect to the groups in item (1).

\end{enumerate}

\sssec{}

In \secref{s:LS} we introduce the counterparts of the categories in items (2) and (3) above, where 
we impose the condition that the underlying object of $\Shv(X)$ be \emph{lisse}:
$$\qLisse^{\on{Weil},0}(X)\subset \qLisse^{\on{Weil,wt}}(X)\subset \qLisse^{\on{Weil,loc.fin}}(X)\subset \qLisse^{\on{Weil}}(X),$$
and
$$\qLisse^{\on{relev}}(X), \,\, \qLisse^{\on{relev,gr}}(X) \text{ and } \qLisse^{\on{relev},0}(X).$$

We introduce the corresponding Galois groups that control their hearts:

\medskip

\begin{itemize}

\item $\on{Gal}^{\on{alg}}(X) \leftrightarrow \qLisse(X)$;

\medskip

\item $\on{Weil}^{\on{alg}}(X) \leftrightarrow \qLisse^{\on{Weil}}(X)$; 

\medskip

\item $\on{Weil}^{\on{alg,loc.fin}}(X) \leftrightarrow \qLisse^{\on{Weil,loc.fin}}(X)$;

\medskip

\item $\on{Weil}^{\on{alg,wt}}(X) \leftrightarrow \qLisse^{\on{Weil,wt}}(X)$;

\medskip

\item $\on{Gal}^{\on{alg,relev}}(X) \leftrightarrow \qLisse^{\on{relev}}(X)$;

\medskip

\item $\on{Gal}^{\on{alg,relev},0}(X) \leftrightarrow \qLisse^{\on{relev},0}(X)$; we show that $\on{Gal}^{\on{alg,relev},0}(X)$ is the maximal reductive quotient of
$\on{Gal}^{\on{alg,relev}}(X)$;

\medskip

\item $\on{Gal}^{\on{alg,gr}}(X) \leftrightarrow \qLisse^{\on{relev,gr}}(X)$.

\end{itemize}

\medskip

We also introduce the corresponding versions of the moduli space of local systems:

\medskip

\begin{itemize}

\item $\qLisse(X)\rightsquigarrow \LS_\sG^{\on{restr}}(X)$;

\medskip

\item $\qLisse^{\on{relev}}(X)\rightsquigarrow \LS_\sG^{\on{restr,relev}}(X)$;

\medskip

\item $\qLisse^{\on{relev},0}(X)\rightsquigarrow \LS_\sG^{\on{restr,relev},0}(X)$;

\medskip

\item $\qLisse^{\on{Weil}} (X)\rightsquigarrow \LS_\sG^{\on{arithm}}(X):=(\LS_\sG^{\on{restr}}(X))^{\Frob}\simeq (\LS_\sG^{\on{restr,relev}}(X))^{\Frob}$;

\medskip

\item $\BZ\mod \underset{\Rep(\BZ^{\on{alg},0})}\otimes \qLisse^{\on{Weil},0}(X) \rightsquigarrow
\LS_\sG^{\on{arithm},0}(X):=(\LS_\sG^{\on{restr,relev},0}(X))^{\Frob}$; note that the Galois group responsible for this
category is what we denote $\on{Weil}^{\on{red}}(X)$, see \secref{sss:Weil red}. 

\end{itemize}

\bigskip

We apply Theorem-Construction \ref{t:contraction on category} to the full subcategory 
$\qLisse^{\on{relev}}(X)\subset \Shv^{\on{relev}}(X)$ and obtain a contraction of the prestack 
$\LS^{\on{restr,relev}}(X)$ onto $\LS^{\on{restr,relev},0}(X)$. 

\medskip 

Finally, we introduce the excursion algebra $\on{Exc}(X,\sG)$ as 
$$\Gamma(\LS_\sG^{\on{arithm}}(X),\CO_{\LS_\sG^{\on{arithm}}(X)}),$$
and prove the main result of this paper, \thmref{t:exc}, which says that the map 
$$\Gamma(\LS_\sG^{\on{arithm}}(X),\CO_{\LS_\sG^{\on{arithm}}(X)})\to
\Gamma(\LS_\sG^{\on{arithm},0}(X),\CO_{\LS_\sG^{\on{arithm},0}(X)})$$
is an isomorphism. 

\sssec{}

In \secref{s:fin} we connect our definition of the excursion algebra $\on{Exc}(X,\sG)$ to (a variant of)
V.~Lafforgue's definition, which uses the Weil group of $X$ as a discrete group.

\medskip

We then establish properties (3) and (4) stated in the preamble. 

\medskip

As was mentioned already, property (3) is deduced from \thmref{t:exc} using a variant of a theorem of Vinberg.

\medskip

Property (4) is deduced from Property (3) using Chebotarev's density. It is a variation (explained to us by V.~Lafforgue)
of the argument from \cite{Laum}, which shows that the map from the Grothendieck group of Weil sheaves on $X$ to 
the space of $\sfe$-valued functions on all $X(\BF_{q^n})$ is injective. 

\sssec{}

In \secref{s:rat} we establish Properties (5) and (6) stated in the preamble. 

\medskip

We again use \thmref{t:exc}, which allows us to interpret $\on{Exc}(X,\sG)$ as the algebra of global functions on the stack
$$\bMaps_{\on{groups}}(\on{Weil}^{\on{red}}(X),\sG)/\on{Ad}(\sG),$$
where $\on{Weil}^{\on{red}}(X)$ as in \secref{sss:Weil red}.

\medskip

The next observation is that the above pro-reductive group $\on{Weil}^{\on{red}}(X)$ is closely related to the 
pro-semi simple completion $\on{Gal-Drinf}^{\on{arithm}}(X)$ of $\on{Gal}^{\on{arithm}}(X)$ studied by V.~Drinfeld in \cite{Dr}. 

\medskip

The main result of {\it loc. cit.} says that $\on{Gal-Drinf}^{\on{arithm}}(X)$ (which is a priori a pro-semi simple group over $\sfe$)
has a rational structure as an object of Drinfeld's category $\on{Pro-ss}(\sfe)$. This category is a rather coarse object; however,
the rational structure on $\on{Gal-Drinf}^{\on{arithm}}(X)$ turns out to be exactly enough to induce a rational structure on $\on{Exc}(X,\sG)$.

\medskip

A small complication arises in establishing the \emph{continuity} of the $\on{Gal}(\ol\BQ/\BQ)$-action on the model of
$\on{Exc}(X,\sG)$ over $\ol\BQ$ (an issue that is not explicitly addressed in \cite{Dr}). We deal with it in Sects.
\ref{ss:cont}-\ref{ss:cont-bis}. 

\ssec{Conventions and notation}

The conventions in this paper follow those of \cite{AGKRRV}. Let us recap the following two points: 

\sssec{}

We will be working with the $(\infty,2)$-category $\DGCat$ of $\sfe$-linear DG categories; 
see \cite[Sect. A.2]{GRV} for a detailed discussion. This $(\infty,2)$-category carries a symmetric
monoidal structure (the Lurie tensor product); the unit for it is the DG category $\Vect_\sfe$ of chain 
complexes of $\sfe$-vector spaces. 

\medskip

For an algebra (resp., coalgebra) object $\bA\in \DGCat$, we let $\bA\mmod$ (resp., $\bA\commod$) denote 
the category of $\bA$-modules (resp., comodules) in $\DGCat$.

\sssec{}

There will be two kinds of algebraic geometry in this paper: 

\medskip

One is taking place over the ground field $\BF_q$,
and the main actor will be a fixed scheme $X_0$. 
We will be interested in various categories of (ind-constructible) sheaves on $X:=\Spec(\ol\BF_q)\underset{\Spec(\BF_q)}\times X_0$;
these will be objects of $\DGCat$.  Derived algebraic geometry will play no role here.

\medskip

The other algebraic geometry is taking place over $\sfe$ and the objects that we will be working with are 
inherently derived; our principal actors, such as $\LS_\sG^{\on{restr}}(X)$ and $\LS_\sG^{\on{arithm}}(X)$ are examples.

\ssec{Acknowledgements}

We wish to thank A.~Beilinson, P. ~Boisseau, R.~van Dobben de Bruyn,  A.~Eteve, V.~Lafforgue, S.~Raskin, J.~Scholbach, P.~Scholze and C.~Xue
for very helpful discussions. 

\medskip

The main part of this work was done while the second author was visiting MPIM in summer 2025. 

\section{Categories with an action of the monoid \texorpdfstring{$\BA^1$}{A1}} \label{s:A1}

The goal of this section is to state and then prove \thmref{t:A1}. 

\ssec{Statement of the result}

\sssec{}

Recall that the operations of equivariantization and de-equivariantization 
$$\bC\mapsto \bC^H \text{ and } \bD\mapsto \Vect_{\sfe}\underset{\Rep(H)}\otimes \bD$$
give rise to mutually inverse equivalences
\begin{equation} \label{e:eq de-eq}
\QCoh(H)\commod \simeq \Rep(H)\mmod
\end{equation} 
for any algebraic group $H$, where $(-)\mmod$ and $(-)\commod$ are taken in $\DGCat$. 

\medskip

The equivalence \eqref{e:eq de-eq} is symmetric monoidal, where we endow the left-hand side with
the symmetric monoidal structure compatible with the forgetful functor
$$\QCoh(H)\commod\to \DGCat,$$
and the right-hand side with tensor product over $\Rep(H)$.

\sssec{}

We take $H=\BG_m$, and obtain an equivalence
\begin{equation} \label{e:G_m actions}
\QCoh(\BG_m)\commod \simeq \BZ\mmod,
\end{equation}
where we have identified
$$\Rep(\BG_m)\mmod \simeq \BZ\mmod,$$
since $\Rep(\BG_m)$ is 
$$\BZ\times \Vect\simeq \Vect^{\BZ}$$ as a (symmetric monoidal) DG category. 

\sssec{}

The question that we want to address in this section is the following: what is the counterpart of 
\eqref{e:G_m actions} if we replace the $\BG_m$ in the left-hand side by the monoid 
$\BA^1$? 

\sssec{} \label{sss:filtr}

Let $\DGCat^{\on{Filtr}}$ denote the category whose objects are DG categories $\bD$ equipped with a sequence
of full subcategories 
$$\bD_{\leq n}\overset{i_{\leq n}}\hookrightarrow \bD,\quad n\in \BZ$$
such that:

\begin{itemize}

\item $\bD_{\leq n}\subset \bD_{\leq n+1}$;

\item Each $i_n$ admits a continuous right adjoint;

\item The functor $\underset{n}{\on{colim}}\, \bD_{\leq n}\to \bD$ is an equivalence.

\end{itemize} 

\medskip

Morphisms in this category are functors $G:\bD'\to \bD''$ that map the corresponding subcategories
to one another, and such that the Beck-Chevalley conditions hold, i.e., for every $n$, the natural transformation
\begin{equation} \label{e:Beck-Chevalley}
G_{\leq n}\circ i_{\leq n}^R\to i_{\leq n}^R\circ G
\end{equation} 
is an isomorphism, where $G_{\leq n}$ denotes the induced functor $\bD'_{\leq n}\to \bD''_{\leq n}$. 

\sssec{}

Let $\BZ\mmod^{\on{Filtr}}$ be the full subcategory in
$$\BZ\mmod\underset{\DGCat}\times \DGCat^{\on{Filtr}},$$ 
whose objects consist of
$$(\bD\in \BZ\mmod, \bD_i\hookrightarrow \bD),$$
for which the action of the generator $1\in \BZ$ maps $\bD_{\leq n}$ isomorphically onto $\bD_{\leq n+1}$. 

\sssec{}

We will prove the following theorem:

\begin{thm} \label{t:A1}
There is a canonical equivalence
$$\QCoh(\BA^1)\commod \simeq \BZ\mmod^{\on{Filtr}}$$
that makes the diagram
\begin{equation} \label{e:A1}
\CD
\QCoh(\BA^1)\commod @>{\sim}>> \BZ\mmod^{\on{Filtr}} \\
@VVV @VVV \\
\QCoh(\BG_m)\commod  @>{\sim}>> \BZ\mmod
\endCD
\end{equation}
commute, where the vertical arrows are the forgetful functors.
\end{thm} 

\begin{rem}

Both sides in \thmref{t:A1} carry naturally defined symmetric monoidal structures, compatible with the vertical
functors in \eqref{e:A1}. It will
follow from the construction that the equivalence in \thmref{t:A1} is symmetric monoidal.

\end{rem} 

\sssec{Example} \label{e:A mod}

Let $A$ be a non-negatively graded associative algebra. The resulting coaction of $\BA^1$ on $A$, which is a
homomorphism
$$A\to A[t],$$
gives rise to a functor
$$A\mod\to A[t]\mod\simeq A\mod\otimes \QCoh(\BA^1),$$
which makes $A\mod$ into an object of $\QCoh(\BA^1)\commod$.
Let us describe the corresponding object of $\BZ\mmod^{\on{Filtr}}$.

\medskip

First, the underlying object of $\BZ\mmod$ is $A\mod^{\on{gr}}$, the 
category of graded $A$-modules.

\medskip

Let $A\mod^{\on{gr}}_{\leq n}$ be the full subcategory generated by $A(m)$ for $m\leq n$, where
$A(m)_i=A_{m+i}$. Informally, $A\mod^{\on{gr}}_{\leq n}$ consists of objects generated by
elements of degrees $\leq n$.

\medskip

Equivalently, $A\mod^{\on{gr}}_{\leq n}$ is the \emph{left} orthogonal of the full subcategory consisting
of modules that are concentrated in degrees $>n$. 

\medskip

Then the object of $\BZ\mmod^{\on{Filtr}}$ corresponding to $A\mod$ under the equivalence of \thmref{t:A1} is 
$$\{A\mod^{\on{gr}}_{\leq n}\}.$$

\ssec{The action of zero}

\sssec{}

Let $\bC$ be an object of $\QCoh(\BA^1)\commod$. Taking the fiber of the coaction functor
$$\bC\to \bC\otimes \QCoh(\BA^1)$$
at $0$, we obtain an endofunctor of $\bC$ that we denote by $[0]$.

\medskip

Let us describe this endofunctor explicitly in terms of the equivalence of \thmref{t:A1}. 

\sssec{}

Set $\bD:=\bC^{\BG_m}$. The endofunctor $[0]$ respects the action of $\BG_m$, and thus
gives rise to an endofunctor of $\bD$; we denote it by the same character $[0]$.

\medskip

The questions of describing the above two versions of $[0]$ are equivalent under
equivariantization and de-equivariantization, so it is sufficient to deal with the latter.

\sssec{}

By \thmref{t:A1}, $\bD$ upgrades to an object of $\BZ\mmod^{\on{Filtr}}$.

\medskip

For $n\in \BZ$ denote by $j_{\geq n}$ the \emph{right adjoint} of the projection
$$\bD\to \bD/\bD_{< n}=:\bD_{\geq n}.$$

Similarly, denote by $j_{=n}$ the \emph{right adjoint} of the projection
$$\bD_{\leq n}\to \bD_{\leq n}/\bD_{<n}=:\bD_{=n}.$$

\sssec{} \label{sss:phi and psi}

Denote
$$\on{gr}(\bD):=\underset{n}\oplus\, \bD_{=n}\simeq \underset{n}\Pi\, \bD_{=n}.$$

Note that the action of $\BZ$ on $\bD$ identifies
$$\on{gr}(\bD)\simeq (\bD_{=0})^{\BZ}.$$

\medskip

Denote by $\phi_\bullet$ the functor
$$\on{gr}(\bD)\to \bD,$$
whose $n$th component $\phi_n$ is $i_{\leq n}\circ j_{=n}$. 

\medskip

Denote by $\psi_\bullet$ the functor
$$\bD\to \on{gr}(\bD),$$
whose $n$th component $\psi_n$ is $j_{=n}^L\circ i^R_{\leq n}$. 

\medskip

Note that the functors $\phi_\bullet$ and $\psi_\bullet$ are compatible with the action of $\BZ$ on the two sides. 

\begin{rem} 

Note that although each $\phi_n$ is fully faithful, the functor $\phi_\bullet$ is not:
the essential images of the different $\phi_n$ in $\bD$ interact non-trivially. 

\end{rem} 

\sssec{}

The following will be proved along with \thmref{t:A1}:

\begin{prop} \label{p:action of 0}
The action of $[0]$ on $\bD$ is given by the composition
$$\bD \overset{\psi_\bullet}\to \on{gr}(\bD) \overset{\phi_\bullet}\to \bD.$$
\end{prop} 

\sssec{}

Recall that $\bD$ was obtained as $\bC^{\BG_m}$. Let
$$\phi:\bD_0\to \bC \text{ and } \psi:\bC\to \bD_0$$
be the functors corresponding to $\phi_\bullet$ and $\psi_\bullet$ by 
the equivalence \eqref{e:G_m actions}.

\medskip

Applying to \propref{p:action of 0} the equivalence \eqref{e:G_m actions}, we obtain:

\begin{cor} \label{c:action of 0} 
The action of $[0]$ on $\bC$ factors as
$$\bC \overset{\psi}\to \bD_0 \overset{\phi}\to \bC.$$
\end{cor} 

In what follows we will refer to $\bD_0$ as the \emph{attracting category} corresponding to $\bC$. Note that by construction,
the endofunctor
$$\psi\circ \phi$$
of $\bD_0$ is canonically isomorphic to the identity. 

\sssec{}

Note that the assignment
$$(\bC\in \QCoh(\BA^1)\commod) \mapsto (\bD_0\in \DGCat)$$
has a structure of right lax symmetric monoidal functor. In particular, it
maps (commutative) algebra objects in $\QCoh(\BA^1)\commod$ to
 (commutative) algebra objects in $\DGCat$, i.e., (symmetric) monoidal
 DG categories. 

\medskip

It will follow from the construction that the functors $\phi$ and $\psi$ appearing in \corref{c:action of 0},
viewed as natural transformations between functors
$$\QCoh(\BA^1)\commod \leftrightarrow \DGCat,$$
have natural symmetric monoidal structures, and the identification in \corref{c:action of 0} is compatible
with these structures. 

\medskip

In particular, if $\bC$ is a (commutative) algebra object in $\QCoh(\BA^1)\commod$,
and hence $\bD_0$ is a (symmetric) monoidal DG category, the functors $\psi$ and $\phi$ 
have natural symmetric monoidal structures, and the identification of \corref{c:action of 0} is compatible
with these structures. 

\sssec{}

Let us see what \corref{c:action of 0} says in the example in \secref{e:A mod}. Let $A_0$ be the degree $0$-component of $A$.
We have the natural homomorphisms
$$A\to A_0 \text{ and } A_0\to A.$$

\medskip

The endofunctor $[0]$ of $A\mod$ acts as
$$A\mod \overset{A_0\underset{A}\otimes (-)}\longrightarrow A_0\mod \overset{A\underset{A_0}\otimes (-)}\longrightarrow A\mod.$$

\medskip

We identify
$$A\mod^{\on{gr}}_{=n} \simeq A_0\mod$$
for any $n$.

\medskip

It is easy to see that under this identification, the functor 
$$\phi_n:A_0\mod\to A\mod^{\on{gr}}$$
corresponds to
$$A_0\mod \overset{\on{deg}=n}\longrightarrow A_0\mod^{\on{gr}} \overset{A\underset{A_0}\otimes (-)}\longrightarrow A\mod^{\on{gr}},$$
and the functor 
$$\psi_n:A\mod^{\on{gr}}\to A_0\mod$$
to
$$A\mod^{\on{gr}} \overset{A_0\underset{A}\otimes (-)}\longrightarrow A_0\mod^{\on{gr}}   \overset{\on{deg}=n}\longrightarrow A_0\mod.$$

From here, we obtain that the functors $\phi$ and $\psi$ identify with
$$A_0\mod \overset{A\underset{A_0}\otimes (-)}\longrightarrow A\mod$$
and 
$$A\mod \overset{A_0\underset{A}\otimes (-)}\longrightarrow A_0\mod,$$
respectively.

\ssec{Proof of \thmref{t:A1}, Step 1} 

\sssec{} \label{sss:sym mon com}

Note that when $H$ is commutative, the equivalence \eqref{e:eq de-eq} is compatible with another pair of
symmetric monoidal structures, where:

\begin{itemize}

\item In the left-hand side, we take
$$\bC'_1,\bC'_2\mapsto (\bC'_1\otimes \bC'_2)^H,$$
which carries an action of $H$ through either of the factors due to the
commutativity of $H$;

\medskip

\item In the left hand side we take
$$\bC''_1,\bC''_2\mapsto \bC'_1\otimes \bC''_2,$$ 
with the action of $\Rep(H)$ given via the (symmetric) monoidal functor $$\Rep(H)\to \Rep(H)\otimes \Rep(H)$$
that arises as pullback along the multiplication map 
$$\on{pt}/H\times \on{pt}/H\to \on{pt}/H,$$
which exists thanks to the commutativity of $H$.

\end{itemize}

\sssec{}

Let us regard \eqref{e:G_m actions} as a symmetric monoidal equivalence with respect to the 
symmetric monoidal structures specified in \secref{sss:sym mon com}. 

\medskip

Note that under the equivalence
$$\Rep(\BG_m)\mmod\simeq \BZ\mmod,$$
this is the symmetric monoidal structure compatible with the forgetful functor
\begin{equation} \label{e:forget Z}
\BZ\mmod\to \DGCat.
\end{equation}

I.e., the monoidal operation is given by  
$$\bC''_1,\bC''_2\mapsto \bC'_1\otimes \bC''_2,$$ 
equipped with the diagonal action of $\BZ$. 

\sssec{}

The operation of pullback along 
$$\BA^1\overset{\BG_m}\times \BA^1\to \BA^1$$
makes $\QCoh(\BA^1)$ into a coalgebra in $\QCoh(\BG_m)\commod$
with respect to the (symmetric) monoidal structure in \secref{sss:sym mon com}.

\medskip

We have a tautological equivalence
\begin{equation} \label{e:red1}
\QCoh(\BA^1)\commod \simeq \QCoh(\BA^1)\commod(\QCoh(\BG_m)\commod).
\end{equation} 

\sssec{}

Under the equivalence \eqref{e:G_m actions}, the coalgebra object 
$$\QCoh(\BA^1)\in \QCoh(\BG_m)\commod$$
corresponds to
$$\QCoh(\BA^1/\BG_m)\in \BZ\mmod,$$
with the coalgebra structure given by pullback along
$$(\BA^1/\BG_m)\times (\BA^1/\BG_m)\to \BA^1/\BG_m.$$

\medskip

Applying \eqref{e:G_m actions}, we obtain an equivalence
\begin{equation} \label{e:red2}
\QCoh(\BA^1)\commod(\QCoh(\BG_m)\commod) \simeq \QCoh(\BA^1/\BG_m)\commod(\BZ\mmod).
\end{equation} 

\sssec{}

Applying the forgetful functor \eqref{e:forget Z}, we can view $\QCoh(\BA^1/\BG_m)$ as a coalgebra
in $\DGCat$. 

\medskip

We will prove an equivalence:
\begin{equation} \label{e:A1,1}
\QCoh(\BA^1/\BG_m)\commod \simeq \DGCat^{\on{Filtr}},
\end{equation}
compatible with the forgetful functors of both sides to $\DGCat$. 

\medskip

The compatibility with the $\BZ$-actions will follow from the construction, and that would 
lead to an equivalence
$$\QCoh(\BA^1/\BG_m)\commod(\BZ\mmod) \simeq \BZ\mmod^{\on{Filtr}}.$$

Composing this with \eqref{e:red1} and \eqref{e:red2}, we obtain the statement of \thmref{t:A1}.

\ssec{Proof of \thmref{t:A1}, Step 2} 

\sssec{}

Let $\on{Filtr}$ denote the endofunctor on $\DGCat$ that sends 
$$\bC\mapsto \on{Filtr}(\bC):=\on{Funct}(\BZ,\bC),$$
where we view $\BZ$ as a poset and hence a category.

\medskip

Note that $\on{Filtr}$ has a natural comonad structure:

\medskip

The binary operation is defined by 
\begin{align*} 
&\on{Funct}(\BZ,\bC)\to \on{Funct}(\BZ,\on{Funct}(\BZ,\bC))=\on{Funct}(\BZ\times \BZ,\bC), \\ 
&\{\bc_n\}\in \on{Funct}(\BZ,\bC) \mapsto \{\bc_{n_1,n_2}\}\in \on{Funct}(\BZ\times \BZ,\bC), \quad \bc_{n_1,n_2}:=\bc_{\on{min}(n_1,n_2)}.
\end{align*} 

The counit is given by
$$\on{Funct}(\BZ,\bC)\to \bC, \quad \{\bc_n\}\in \on{Funct}(\BZ,\bC) \mapsto \underset{n}{\on{colim}}\, \bc_n.$$

\sssec{}

Recall that for $\bC\in \DGCat$, the category 
$$\bC\otimes \QCoh(\BA^1/\BG_m)$$
identifies canonically with $\on{Filtr}(\bC)$, see, e.g., \cite[Chapter 9, Sect. 1.3]{GaRo2}.

\medskip

We claim:

\begin{prop} \label{p:comonads}
The isomorphism of endofunctors of $\DGCat$
\begin{equation} \label{e:comonads}
\on{Filtr}(-)\simeq (-)\otimes \QCoh(\BA^1/\BG_m)
\end{equation} 
upgrades to an isomorphism of comonads, where the comonad structure 
on the right-hand side comes from the coalgebra structure on $\QCoh(\BA^1/\BG_m)$. 
\end{prop}

\begin{proof}

Note that the comonad $\on{Filtr}$ on $\DGCat$ commutes with tensor products
by objects of $\DGCat$. Hence, it is also given by tensoring with a coalgebra
object in $\DGCat$, to be denoted $\QCoh(\BA^1/\BG_m)'$.

\medskip

It follows from \eqref{e:comonads} that the plain object of $\DGCat$ underlying $\QCoh(\BA^1/\BG_m)'$
identifies with $\QCoh(\BA^1/\BG_m)$. Thus, it remains to show that the two comonad structures
agree.

\medskip

It follows by direct inspection that the two binary operations 
$$\QCoh(\BA^1/\BG_m)\rightrightarrows \QCoh(\BA^1/\BG_m)\otimes \QCoh(\BA^1/\BG_m)$$
agree.

\medskip

We claim that this implies that the entire comonad structures agree. Indeed, note that both structures are 
given by t-exact functors, and the categories involved are derived categories of their hearts.

\end{proof} 

\sssec{}

Thus, we obtain that in order to prove \eqref{e:A1,1}, we need to establish an equivalence 
\begin{equation} \label{e:A1,2}
\on{Filtr}\commod(\DGCat) \simeq \DGCat^{\on{Filtr}},
\end{equation}
compatible with the forgetful functors to $\DGCat$.

\ssec{Proof of \thmref{t:A1}, Step 3} 

\sssec{}

We claim that the forgetful functor
$$\oblv_{\on{Filtr}}:\DGCat^{\on{Filtr}}\to \DGCat$$
admits a right adjoint, to be denoted $\ind_{\on{Filtr}}$.

\medskip

Namely, $\ind_{\on{Filtr}}$ sends $\bC\in \DGCat$ to an object of $\DGCat^{\on{Filtr}}$, whose underlying DG category
is $\on{Funct}(\BZ,\bC)$, and where
$$\on{Funct}(\BZ,\bC)_{\leq n}:=(\{\bc_n\}\,\mid\, \bc_m\to \bc_{m+1} \text{ are isomorphisms for } m\geq n).$$

\sssec{}

Let us construct the adjunction data. The counit of the adjunction, which is a functor
$$\on{Funct}(\BZ,\bC)\to \bC,$$
is given by
$$\{\bc_n\}\mapsto \underset{n}{\on{colim}}\, \bc_n.$$

\medskip

The unit of the adjunction, which is a functor, 
$$\bD\mapsto \on{Funct}(\BZ,\bD), \quad \bD\in \DGCat^{\on{Filtr}},$$
is given by
$$\bd\mapsto \{\bd_n\}, \quad \bd_n:=i_{\leq n}\circ i_{\leq n}^R(\bd).$$

\sssec{}

It follows from the construction that the comonad on $\DGCat$ attached to the above adjunction 
identifies with $\on{Filtr}$.

\medskip

This gives rise to a functor
$$\DGCat^{\on{Filtr}}\to \on{Filtr}\commod(\DGCat).$$

\medskip

In order to show that the functor is an equivalence, we have to show that the functor $\oblv_{\on{Filtr}}$
satisfies the assumptions of the (comonadic) Barr-Beck-Lurie theorem. 

\sssec{}

First, we claim that $\oblv_{\on{Filtr}}$ is conservative. 

\medskip

Let $G:\bD'\to \bD''$ be an equivalence, and we need to show that it induces an equivalence 
$$\bD'_{\leq n}\to \bD''_{\leq n}.$$

However, this follows from the Beck-Chevalley isomorphisms \eqref{e:Beck-Chevalley}.

\sssec{}

We now claim that $\oblv_{\on{Filtr}}$ commutes with all limits. Let
$$\alpha\mapsto (\bD_\alpha,\{\bD_{\alpha,\leq n}\})$$
be a diagram in $\DGCat^{\on{Filtr}}$. 

\medskip

Set
$$\bD:=\underset{\alpha}{\on{lim}}\, \bD_\alpha \text{ and } \bD_{\leq n}:=\underset{\alpha}{\on{lim}}\, \bD_{\alpha,\leq n}.$$

The adjoint functors
$$i_{\alpha,\leq n}:\bD_{\alpha,\leq n}\rightleftarrows \bD_\alpha:(i_{\alpha,\leq n})^R$$
give rise to adjoint functors
$$i_{\leq n}:\bD_{\leq n}\rightleftarrows \bD:(i_{\leq n})^R,$$
and both commute with 
$$\bD_{\leq n}\to \bD_{\alpha,\leq n} \text{ and } \bD\to \bD_\alpha.$$

We claim that $(\bD,\{\bD_{\leq n}\})$ is indeed an object of $\DGCat^{\on{Filtr}}$. For that we need to show that 
the functor
$$\underset{n}{\on{colim}}\, \bD_{\leq n}\to \bD$$
is an equivalence.

\medskip

By \cite[Chapter 1, Proposition 2.5.7]{GaRo1}, this is equivalent to 
$$\bD \to \underset{n}{\on{lim}}\, \bD_{\leq n}$$
being an equivalence, where the limit is taken with respect to the right adjoints of $\bD_{\leq n}\to \bD_{\leq n+1}$. 

\medskip

The required assertion follows now from the commutative diagram
$$
\CD
\bD @>>>  \underset{n}{\on{lim}}\, \bD_{\leq n} \\
@V{=}VV @VV{=}V \\
\underset{\alpha}{\on{lim}}\, \bD_\alpha @>>> \underset{\alpha}{\on{lim}}\, \underset{n}{\on{lim}}\, \bD_{\alpha,\leq n}\simeq
\underset{n}{\on{lim}}\, \underset{\alpha}{\on{lim}}\, \bD_{\alpha,\leq n},
\endCD
$$
where the bottom horizontal arrow is an equivalence again by  \cite[Chapter 1, Proposition 2.5.7]{GaRo1}.

\medskip

It is clear that $(\bD,\{\bD_{\leq n}\})$ constructed above is the limit of $(\bD_\alpha,\{\bD_{\alpha,\leq n}\})$ in 
$\DGCat^{\on{Filtr}}$. By construction
$$\oblv_{\on{Filtr}}(\bD,\{\bD_{\leq n}\})=\bD=
\underset{\alpha}{\on{lim}}\, \oblv_{\on{Filtr}}(\bD_\alpha,\{\bD_{\alpha,\leq n}\}),$$
as required. 

\qed[\thmref{t:A1}]

\sssec{}

Unwinding, we obtain that the functor 
$$\on{Filtr}\commod(\DGCat)\to \DGCat^{\on{Filtr}}$$
is given as follows.

\medskip

It sends 
$$(\bC,\bC\overset{\Phi}\to \on{Funct}(\BZ,\bC))\in \on{Filtr}\commod(\DGCat)$$
to the object 
$$(\bD,\{\bD_{\leq n}\})\in  \DGCat^{\on{Filtr}},$$
where $\bD=\bC$, and $\bD_{\leq n}=:\bC_{\leq n}$ is the full subcategory of $\bC$ consisting of
those objects $\bc$, for which the map
$$\Phi(\bc)_{n'}\to \underset{m}{\on{colim}}\, \Phi(\bc)_m=\bc$$
is an isomorphism for all $n'\geq n$. 

\ssec{Identifying the action of zero} 

We will now show that the equivalence established in \thmref{t:A1} has the property stated in \propref{p:action of 0}.

\sssec{} \label{sss:gr}
 
Note that for $\bC\in \DGCat$, the functor
$$\on{Filtr}(\bC):= \on{Funct}(\BZ,\bC) \simeq \bC\otimes \QCoh(\BA^1/\BG_m) \overset{\on{Id}\otimes \iota^*}\longrightarrow 
\bC\otimes \QCoh(\on{pt}/\BG_m)\simeq \bC^{\BZ}$$
(where $\iota$ is the embedding $\on{pt}/\BG_m\to \BA^1/\BG_m$) is the functor of the \emph{associated graded}
$$\on{gr}(\{\bc_n\})=\{\on{cofib}(\bc_{n-1}\to \bc_n)\}.$$

\sssec{}

For $\bD\in \DGCat^{\on{Filtr}}$, the corresponding object in $\QCoh(\BA^1/\BG_m)\commod$ has $\bD$ as the underlying category
and the coaction 
$$\bD\to \bD\otimes \QCoh(\BA^1/\BG_m)\simeq \on{Funct}(\BZ,\bD)$$
is given by
$$\bd\mapsto \{i_{\leq n}\circ i_{\leq n}^R(\bd)\}.$$

\medskip

Combining with \secref{sss:gr}, we obtain that the functor $[0]$, which in the present context is the functor
\begin{equation} \label{e:action of 0 comp}
\bD\to \bD\otimes \QCoh(\BA^1/\BG_m)  \overset{\on{Id}\otimes \iota^*}\longrightarrow \bD\otimes \QCoh(\on{pt}/\BG_m)\simeq \bD^{\BZ}
\end{equation} 
is given by
$$\bd \mapsto \{\on{cofib}(i_{\leq n-1}\circ i_{\leq n-1}^R(\bd)\to i_{\leq n}\circ i_{\leq n}^R(\bd))\}.$$

\sssec{}

Finally, we note that for every $n$, we have
$$\on{cofib}(i_{\leq n-1}\circ i_{\leq n-1}^R(-)\to i_{\leq n}\circ i_{\leq n}^R(-))\simeq \phi_n\circ \psi_n,$$
where $\phi_n$ and $\psi_n$ are as in \secref{sss:phi and psi}. 

\section{A contraction on the category of Weil sheaves}  \label{s:Weil}

The goal of this section is to state and prove Theorem-Construction \ref{t:contraction on category}.  In order to do so, we will
need to introduce a little zoo of various versions of the category of Weil sheaves. 

\ssec{Various algebro-geometric completions of \texorpdfstring{$\BZ$}{Z}} 

In this subsection we will introduce the pro-algebraic groups
$$\BZ^{\on{alg}},\,\, \BZ^{\on{alg,wt}} \text{ and } \BZ^{\on{alg},0},$$
which would play an important role throughout the paper: the will serve as ``equivariantization/de-equivariantization" agents. 

\sssec{}

Let $\BZ\mod$ be the category of representations of $\BZ$. We consider
it as a symmetric monoidal category under tensor product, so that the forgetful functor
$$\BZ\mod\to \Vect_{\sfe}$$
is symmetric monoidal.

\medskip

We have
\begin{equation} \label{e:Z vs Gm}
\BZ\mod\simeq \QCoh(\BG_m),
\end{equation} 
where we consider $\QCoh(\BG_m)$ as a symmetric monoidal category with respect to \emph{convolution}.
Under this equivalence, the forgetful functor $\BZ\mod\to \Vect_\sfe$ corresponds to the functor
$$\Gamma(\BG_m,-):\QCoh(\BG_m)\to \Vect_\sfe.$$

\medskip

Let  
$$\BZ\mod_{\on{loc.fin}}\subset \BZ\mod$$ be the full subcategory that consists of locally finite representations.
Under \eqref{e:Z vs Gm} it corresponds to the full subcategory 
$$\QCoh(\BG_m)_{\on{tors}}\subset \QCoh(\BG_m)$$
consisting of quasi-coherent sheaves, whose set-theoretic support lies in the union of closed points of $\BG_m$. 

\medskip

The category $\BZ\mod_{\on{loc.fin}}$ together with the forgetful functor to $\Vect_{\sfe}$ corresponds
to a pro-algebraic group over $\sfe$, which we denote by $\BZ^{\on{alg}}$.

\sssec{} \label{sss:Weil num}

Recall that an element of $\sfe:=\ol\BQ_\ell$ is called a $q$-\emph{Weil number} of weight $n$ if:

\medskip

\begin{enumerate}

\item It belongs to $\ol\BQ$;

\medskip

\item Its image under any embedding $\ol\BQ\to \BC$ has absolute value $q^{n/2}$;

\medskip

\item It is an algebraic integer outside of $p$ (i.e., becomes an algebraic integer after multiplication by $p^m$ for some $m$).

\end{enumerate}

\medskip

It is easy to see that any Weil number is actually a \emph{unit} outside of $p$ (indeed, take its norm and we obtain a
rational number with norm a power of $p$).

\medskip

We define the notion of \emph{mock} $q$-Weil number by omitting condition (3) in the definition of Weil numbers. 

\sssec{}

Let 
$$\BZ\mod_{\on{wt}} \subset \BZ\mod_{\on{loc.fin}}$$
be the full subcategory, consisting of objects, on which the eigenvalues of the generator $1\in \BZ$ are $q$-Weil numbers
in $\sfe$. This is a symmetric monoidal subcategory because Weil numbers form a subgroup of $\sfe^\times$.

\medskip

Under \eqref{e:Z vs Gm}, it corresponds to the full subcategory 
$$\QCoh(\BG_m)_{\on{wt}}\subset \QCoh(\BG_m)_{\on{tors}}$$
consisting of quasi-coherent sheaves, whose set-theoretic support belongs to the set of $q$-Weil numbers.

\sssec{} \label{sss:structure W}

The category $\BZ\mod_{\on{wt}}$ together with the forgetful functor to $\Vect_{\sfe}$ corresponds
to a pro-algebraic group over $\sfe$, which we denote by $\BZ^{\on{alg,wt}}$.

\medskip

By construction, $\BZ^{\on{alg,wt}}$ comes equipped with the following structures:

\begin{itemize}

\item The inclusion $\BZ\mod_{\on{wt}} \hookrightarrow \BZ\mod$ gives rise to a 
homomorphism $\BZ\to \BZ^{\on{alg,wt}}$ with Zariski-dense image; we denote by $\Frob\in \BZ^{\on{alg,wt}}$ 
the image of the generator $1\in \BZ$.

\medskip

\item We have a canonical homomorphism $\BG_m\to \BZ^{\on{alg,wt}}$; it corresponds to the
decomposition of $\Rep(\BZ^{\on{alg,wt}})\simeq \BZ\mod_{\on{wt}}$ according to weights;
a choice of $q^{\frac{1}{2}}\in \sfe$ gives rise to a left inverse of the above map.

\end{itemize}

\begin{rem} \label{r:Z wt over Q}

Note the following key difference between $\BZ^{\on{alg,wt}}$ and $\BZ^{\on{alg}}$: the former
is defined over $\BQ\subset \sfe$, whereas the latter is not.

\end{rem}

\sssec{}

Let $\BZ^{\on{alg},0}$ denote the quotient $\BZ^{\on{alg,wt}}/\BG_m$. Note that
$\Rep(\BZ^{\on{alg},0})$ is a full subcategory of 
$$\Rep(\BZ^{\on{alg,wt}})=\BZ\mod_{\on{wt}}\subset \BZ\mod_{\on{loc.fin}}\subset \BZ\mod$$
consisting of modules, in which the generator $1\in \BZ$ acts with generalized eigenvalues that are 
$q$-Weil numbers of weight $0$. 

\medskip

Equivalently, 
$$\Rep(\BZ^{\on{alg},0})\subset \Rep(\BZ^{\on{alg,wt}})=\QCoh(\BG_m)_{\on{wt}}\subset  \QCoh(\BG_m)_{\on{tors}}\subset \QCoh(\BG_m)$$
is the subcategory
$$\QCoh(\BG_m)_{0}\subset \QCoh(\BG_m)$$
that consists of quasi-coherent sheaves, whose set-theoretic support belongs to the set of $q$-Weil numbers
of weight $0$. 

\sssec{}

Here are some more remarks about the structure of $\BZ^{\on{alg,wt}}$:

\medskip

\begin{itemize}

\item Its unipotent factor is isomorphic to $\BG_a$; indeed, its category of representations is 
the subcategory $\QCoh(\BG_m)_{\{1\}}\subset \QCoh(\BG_m)$ of quasi-coherent sheaves 
set-theoretically supported at $1\in \BG_m$;

\medskip

\item $\pi_0(\BZ^{\on{alg,wt}})\simeq \wh\BZ$; indeed its characters are the roots of unity in $\sfe$;

\medskip

\smallskip

\item The reductive part of the connected component of $\BZ^{\on{alg},0}$ is the pro-torus,
whose group of characters is the quotient of the group of Weil numbers of weight $0$ by roots 
of unity.

\end{itemize}

%
%

\sssec{}

Let $\BH$ be $\BZ^{\on{alg}}$, $\BZ^{\on{alg,wt}}$ or $\BZ^{\on{alg},0}$. 
Write 
$$\BH\simeq \underset{\alpha}{\on{lim}}\, H_\alpha,$$
where $H_\alpha$ are finite-dimensional algebraic groups.

\medskip

We claim:

\begin{prop} \label{p:W rep} \hfill

\smallskip

\noindent{\em(a)} The category $\Rep(\BH)$ is compactly generated by finite-dimensional representations;

\smallskip

\noindent{\em(b)} The monoidal category $\Rep(\BH)$ is rigid;

\smallskip

\noindent{\em(c)} The functor
\begin{equation} \label{e:colimit groups}
\underset{\alpha}{\on{colim}}\, \Rep(H_\alpha)\to \Rep(\BH)
\end{equation} 
is an equivalence.

\end{prop}

\begin{proof}

Point (a) is obvious from the description of $\Rep(\BH)$ as a full subcategory in $\QCoh(\BG_m)_{\on{tors}}$.

\medskip

Point (b) follows from the fact that compact objects in $\Rep(\BH)$ are dualizable (which in turn follows from point (a)). Point (c)
holds for any pro-algebraic group for which (a) holds:

\medskip

Indeed, it suffices to show that the functor \eqref{e:colimit groups} is fully faithful, and given (a) it suffices to check that it
is fully faithful on compact objects. I.e., it suffices 
to show that for an index $\alpha_0$, $i\in \BZ^{\geq 0}$ and $V_1,V_2\in \Rep(H_{\alpha_0})^\heartsuit$, the map
$$\underset{\alpha\geq \alpha_0}{\on{colim}}\, \Ext^i_{\Rep(H_\alpha)}(\Res^{H_{\alpha_0}}_{H_\alpha}(V_1),\Res^{H_{\alpha_0}}_{H_\alpha}(V_2))\to 
\Ext^i_{\Rep(\BH)}(\Res^{\BH}_{H_\alpha}(V_1),\Res^{\BH}_{H_\alpha}(V_2))$$
is an isomorphism. However, the latter is the case for any pro-algebraic group. 

\end{proof} 

\sssec{}

From \propref{p:W rep} we obtain:

\begin{cor} \label{c:eq de-eq W}
For $\BH$ being $\BZ^{\on{alg}}$, $\BZ^{\on{alg,wt}}$ or $\BZ^{\on{alg},0}$, the operations
$$\bC\mapsto \bC^\BH \text{ and } \bD\mapsto \Vect_{\sfe}\underset{\Rep(\BH)}\otimes \bD$$
define mutually inverse equivalences
$$\QCoh(H)\commod \leftrightarrow \Rep(H)\mmod.$$
\end{cor}

\ssec{Sheaf theories} \label{ss:sheaf theories}

\sssec{}  \label{sss:ground}

Let $k$ be a ground field, and let $\Sch_k$ be the category of separated schemes of finite type over $k$.
Denote $\on{pt}:=\Spec(k)$.

\medskip

Let $\on{Corr}(\Sch_k)$ be the category of correspondences on $\Sch_k$ (see \cite[Sect. 2.1]{GRV}), and 
let 
$$\Shv(-):\on{Corr}(\Sch_k)\to \DGCat$$
be a sheaf theory in the sense of \cite[Sect. 2.2]{GRV}, equipped with a \emph{right-lax} symmetric monoidal structure. 

\medskip

Informally, this means:

\medskip

\begin{itemize}

\item For every $X\in \Sch_k$ we have an object $\Shv(X)\in \DGCat$;

\medskip

\item For every morphism $f:X_1\to X_2$ we have functors $$f_*:\Shv(X_1)\to \Shv(X_2) \text{ and }
f^!:\Shv(X_2)\to \Shv(X_1)$$ that satisfy base change for fiber squares;

\medskip

\item We have a functor $\Shv(X_1)\boxtimes \Shv(X_2)\overset{\boxtimes}\to \Shv(X_1\times X_2)$;

\medskip

\item A homotopy-coherent data of compatibility between the above pieces of structure.

\end{itemize}

\sssec{}

Note that the symmetric monoidal structure on $\Shv(-)$ endows $\Shv(\on{pt})$ with a structure
of symmetric monoidal DG category. Moreover, the functor $\Shv(-)$ naturally factors via
$\Shv(\on{pt})\mmod$. 

\medskip

We will make the following additional assumptions (cf. \cite[Sect. 4.1.1]{GRV}): 

\begin{enumerate} 

\item For every $X\in \Sch_k$, the category $\Shv(X)$ is generated by objects that are 
\emph{totally compact}\footnote{Recall that for a monoidal DG category $\bA$ and $\CM\in \bA\mmod$, an object $m\in \CM$ is said to be totally
compact with respect to $\bA$ if the functor $\uHom_{\CM,\bA}(m,-):\CM\to \bA$ (of inner Hom relative to $\bA$) 
is a map of $\bA$-module categories, see \cite[Sect. B.4]{GRV}.}
with respect to $\Shv(\on{pt})$;

\medskip
 
\item For every $X\in \Sch_k$, the map
$$\Shv(X)\otimes \Shv(X)\to \Shv(\on{pt}),$$
induced by the unit of the canonical self-duality of $X\in \on{Corr}(\Sch_k)$,
is the counit of a self-duality of $\Shv(X)$ as an object of $\Shv(\on{pt})\mmod$. 

\end{enumerate} 

Note that given condition (1), condition (2) can be reformulated as follows: there exists an anti self-equivalence
$$\BD:(\Shv(X)^{\on{tot.comp}})^{\on{op}}\simeq \Shv(X)^{\on{tot.comp}},$$
such that for $\CF\in \Shv(X)^{\on{tot.comp}}$ and $\CF'\in \Shv(X)$,
$$\uHom_{\Shv(X),\Shv(\on{pt})}(\CF,\CF')\simeq (p_X)_*\circ \Delta_X^!(\BD(\CF)\boxtimes \CF'),$$
where:

\begin{itemize}

\item $\uHom_{\Shv(X),\Shv(\on{pt})}(-,-)$ denotes the inner Hom in $\Shv(X)$ relative to the action of $\Shv(\on{pt})$;

\item $\Delta_X$ denotes the diagonal map $X\to X\times X$;

\item $p_X$ denotes the projection $X\to \on{pt}$.

\end{itemize} 

\begin{rem}

In practice, the symmetric monoidal category $\Shv(\on{pt})$ is often semi-rigid
(or even rigid). In this case ``totally compact" is the same as just ``compact".

\end{rem} 

\sssec{}

Let $\Phi:\Shv_1(-)\to \Shv_2(-)$ be a natural transformation between two sheaf theories satisfying 
the above assumptions.

\medskip

Assume that for every $X$, the functor
$$\Phi_X:\Shv_1(X)\to \Shv_2(X)$$
maps objects that are totally compact with respect to $\Shv_1(\on{pt})$ to 
objects that are totally compact with respect to $\Shv_2(\on{pt})$.

\medskip

Note that in this case, for $\CF\in \Shv_1(X)^{\on{tot.comp}}$ we have a tautologically defined map
\begin{equation} \label{e:Phi Verdier}
\BD(\Phi(\CF))\to \Phi(\BD(\CF)). 
\end{equation} 

\medskip

We observe: 

\begin{prop} \label{p:abstract base change}
Assume that the natural transformation \eqref{e:Phi Verdier} is an isomorphism. Then 
the functor
$$\Shv_2(\on{pt})\underset{\Shv_1(\on{pt})}\otimes \Shv_1(X)\to \Shv_2(X)$$
is fully faithful.
\end{prop} 

\ssec{Proof of \propref{p:abstract base change}}

\sssec{} \label{sss:ULA 1}

Let $\bA$ be a monoidal category; let $\CM$ be a dualizable $\bA$-module category, and let 
$m\in \CM^{\on{tot.comp}}$. Denote by $\BD(m)$ the object of $\CM^\vee$ so that
$$\uHom_{\CM,\bA}(m,m')=\langle m',\BD(m) \rangle, \quad m'\in \CM,$$
where:

\begin{itemize}

\item $\uHom_{\CM,\bA}(-,-)$ denotes the inner Hom in $\CM$ relative to $\bA$;

\item $\CM^\vee\in \bA^{\on{rev}}\mmod$ denotes the dual of $\CM$;

\item $\langle-,-\rangle$ denotes the pairing $\CM\otimes \CM^\vee\to \bA$.

\end{itemize}

\sssec{}  \label{sss:ULA 2}

Let now $\Phi_\bA:\bA\to \bA'$ be a homomorphism. Then the object
$$\one_{\bA'}\otimes m\in \bA'\underset{\bA}\otimes \CM$$
is totally compact and we have
$$\BD(\one_{\bA'}\otimes m)\simeq \BD(m)\otimes \one_{\bA'}$$
as objects of
$$\CM^\vee\underset{\bA}\otimes \bA'\simeq (\bA'\underset{\bA}\otimes \CM)^\vee.$$

\sssec{}

Let now $\CM'$ be an object of $\bA'\mmod$, and let us be given a functor
$$\Phi_\CM:\CM\to \CM'$$
as $\bA$-module categories.

\medskip

Suppose that $\CM'$ is also dualizable, and let us be given a functor
$$\Phi_{\CM^\vee}:\CM^\vee\to \CM'{}^\vee$$
of $\bA^{\on{rev}}$-module categories, so that the diagram
\begin{equation} \label{e:abstract functor and duality}
\CD
\CM \otimes \CM^\vee @>{\langle -,-\rangle}>> \bA \\ 
@V{\Phi_{\CM}\otimes \Phi_{\CM^\vee}}VV @VV{\Phi_\bA}V \\
\CM' \otimes \CM'{}^\vee @>{\langle -,-\rangle}>> \bA'
\endCD
\end{equation} 
commutes.

\medskip

Suppose that $\Phi_\CM$ maps totally compact objects to totally compact objects. From 
\eqref{e:abstract functor and duality} we obtain a map
\begin{equation} \label{e:Phi Verdier abstract}
\BD(\Phi_{\CM}(m))\to \Phi_{\CM^\vee}(\BD(m)), \quad m\in \CM^{\on{tot.comp}}.
\end{equation}

\sssec{}

We claim: 

\begin{prop} \label{p:Phi Verdier abstract}
Suppose that $\CM$ is generated by totally compact objects and suppose that
the natural transformation \eqref{e:Phi Verdier abstract} is an isomorphism.
Then the induced functor
$$\Phi'_{\CM}:\bA'\underset{\bA}\otimes \CM\to \CM'$$
is fully faithful.
\end{prop} 

It is clear that the assertion of  \propref{p:abstract base change} is a particular case of that of \propref{p:Phi Verdier abstract}. 

\sssec{Proof of \propref{p:Phi Verdier abstract}}

It suffices to show that for $m\in \CM^{\on{tot.comp}}$ and $m'\in \bA'\underset{\bA}\otimes \CM$, the map
\begin{multline} \label{e:Phi Verdier abstract 1}
\uHom_{\bA'\underset{\bA}\otimes \CM,\bA'}(\one_{\bA'}\otimes m,m') \to 
\uHom_{\CM',\bA'}(\Phi'_{\CM}(\one_{\bA'}\otimes m),\Phi'_{\CM}(m'))\simeq \\
\simeq \uHom_{\CM',\bA'}(\Phi_\CM(m),\Phi'_{\CM}(m'))
\end{multline}
is an isomorphism.

\medskip

By \secref{sss:ULA 2}, we rewrite the above map as
\begin{equation} \label{e:Phi Verdier abstract 2}
\langle m',\BD(m)\otimes \one_{\bA'}\rangle \to \langle \Phi'_{\CM}(m'),\BD(\Phi_\CM(m))\rangle.
\end{equation} 

To prove that \eqref{e:Phi Verdier abstract 2} is an isomorphism, we can take $m'$ of the form $\one_{\bA'}\otimes m_1$ for $m_1\in \CM$. 
In this case, the map \eqref{e:Phi Verdier abstract 2} becomes
\begin{equation} \label{e:Phi Verdier abstract 3}
\Phi_\bA(\langle m_1,\BD(m)\rangle) \to \langle \Phi_\CM(m_1),\BD(\Phi_\CM(m))\rangle.
\end{equation} 

When we compose the latter map with (the isomorphism) \eqref{e:Phi Verdier abstract} we obtain a map
\begin{equation} \label{e:Phi Verdier abstract 4}
\Phi_\bA(\langle m_1,\BD(m)\rangle) \to  \langle \Phi'_{\CM}(m_1),\Phi_{\CM^\vee}(\BD(m))\rangle.
\end{equation} 

However, unwinding, we obtain that the map \eqref{e:Phi Verdier abstract 4} is the isomorphism that expresses
the commutativity of \eqref{e:abstract functor and duality}. 

\qed[ \propref{p:Phi Verdier abstract}]

\ssec{Weil sheaves}

\sssec{} \label{sss:finite field}

Let $X_0$ be a scheme of finite type over $\BF_q$, and let $X$ be its base change to $\ol\BF_q$. Let
$\Shv(X)$ denote the category of $\ell$-adic sheaves on $X$, as defined in \cite[Sect. 1.1]{AGKRRV}.

\medskip

By construction,
$$\Shv(X)\simeq \on{Ind}(\Shv(X)^{\on{constr}}).$$

The category $\Shv(X)$ carries a canonical t-structure, whose heart is the category $\on{Ind}(\on{Perv}(X))$,
where $\on{Perv}(X)$ is the abelian category of perverse sheaves on $X$, as defined in \cite{BBD}. 

\medskip

Recall (see \cite[Theorem 1.1.6]{AGKRRV}) that $\Shv(X)$ is left-complete in its t-structure. 

\sssec{}

Let $\Frob$ denote the geometric Frobenius acting on $X$. Consider the resulting automorphism
$\Frob^*$ of $\Shv(X)$.

\medskip

Define $\Shv^{\on{weak-Weil}}(X)$ to be the category of pairs $(\CF,\alpha)$, where $\CF\in \Shv(X)$
and $\alpha$ is a map $\Frob^*(\CF)\to \CF$, which by adjunction is the same as a map
\begin{equation} \label{e:other alpha}
\CF\to \Frob_*(\CF).
\end{equation} 

\medskip

We have a monadic adjunction
$$\ind_{\on{weak-Weil}}:\Shv(X)\rightleftarrows \Shv^{\on{weak-Weil}}(X):\oblv_{\on{weak-Weil}},$$
where $\ind_{\on{weak-Weil}}$ sends $\CF$ to
$$\underset{n\geq 0}\oplus\, (\Frob^n)_*(\CF),$$
equipped with the natural weak Weil structure.

\sssec{} \label{sss:weak Weil t}

Since the above monad on $\Shv(X)$ is t-exact, we obtain that $\Shv^{\on{weak-Weil}}(X)$ carries
a unique t-structure, so that both functors $\ind_{\on{weak-Weil}}$ and $\oblv_{\on{weak-Weil}}$ are t-exact. 

\medskip

It follows formally that the category $\Shv^{\on{weak-Weil}}(X)$ is also left-complete in its t-structure.

\sssec{}

Let 
$$\Shv^{\on{Weil}}(X)\subset \Shv^{\on{weak-Weil}}(X)$$
be the full subcategory consisting of objects $(\CF,\alpha)$, for which $\alpha$ is an isomorphism. 
Since $\Frob$ is an equivalence, this is the same as requiring that \eqref{e:other alpha} be an isomorphism.

\medskip

We have a localization
$$\jmath^*:\Shv^{\on{weak-Weil}}(X)\rightleftarrows \Shv^{\on{Weil}}(X):\jmath_*,$$
where $\jmath^*$ sends 
$$(\CF,\alpha)\mapsto \underset{n}{\on{colim}}\, (\Frob^n)_*(\CF),$$
with the transition maps induced by $\alpha$, and where the colimit is equipped with a natural
Frobenius-equivariant structure. 

\medskip

Denote 
$$\ind_{\on{Weil}}=\jmath^*\circ \ind_{\on{weak-Weil}} \text{ and } \oblv_{\on{Weil}}:=\oblv_{\on{weak-Weil}}\circ \jmath_*.$$

The same logic as in \secref{sss:weak Weil t} applies, and we obtain that $\Shv^{\on{Weil}}(X)$ carries a unique t-structure,
for which the functors $\ind_{\on{Weil}}$ and $\oblv_{\on{Weil}}$ are t-exact. Moreover, $\Shv^{\on{Weil}}(X)$ is left-complete 
in its t-structure. 

\sssec{}

Set
$$\Shv^{\on{Weil}}(X)^{\on{constr}}:=\Shv^{\on{Weil}}(X)\underset{\Shv(X)}\times \Shv(X)^{\on{constr}}.$$

Define
$$\Shv^{\on{Weil,loc.fin}}(X):=\on{Ind}(\Shv^{\on{Weil}}(X)^{\on{constr}}).$$

\medskip

We have a tautologically defined functor:

\begin{equation} \label{e:Weil loc fin}
\Shv^{\on{Weil,loc.fin}}(X)\to \Shv^{\on{Weil}}(X).
\end{equation}

By a slight abuse of notation, we will denote by the same symbol $\oblv_{\on{Weil}}$ 
the resulting functor
$$\Shv^{\on{Weil,loc.fin}}(X)\to \Shv(X).$$

\sssec{}

We claim:

\begin{prop} \label{p:Weil loc fin} \hfill

\smallskip

\noindent{\em(a)}
The functor \eqref{e:Weil loc fin} preserves compactness;

\smallskip

\noindent{\em(b)}
The functor \eqref{e:Weil loc fin} is fully faithful;

\end{prop} 

\begin{proof}

To prove point (a), it suffices to show that if an object $(\CF,\alpha)\in \Shv^{\on{Weil}}(X)$ 
is such that $\CF\in \Shv(X)$ is compact (i.e., lies in $\Shv(X)^{\on{constr}}$), then this
object is compact. Since $(\jmath^*,\jmath_*)$ is a localization, it suffices to prove the same assertion for 
 $\Shv^{\on{weak-Weil}}(X)$. 

\medskip

However, this follows from the fact that $(\CF,\alpha)$ can be written as
$$\on{cofib}(\ind_{\on{weak-Weil}}(\Frob_*(\CF))\to \ind_{\on{weak-Weil}}(\CF)),$$
for the map obtained by the $(\ind_{\on{weak-Weil}},\oblv_{\on{weak-Weil}})$-adjunction from the map
$$\Frob_*(\CF)\to \underset{n\geq 0}\oplus\, (\Frob^n)_*(\CF)=\oblv_{\on{weak-Weil}}\circ \ind_{\on{weak-Weil}}(\CF).$$

\medskip

Point (b) follows formally from point (a), since the functor \eqref{e:Weil loc fin} is by definition fully faithful on compacts.

\end{proof}

\sssec{}

The inclusion
$$\Shv^{\on{Weil,loc.fin}}(X)^c=\Shv^{\on{Weil}}(X)^{\on{constr}}\hookrightarrow \Shv^{\on{Weil}}(X)$$
is stable under the truncations with respect to the t-structure. Hence $\Shv^{\on{Weil,loc.fin}}(X)^c$ inherits
a t-structure, which in turn induces a t-structure on $\Shv^{\on{Weil,loc.fin}}(X)$.

\medskip

By construction, the functor \eqref{e:Weil loc fin} is t-exact, and hence so is the functor 
$$\oblv_{\on{Weil}}:\Shv^{\on{Weil,loc.fin}}(X)\to \Shv(X).$$

Since the latter functor is conservative, an object of $\Shv^{\on{Weil,loc.fin}}(X)$ is connective/coconnective
if and only if its image under $\oblv_{\on{Weil}}$ is connective/coconnective. 

\sssec{} \label{sss:weil shv wt}

We now introduce two more categories that we will use later. 

\medskip

One is:
$$\Shv^{\on{Weil,wt}}(X):=\on{Ind}(\Shv^{\on{Weil,wt}}(X)^{\on{constr}}),$$
where 
$$\Shv^{\on{Weil,wt}}(X)^{\on{constr}}\subset \Shv^{\on{Weil}}(X)^{\on{constr}}$$
is a full subcategory that consists of objects, whose 
perverse cohomologies have irreducible subquotients that are pure of some integral weight.

\medskip

Note that for $X_0=\on{pt}$, the category $\Shv^{\on{Weil,wt}}(X)$ is by definition $\Rep(\BZ^{\on{alg,wt}})$. 

\medskip

By construction, the tautological functor
$$\Shv^{\on{Weil,wt}}(X)\to \Shv^{\on{Weil,loc.fin}}(X)$$
preserves compactness and hence is fully faithful. 

\sssec{}

The other is:
$$\Shv^{\on{Weil},0}(X):=\on{Ind}(\Shv^{\on{Weil},0}(X)^{\on{constr}}),$$
where 
$$\Shv^{\on{Weil},0}(X)^{\on{constr}}\subset \Shv^{\on{Weil}}(X)^{\on{constr}}$$
is a full subcategory that consists of objects, whose 
perverse cohomologies have irreducible subquotients that are pure of weight $0$.

\medskip

Note that for $X_0=\on{pt}$, the category $\Shv^{\on{Weil},0}(X)$ is by definition $\Rep(\BZ^{\on{alg},0})$. 

\medskip

By construction, the tautological functor
$$\Shv^{\on{Weil},0}(X)\to \Shv^{\on{Weil,loc.fin}}(X)$$
preserves compactness and hence is fully faithful. 

\sssec{}

The t-structure on $\Shv^{\on{Weil,loc.fin}}(X)^c$ induces t-structures on the (small) categories $$\Shv^{\on{Weil,wt}}(X)^{\on{constr}} 
\text{ and } \Shv^{\on{Weil},0}(X)^{\on{constr}},$$ which in turn induce t-structures on 
$$\Shv^{\on{Weil,wt}}(X) \text{ and } \Shv^{\on{Weil},0}(X),$$
respectively. 

\medskip

By construction, the embeddings
$$\Shv^{\on{Weil},0}(X)\hookrightarrow \Shv^{\on{Weil,wt}}(X) \hookrightarrow \Shv^{\on{Weil,loc.fin}}(X)$$
are t-exact. 

\ssec{The category of relevant sheaves} \label{ss:relev}

\sssec{}

In the setting of \secref{sss:ground}, take $k=\BF_q$. We will consider several sheaf theories associated to this situation.
In the notations of \secref{sss:finite field}, each sends $X_0\in \Sch_{/\BF_q}$ to the following category: 

\medskip

\begin{enumerate}

\item $\Shv(X)$; 

\medskip

\item $\Shv^{\on{Weil,loc.fin}}(X)$;

\medskip

\item $\Shv^{\on{Weil,wt}}(X)$, which is the full subcategory of $\Shv^{\on{Weil,loc.fin}}(X)$ consisting of objects $\CF$, such that for every $x\in |X|$, 
the eigenvalues of $\Frob_x$ acting on the *-stalks of irreducible subquotients of perverse cohomologies of $\CF$ are $q_x$-Weil numbers.

\end{enumerate} 

\medskip

Evaluating on $X_0=\on{pt}$, the above examples produce the categories
$$\Vect_{\sfe},\,\, \Rep(\BZ^{\on{alg}}) \text{ and } \Rep(\BZ^{\on{alg,wt}}),$$
respectively. 

\begin{rem} 

The above definition of $\Shv^{\on{Weil,wt}}(X)$ is a variant of the (usual) category $\Shv^{\on{Weil,mxd}}(X)$, the difference being
that we require the eigenvalues to be Weil numbers as opposed to mock Weil numbers (see \secref{sss:Weil num} for what this 
means). 

\medskip

The fact that $\Shv^{\on{Weil,mxd}}(X)$ is a sheaf theory is a consequence of Weil-II (specifically, Theorems 3.3.1 and 6.1.2
in \cite{De2}). 

\medskip

The fact that $\Shv^{\on{Weil,wt}}(X)$ is a sheaf theory is obtained by combining \cite[Theorem 5.2.2]{De1}
(stability under !-puhforwards) and the ``de finitude" trick of \cite[Sects. 6.1.4-6.1.10]{De2} (stability under *-puhforwards). 

\medskip 

That said, the categories $\Shv^{\on{Weil,mxd}}(X)$ and $\Shv^{\on{Weil,wt}}(X)$ are expressible through one another.
Namely, the proof of \thmref{t:LL} below shows that the natural functor
$$\Rep(\BZ^{\on{alg,mxd}})\underset{\Rep(\BZ^{\on{alg,wt}})}\otimes \Shv^{\on{Weil,wt}}(X)\to \Shv^{\on{Weil,mxd}}(X)$$
is an equivalence, where $\BZ^{\on{alg,mxd}}$ is a variant of $\Rep(\BZ^{\on{alg,wt}})$, where we replace ``Weil" by ``mock Weil".

\end{rem}

\sssec{}

Note also that we have the natural transformations
$$\Shv(-)^{\on{Weil,wt}}\to \Shv(-)^{\on{Weil,loc.fin}} \to \Shv(-)$$
satisfying the assumptions of \propref{p:abstract base change}.  

\medskip

Hence, we obtain the fully faithful functors
\begin{equation} \label{e:tensor up 1}
\Rep(\BZ^{\on{alg}}) \underset{\Rep(\BZ^{\on{alg,wt}})}\otimes\Shv^{\on{Weil,wt}}(X)\to \Shv^{\on{Weil,loc.fin}}(X),
\end{equation} 
\begin{equation} \label{e:tensor up 2}
\Vect_{\sfe}\underset{\Rep(\BZ^{\on{alg}})}\otimes\Shv^{\on{Weil,loc.fin}}(X)\to \Shv(X),
\end{equation} 
and 
\begin{equation} \label{e:tensor up 3}
\Vect_{\sfe}\underset{\Rep(\BZ^{\on{alg,wt}})}\otimes\Shv^{\on{Weil,wt}}(X)\to \Shv(X).
\end{equation} 

\sssec{}

Let 
$$\Shv^{\on{relev}}(X)\subset \Shv(X)$$
denote the essential image of the functor \eqref{e:tensor up 2}.

\medskip

We claim:

\begin{lem} \label{l:rel gener}
The category $\Shv^{\on{relev}}(X)$ is compactly generated by irreducible
objects $\CF\in  \Perv(X)$ for which there exists an integer $n$ such that $(\Frob^n)^*(\CF)\simeq \CF$.
\end{lem}

\begin{proof}

By construction, $\Shv^{\on{relev}}(X)$ is generated by images under $\oblv_{\on{Weil}}$ of irreducible
objects in $\Shv^{\on{Weil,loc.fin}}(X)^\heartsuit$, i.e., irreducible Weil perverse sheaves. It is easy to
see that any such is of the form 
\begin{equation} \label{e:power Frob}
\CF\oplus \Frob^*(\CF)\oplus...\oplus (\Frob^{n-1})^*(\CF),
\end{equation} 
where $\CF$ is an irreducible object in $\Perv(X)$ such that $(\Frob^n)^*(\CF)\simeq \CF$,
while $(\Frob^m)^*(\CF)\not\simeq \CF$ for $m<n$. 

\end{proof}

\begin{cor} \label{c:t relev}  \hfill

\smallskip

\noindent{\em(a)}
The inclusion $\Shv^{\on{relev}}(X)\hookrightarrow \Shv(X)$ is compatible with the t-structure on $\Shv(X)$.

\smallskip

\noindent{\em(b)}
The induced t-structure on $\Shv^{\on{relev}}(X)$ preserves the subcategory of compact objects. 

\smallskip

\noindent{\em(c)} The irreducible objects in $\Shv^{\on{relev}}(X)^\heartsuit$ are those described in
\lemref{l:rel gener}.

\end{cor} 

\sssec{}

We claim:

\begin{thm} \label{t:LL}
The functor \eqref{e:tensor up 1} is an equivalence.
\end{thm} 

\begin{proof}

We need to show that the essential image of the functor \eqref{e:tensor up 1} generates the target
category. Hence, it suffices to show that every object $\CF\in (\Shv^{\on{Weil,loc.fin}}(X))^\heartsuit$ 
lies in the essential image of \eqref{e:tensor up 1}. 

\medskip

We can assume that $\CF$ is
irreducible. We will show that in this case $\CF$ has the form
$$\CF^0\otimes \fl,$$
where $\CF^0\in (\Shv^{\on{Weil,wt}}(X))^\heartsuit$ and $\fl$ is a line in $\Shv^{\on{Weil}}(\on{pt})$.

\medskip

Any irreducible object of $(\Shv^{\on{Weil,loc.fin}}(X))^\heartsuit$ is a Goresky-MacPherson extension
of an irreducible \emph{lisse} object on a locally-closed smooth subscheme of $X$. Since the operation of Goresky-MacPherson 
extension preserves weights, this allows us to assume that $\CF$ is lisse. 

\medskip

Now, the desired result follows from \cite[Corollary VII.8]{LLaf}\footnote{With the correction in \cite[Sect. 1.7-1.9]{De3}.}.

\end{proof}

\begin{cor} \label{c:relev}
The essential image of the functor \eqref{e:tensor up 3} equals $\Shv^{\on{relev}}(X)$.
\end{cor}

\sssec{}

From Corollaries \ref{c:relev} and \ref{c:eq de-eq W} we obtain:

\begin{cor} \label{c:W acts on rel}
The category $\Shv^{\on{relev}}(X)$ upgrades to an object of $\QCoh(\BZ^{\on{alg,wt}})\commod$ so that
$$(\Shv^{\on{relev}}(X))^{\BZ^{\on{alg,wt}}}\simeq \Shv^{\on{Weil,wt}}(X),$$
as objects of $\Rep(\BZ^{\on{alg,wt}})\mmod$.
\end{cor} 

\ssec{Compatibility of t-structures}

\sssec{} \label{sss:ten prod t-structure}

Let $\bA$ be a monoidal category equipped with a t-structure so that the monoidal operation is right t-exact,
and $\one_\bA$ is connective. Let $\CM_1$ and $\CM_2$ be right and left $\bA$-module categories each equipped with a t-structure,
and such that the action functors are right t-exact. 

\medskip

In this case we equip 
$$\CM_1\underset{\bA}\otimes \CM_2$$ 
with a t-structure by declaring that the connective subcategory is generated under colimits
by the image of connective objects under
$$\CM_1\otimes \CM_2\to \CM_1\underset{\bA}\otimes \CM_2.$$ 

We call the resulting t-structure on $\CM_1\underset{\bA}\otimes \CM_2$ the \emph{tensor product} t-structure. 

\sssec{}

We will now consider a paradigm in which the tensor product t-structure is particularly well-behaved.
Let $\bA$ and $\CM\in \bA\mmod$ be as above. Let $\bA'$ be another monoidal category equipped with a 
t-structure as above. Let $\Phi:\bA\to \bA'$ be a right t-exact monoidal functor. We equip 
$$\CM':=\bA'\underset{\bA}\otimes \CM$$
with a t-structure by the above recipe. 

\medskip

Assume that $\Phi$ admits a continuous right adjoint. Assume that:

\begin{itemize}

\item The natural structure on $\Phi^R$ of right-lax compatibility with the $\bA$-bimodule structure is strict;

\item The functor $\Phi^R$ is conservative;

\item The functor $\Phi^R$ is t-exact; 

\item The action of $\Phi^R(\one_{\bA'})$ on $\CM$ is t-exact. 

\end{itemize}

Under these circumstances, we claim:

\begin{lem} \label{l:good t}  
Both functors in the adjunction
\begin{equation} \label{e:good t}
\CM\simeq \bA\underset{\bA}\otimes \CM \overset{\Phi\otimes \on{Id}}{\underset{\Phi^R\otimes \on{Id}}{\rightleftarrows}} \bA'\underset{\bA}\otimes \CM=:\CM'
\end{equation} 
are t-exact, and the right adjoint is conservative. 
\end{lem}

\begin{proof}

The fact that $\Phi^R$ is conservative is equivalent to the fact that the essential image of $\Phi$ generates the target category,
which implies that the essential image of $\Phi\otimes \on{Id}$ generates the target category, which in turn implies that
$\Phi^R\otimes \on{Id}$ is conservative. 

\medskip

The left adjoint in \eqref{e:good t} is right t-exact by construction. Hence, the right adjoint in \eqref{e:good t} is left t-exact
(by adjunction). However, it is also right t-exact, since $\Phi^R$ was assumed to be right t-exact. 

\medskip

Under these circumstances,
the left adjoint is t-exact if and only if the resulting monad is t-exact. However, this monad given by the action of $\Phi^R(\one_{\bA'})$ on $\CM$.
Since this action it t-exact, this implies the assertion of the lemma. 

\end{proof} 

\sssec{}

We apply \lemref{l:good t} to 
$$\bA:=\Rep(\BZ^{\on{alg}}) ,\,\, \bA':=\Vect_{\sfe} \text{ and } \CM:=\Shv^{\on{Weil,loc.fin}}(X).$$
  
We obtain that the category
$$\Vect_{\sfe}\underset{\Rep(\BZ^{\on{alg}})}\otimes\Shv^{\on{Weil,loc.fin}}(X)$$
acquires a t-structure, so that the adjoint functors
$$\Shv^{\on{Weil,loc.fin}}(X)\rightleftarrows \Vect_{\sfe}\underset{\Rep(\BZ^{\on{alg}})}\otimes\Shv^{\on{Weil,loc.fin}}(X)$$
are both t-exact. 

\sssec{}

Recall now (see \corref{c:t relev}) that the category $\Shv^{\on{relev}}(X)$ carries a t-structure, induced by that on $\Shv(X)$. 
We claim:

\begin{prop} \label{p:t on relev}
The equivalence
$$\Vect_{\sfe}\underset{\Rep(\BZ^{\on{alg}})}\otimes\Shv^{\on{Weil,loc.fin}}(X) \overset{\sim}\to \Shv^{\on{relev}}(X)$$
is t-exact.
\end{prop}

\begin{proof}

The fact that the functor $\to$ is right t-exact is immediate. To show that it is t-exact, it is sufficient to
show that its inverse is right t-exact. 

\medskip

Since the t-structure on $\Shv^{\on{relev}}(X)$ is obtained by extending the t-structure on 
$\Shv^{\on{relev}}(X)^c$, it suffices to show that the inverse functor sends irreducible objects
in $\Shv^{\on{relev}}(X)^\heartsuit$ to objects in the heart in the left-hand side. 

\medskip

According to \corref{c:t relev}, irreducible objects in $\Shv^{\on{relev}}(X)^\heartsuit$ are irrreducible
objects in $\Shv(X)$ that are invariant under some power of the Frobenius. Such an object is direct
summand of an object \eqref{e:power Frob}, so it is enough to show that the latter object gets sent to
an object in the heart in the left-hand side. However, \eqref{e:power Frob} is the image of an object 
in $\Shv^{\on{Weil,loc.fin}}(X)^\heartsuit$ under
$$\Shv^{\on{Weil,loc.fin}}(X)\to \Vect_{\sfe}\underset{\Rep(\BZ^{\on{alg}})}\otimes\Shv^{\on{Weil,loc.fin}}(X)\to \Shv(X),$$
while the composite functor is t-exact. 
\end{proof} 

\begin{rem}

We can also apply \lemref{l:good t} to 
$$\bA:=\Rep(\BZ^{\on{alg,wt}}),\,\, \bA':=\Rep(\BZ^{\on{alg}}),\,\, \CM:=\Shv^{\on{Weil,wt}}(X),$$
and obtain a t-structure on
$$\Rep(\BZ^{\on{alg}})\underset{\Rep(\BZ^{\on{alg,wt}})}\otimes\Shv^{\on{Weil,wt}}(X).$$

As in \propref{p:t on relev}, one shows that the equivalence \eqref{e:tensor up 1} is t-exact. 

\end{rem} 

\begin{rem}

Furthermore, we can take
$$\bA:=\Rep(\BZ^{\on{alg,wt}}),\,\,\bA':=\Vect_{\sfe} \text{ and } \CM:=\Shv^{\on{Weil,wt}}(X),$$
and produce a t-structure on 
$$\Vect_{\sfe}\underset{\Rep(\BZ^{\on{alg,wt}})}\otimes\Shv^{\on{Weil,wt}}(X).$$

It follows, however, that this is the same t-structure as the one arising from the equivalence
$$\Vect_{\sfe}\underset{\Rep(\BZ^{\on{alg,wt}})}\otimes\Shv^{\on{Weil,wt}}(X)\simeq 
\Shv^{\on{relev}}(X)\simeq \Vect_{\sfe}\underset{\Rep(\BZ^{\on{alg}})}\otimes\Shv^{\on{Weil,loc.fin}}(X).$$

\end{rem} 

\sssec{}

We now claim:

\begin{thm} \label{t:t on relev} \hfill

\smallskip

\noindent{\em(a)} The category $\Shv^{\on{relev}}(X)$ is left-complete with respect to its t-structure.

\smallskip

\noindent{\em(b)} The category $\Shv^{\on{Weil,loc.fin}}(X)$ is left-complete with respect to its t-structure.

\end{thm}

\begin{proof}

Point (a) of the theorem follows by repeating the argument of \cite[Theorem 1.1.6]{AGKRRV}\footnote{Note that Beilinson's
theorem, used in \cite[Sect. E.1]{AGKRRV}, is equally applicable to $\Shv^{\on{relev}}(X)$.}.

\medskip

We now prove point (b). Note that by \corref{c:eq de-eq W}, we can regard $\Shv^{\on{relev}}(X)$ as a category acted on by
$\BZ^{\on{alg}}$ so that
$$\Shv^{\on{Weil,loc.fin}}(X)\simeq (\Shv^{\on{relev}}(X))^{\BZ^{\on{alg}}}.$$

In other words, $\Shv^{\on{Weil,loc.fin}}(X)$ is the totalization of the cosimplicial category with terms
$$\Shv^{\on{relev}}(X)\otimes \QCoh(\BZ^{\on{alg}})^{\otimes n}$$
and t-exact face maps. Since all terms are left-complete (e.g., by \cite[Lemma 2.1.3(a)]{AGKRRV}), we obtain that 
$\Shv^{\on{Weil,loc.fin}}(X)$ is also left-complete.

\end{proof} 

\sssec{}

From point (a) of \thmref{t:t on relev}, we obtain:

\begin{cor} \label{c:char relev}
An object of $\Shv(X)$ belongs to $\Shv^{\on{relev}}(X)$ if and only if all irreducible subquotients of its perverse cohomologies
are invariant under some power of the Frobenius. 
\end{cor} 

\begin{rem} \label{r:wt left complete}
An assertion parallel to \thmref{t:t on relev}(b) holds for $\Shv^{\on{Weil,wt}}(X)$: in the proof one uses 
$\BZ^{\on{alg,wt}}$ instead of $\BZ^{\on{alg}}$ and \corref{c:W acts on rel}. 
\end{rem}

\ssec{A semi-simple category of sheaves}

\sssec{}

Denote
$$\Shv^{\on{relev},0}(X):=\Vect_{\sfe}\underset{\Rep(\BZ^{\on{alg},0})}\otimes\Shv^{\on{Weil},0}(X).$$

We equip $\Shv^{\on{relev},0}(X)$ with a t-structure by the recipe of \secref{sss:ten prod t-structure}.

\sssec{}

Note that we have a naturally defined forgetful functor
\begin{equation}  \label{e:zero to all}
\phi:\Shv^{\on{relev},0}(X)\to \Shv^{\on{relev}}(X),
\end{equation} 
but it is \emph{not} an equivalence. 

\sssec{}

Recall that a DG category $\bC$ equipped with a t-structure is said to be \emph{semi-simple} if:

\begin{itemize}

\item $\bC^\heartsuit$ is a semi-simple abelian category;

\item The functor $D(\bC^\heartsuit) \to \bC$ is an equivalence.

\end{itemize} 

\sssec{}

We claim:

\begin{prop} \label{p:relev 0} \hfill

\smallskip

\noindent{\em(a)} The category $\Shv^{\on{relev},0}(X)$ is semi-simple.  

\smallskip

\noindent{\em(b)} Its irreducible objects are the direct summands of 
images under the tautological functor
$$\Shv^{\on{Weil},0}(X)\to \Shv^{\on{relev},0}(X)$$
of objects \eqref{e:power Frob}, where $\CF$ is an irreducible perverse sheaf with a Weil structure
\emph{over} $\BF_{q^n}$ that is pure of weight $0$. Each such direct summand is of the form
$(\Frob^i)^*(\CF)$ for $0\leq i<n$. 
\end{prop}

\begin{rem} \label{r:semi-simple 0}
Note that although the functor $\phi$ of \eqref{e:zero to all} is not an equivalence, 
it induces a bijection on
semi-simple objects: indeed, it follows from \corref{c:t relev}(c) and the proof of 
\thmref{t:LL} that any irreducible perverse sheaf
in $\Shv^{\on{relev}}(X)$ lies in the image of $\phi$.

\medskip

In particular, the functor $\phi$ is t-exact. 

\end{rem} 

\sssec{}

Before we prove the proposition, we give a criterion for semi-simplicity:

\begin{lem} 
Let $\bC$ be a DG category equipped with a t-structure. Assume that $\bC$ is compactly generated
by a collection of objects $\bc_\alpha\in \bC^\heartsuit$ such that: 

\smallskip

\noindent{\em(i)} $\bc_\alpha$ are semi-simple as objects of $\bC^\heartsuit$;

\smallskip

\noindent{\em(ii)} For any $\alpha_1,\alpha_2$, the object $\CHom_\bC(\bc_{\alpha_1},\bc_{\alpha_2})$ is acyclic off degree $0$. 

\medskip

Then $\bC$ is semi-simple.

\end{lem} 

\sssec{Proof of \propref{p:relev 0}, Step 1}

It suffices to show that for a pair of irreducible objects $\CF_1,\CF_2\in (\Shv^{\on{Weil},0}(X))^\heartsuit$, the object 
\begin{equation} \label{e:relev 0 1}
\CHom_{\Shv^{\on{relev},0}(X)}(\one_{\Vect_{\sfe}}\otimes \CF_1,\one_{\Vect_{\sfe}}\otimes \CF_2)\in \Vect_{\sfe}
\end{equation}
is acyclic off degree $0$. 

\medskip

Note that the object \eqref{e:relev 0 1} is obtained by applying the forgetful functor $\Rep(\BZ^{\on{alg},0})\to \Vect_{\sfe}$
to
\begin{equation} \label{e:relev 0 2}
\uHom_{\Shv^{\on{Weil},0}(X),\Rep(\BZ^{\on{alg},0})}(\CF_1,\CF_2)\in \Rep(\BZ^{\on{alg},0}).
\end{equation}

Hence, it suffices to show that \eqref{e:relev 0 2} lies in the heart of the t-structure. 

\sssec{Proof of \propref{p:relev 0}, Step 2}

In its turn, the object \eqref{e:relev 0 2} 
is obtained from
\begin{equation} \label{e:relev 0 3}
\uHom_{\Shv^{\on{Weil,loc.fin}}(X),\BZ\mod_{\on{loc.fin}}.}(\CF_1,\CF_2)\in \BZ\mod_{\on{loc.fin}}.
\end{equation}
by applying the right adjoint to the embedding
\begin{equation} \label{e:Weil numbers 0}
\Rep(\BZ^{\on{alg},0})\hookrightarrow \Rep(\BZ^{\on{alg}})=\BZ\mod_{\on{loc.fin}}.
\end{equation} 

\medskip

We will show that this right adjoint annihilates $H^i(\uHom_{\Shv^{\on{Weil,loc.fin}}(X),\BZ\mod_{\on{loc.fin}}.}(\CF_1,\CF_2))$
for $i>0$ and acts as identity for $i=0$. 

\sssec{Proof of \propref{p:relev 0}, Step 3}

Note that the object of $\Vect_{\sfe}$ underlying \eqref{e:relev 0 3} is
\begin{equation} \label{e:relev 0 4}
\CHom_{\Shv(X)}(\CF_1,\CF_2),
\end{equation}
with the action of $1\in \BZ$ given by the Weil structure.

\medskip

Now, by \cite{BBD}, the generalized eigenvalues of the Frobenius on
$$H^i(\CHom_{\Shv(X)}(\CF_1,\CF_2))\simeq \on{Ext}^i_{\Shv(X)}(\CF_1,\CF_2)$$
are Weil numbers of weights $\geq i$. 

\medskip

Hence, all $H^i(\CHom_{\Shv(X)}(\CF_1,\CF_2))$ with $i>0$ are annihilated by the right
adjoint of \eqref{e:Weil numbers 0}. 

\sssec{Proof of \propref{p:relev 0}, Step 4}

Finally, for $i=0$, 
\begin{equation} \label{e:relev 0 5}
H^0(\CHom_{\Shv(X)}(\CF_1,\CF_2))\simeq \Hom_{\Shv(X)}(\CF_1,\CF_2),
\end{equation}
and let us consider two cases:

\medskip

\noindent{\it Case 1:} $\oblv_{\on{Weil}}(\CF_1)$ and $\oblv_{\on{Weil}}(\CF_2)$ does not
have common isotypic components. In this case $\Hom_{\Shv(X)}(\CF_1,\CF_2)=0$
and there is nothing to prove,

\medskip 

\noindent{\it Case 2:} $\oblv_{\on{Weil}}(\CF_1)$ and $\oblv_{\on{Weil}}(\CF_2)$ do 
have common isotypic components. Then by the description of the irreducible
objects in $\Shv^{\on{Weil}}(X)$ (see the proof of \lemref{l:rel gener}), we have 
$$\CF_1\simeq \CF_2\otimes \fl,$$ where $\fl$ is a line in $\Rep(\BZ^{\on{alg},0})$, and 
$\Hom_{\Shv(X)}(\CF_1,\CF_2)$ is isomorphic to $\fl^{\oplus n}$ for some $n$.
Hence, \eqref{e:relev 0 5} already belongs to $\Rep(\BZ^{\on{alg},0})$, and hence, the right
adjoint of \eqref{e:Weil numbers 0} applied to $H^0(\CHom_{\Shv(X)}(\CF_1,\CF_2))$ 
acts as identity. 
 
\qed[\propref{p:relev 0}]

\begin{rem}

As in \thmref{t:t on relev}(b), from \propref{p:relev 0} one deduces that the category
$\Shv^{\on{Weil},0}(X)$ is left-complete in its t-structure.

\end{rem}

\ssec{A contraction on the category of sheaves}

\sssec{}

We are now ready to carry out the key construction of this section:

\begin{thmconstr} \label{t:contraction on category}
The DG category $\Shv^{\on{relev}}(X)$ canonically lifts to an object of $\QCoh(\BA^1)\commod$. Moreover:

\smallskip

\noindent{\em(i)} The \emph{attracting category} corresponding to this coaction is $\Shv^{\on{relev},0}(X)$. 

\smallskip

\noindent{\em(ii)} The resulting functors
$$\Shv^{\on{relev}}(X) \overset{\psi}\to \Shv^{\on{relev},0}(X)\overset{\phi}\to \Shv^{\on{relev}}(X)$$
are compatible with the coactions of $\QCoh(\BZ^{\on{alg,wt}})$, where the coaction 
on $\Shv^{\on{relev},0}(X)$ is via $\BZ^{\on{alg,wt}}\to \BZ^{\on{alg},0}$. 
\end{thmconstr}

\begin{rem} 
Note that as a consequence of \thmref{t:contraction on category} we obtain that the forgetful functor $\phi$
of \eqref{e:zero to all} admits a retraction, to be denoted
$$\psi:\Shv^{\on{relev}}(X)\to \Shv^{\on{relev},0}(X).$$

\medskip

In view of \propref{p:relev 0}, we obtain that on the category $\Shv^{\on{relev}}(X)$ there is a 
well-defined functor of \emph{semi-simplification}. 
\end{rem} 

\sssec{}

The rest of this subsection is devoted to the proof of \thmref{t:contraction on category}. 

\sssec{}

Recall that according to \corref{c:relev}, we have
$$\Shv^{\on{relev}}(X)\simeq \Vect_{\sfe}\underset{\Rep(\BZ^{\on{alg,wt}})}\otimes\Shv^{\on{Weil,wt}}(X).$$

Set\footnote{The reason for the choice of notation is that the role of $\Shv^{\on{relev,gr}}(X)$ vis-a-vis $\Shv^{\on{relev}}(X)$ is
similar to that of $A\mod^{\on{gr}}$ vis-a-vis $A\mod$, for a $\BZ$-graded algebra $A$.} 
$$\Shv^{\on{relev,gr}}(X):=\Rep(\BG_m)\underset{\Rep(\BZ^{\on{alg,wt}})}\otimes\Shv^{\on{Weil,wt}}(X),$$
where
$$\Rep(\BZ^{\on{alg,wt}})\to \Rep(\BG_m)$$
is given by restriction along the canonical homomorphism $\BG_m\to \BZ^{\on{alg,wt}}$
(see \secref{sss:structure W}).

\begin{rem}

The category $\Shv^{\on{relev,gr}}(X)$ has been introduced and studied in \cite{HL}.

\end{rem} 

\sssec{}

Thus, 
$$\Shv^{\on{relev}}(X)\simeq  \Vect_{\sfe}\underset{\Rep(\BG_m)}\otimes \Shv^{\on{relev,gr}}(X).$$

According to \thmref{t:A1}, in order to prove \thmref{t:contraction on category}, we need to lift
$\Shv^{\on{relev,gr}}(X)$ to an object of $\BZ\mmod^{\on{Filtr}}$, so that the corresponding category
$(\Shv^{\on{relev,gr}}(X))_0$ identifies with $\Shv^{\on{relev},0}(X)$.

\sssec{}

Note that the short exact sequence
$$1\to \BG_m\to \BZ^{\on{alg,wt}}\to \BZ^{\on{alg},0}\to 1$$
leads to an equivalence 
$$\Rep(\BG_m)\simeq \Vect_{\sfe}\underset{\Rep(\BZ^{\on{alg},0})}\otimes \Rep(\BZ^{\on{alg,wt}}).$$

Hence, we can rewrite
$$\Shv^{\on{relev,gr}}(X) \simeq  \Vect_{\sfe}\underset{\Rep(\BZ^{\on{alg},0})}\otimes \Shv^{\on{Weil,wt}}(X).$$

We will lift $\Shv^{\on{Weil,wt}}(X)$ to an object of $\DGCat^{\on{Filtr}}$ in a way compatible with the action of
$\Rep(\BZ^{\on{alg,wt}})$ and such that
$$(\Shv^{\on{Weil,wt}}(X))_0\simeq \Shv^{\on{Weil},0}(X).$$

This will induce the required structure on $\Shv^{\on{relev,gr}}(X)$.

\sssec{}

We let 
$$(\Shv^{\on{Weil,wt}}(X))_{\leq n}\subset \Shv^{\on{Weil,wt}}(X)$$
be the full subcategory generated by
$$(\Shv^{\on{Weil,wt}}(X))^{\on{constr}}_{\leq n}\subset (\Shv^{\on{Weil,wt}}(X))^{\on{constr}},$$
where the latter is the subcategory consisting of objects, all of whose perverse cohomologies
have weights $\leq n$.

\medskip

\noindent {\it Warning}: note that $(\Shv^{\on{Weil,wt}}(X))_{\leq n}$ is \emph{not} what is called 
``mixed complexes of weight $\leq n$" in \cite{BBD}. 

\sssec{}

By construction, the category $(\Shv^{\on{Weil,wt}}(X))_{\leq n}$ is compactly generated, and the embedding
$$(\Shv^{\on{Weil,wt}}(X))_{\leq n}\hookrightarrow (\Shv^{\on{Weil,wt}}(X))_{\leq n+1}$$
preserves compactness (and, hence, admits a colimit-preserving right adjoint). 

\medskip

Furthermore, it is clear that 
$$\underset{n}{\on{colim}}\, (\Shv^{\on{Weil,wt}}(X))_{\leq n}\to \Shv^{\on{Weil,wt}}(X)$$
is an equivalence.

\sssec{}

It remains to identify
$$(\Shv^{\on{Weil,wt}}(X))_{\leq 0}/(\Shv^{\on{Weil,wt}}(X))_{\leq -1}\simeq \Shv^{\on{Weil},0}(X).$$

\medskip

We will think of $\Shv^{\on{Weil},0}(X)$ as a full subcategory of $(\Shv^{\on{Weil,wt}}(X))_{\leq 0}$.
We will show that as such it identifies with the right orthogonal
$$(\Shv^{\on{Weil,wt}}(X))_{\leq -1}^\perp\subset (\Shv^{\on{Weil,wt}}(X))_{\leq 0}.$$

\sssec{}

First, we claim that for $\CF_1\in (\Shv^{\on{Weil,wt}}(X))_{\leq -1}$ and $\CF_2\in \Shv^{\on{Weil},0}(X)$,
we have
$$\CHom_{\Shv^{\on{Weil,wt}}(X)}(\CF_1,\CF_2)=0.$$

To prove this, we can assume that $\CF_1$ and $\CF_2$ are constructible and perverse. We have
$$\CHom_{\Shv^{\on{Weil,wt}}(X)}(\CF_1,\CF_2)=\CHom_{\Shv^{\on{Weil,loc.fin}}(X)}(\CF_1,\CF_2)\simeq
(\CHom_{\Shv(X)}(\CF_1,\CF_2))^{\Frob}.$$

Now, $H^i(\CHom_{\Shv(X)}(\CF_1,\CF_2))=0$ for $i<0$ and for $i\geq 0$ the Frobenius eigenvalues
on it have weights $\geq i+1$. This implies that 
$$\left(H^i(\CHom_{\Shv(X)}(\CF_1,\CF_2))\right)^{\Frob}=0,$$
as desired. 

\sssec{}

Thus, we obtain a (fully faithful) inclusion
$$\Shv^{\on{Weil},0}(X)\subset (\Shv^{\on{Weil,wt}}(X))_{\leq -1}^\perp.$$

It remains to show that the essential image of this inclusion generates the target. This is equivalent to showing
that the essential image of
$$\Shv^{\on{Weil},0}(X)\hookrightarrow (\Shv^{\on{Weil,wt}}(X))_{\leq 0}\to (\Shv^{\on{Weil,wt}}(X))_{\leq 0}/(\Shv^{\on{Weil,wt}}(X))_{\leq -1}$$
generates the target. 

\medskip

However, this is obvious: the right-hand side is generated by perverse sheaves of weights $\leq 0$, while perverse sheaves of weights
$\leq -1$ go to zero.

\qed[\thmref{t:contraction on category}]

\begin{rem}

Recall that the category $\Shv^{\on{Weil,wt}}(X)$ is left-complete in its t-structure (see Remark \ref{r:wt left complete}). 
We claim that an object of $\Shv^{\on{Weil,wt}}(X)$ belongs to $(\Shv^{\on{Weil,wt}}(X))_{\leq n}$ if (and only if) all of its
perverse cohomologies do (in particular, $(\Shv^{\on{Weil,wt}}(X))_{\leq n}$ is also left-complete in its t-structure). 

\medskip

Indeed, note that $(\Shv^{\on{Weil,wt}}(X))_{>n}$ also acquires a t-structure, such that the projection
$$(j_{>n})^L:\Shv^{\on{Weil,wt}}(X)\to (\Shv^{\on{Weil,wt}}(X))_{>n}$$
is t-exact. 

\medskip

Let $\CF\in \Shv^{\on{Weil,wt}}(X)$ be such that all of its perverse cohomologies have weights $\leq n$. 
We wish to show that $\CF\in (\Shv^{\on{Weil,wt}}(X))_{\leq n}$, i.e., that $(j_{>n})^L(\CF)=0$. 
Now, by assumption all of the cohomologies of $(j_{>n})^L(\CF)$ vanish. Hence, it suffices to show
that the t-structure on $(\Shv^{\on{Weil,wt}}(X))_{>n}$ is separated. We claim that $(\Shv^{\on{Weil,wt}}(X))_{>n}$
is in fact left-complete in its t-structure:

\medskip

Note that the functor 
$$j_{>n}: (\Shv^{\on{Weil,wt}}(X))_{>n}\to \Shv^{\on{Weil,wt}}(X)$$
is also t-exact. Since $j_{>n}$ commutes with limits, it follows formally that $\Shv^{\on{Weil,wt}}(X))_{>n}$ is 
left-complete in its t-structure.

\end{rem} 

\begin{rem}

Applying \lemref{l:good t} to 
$$\bA:=\Rep(\BZ^{\on{alg},0}),\,\, \bA':=\Vect_{\sfe} \text{ and } \CM:=\Shv^{\on{Weil,wt}}(X)$$
we obtain that the category $\Shv^{\on{relev,gr}}(X)$ also acquires a t-structure so that the adjoint
functors
$$\Shv^{\on{Weil,wt}}(X)\rightleftarrows \Shv^{\on{relev,gr}}(X)$$
are t-exact. Since $\Shv^{\on{Weil,wt}}(X)$ is left-complete in its t-structure, the same is true
for $\Shv^{\on{relev,gr}}(X)$.

\medskip

Furthermore, the categories
$$(\Shv^{\on{relev,gr}}(X))_{\leq n}\simeq  \Vect_{\sfe}\underset{\Rep(\BZ^{\on{alg},0})}\otimes (\Shv^{\on{Weil,wt}}(X))_{\leq n}$$
and 
$$(\Shv^{\on{relev,gr}}(X))_{>n}\simeq  \Vect_{\sfe}\underset{\Rep(\BZ^{\on{alg},0})}\otimes (\Shv^{\on{Weil,wt}}(X))_{>n}$$
also acquire t-structures, in which they are left-complete, and the functors $(i_{\leq n},i_{\leq n}^R)$ and $(j_{\geq n}^L,j_{\geq n})$
are t-exact. 

\end{rem} 

\ssec{The t-exactness property of the contraction}

\sssec{}

We claim:

\begin{prop}
The coaction functor
$$\Shv^{\on{relev}}(X)\to \Shv^{\on{relev}}(X)\otimes \QCoh(\BA^1)$$
is t-exact.
\end{prop} 

\begin{proof}

By \lemref{l:good t}, it suffices to show that the coaction functor
$$\Shv^{\on{Weil,wt}}(X)\to \Shv^{\on{Weil,wt}}(X)\otimes \QCoh(\BA^1/\BG_m),$$
is t-exact.

\medskip

It suffices to show that the above functor sends objects in the heart to objects in the heart. 
By construction, this functor send $\CF\in \Shv^{\on{Weil,wt}}(X)$ to the object in 
$$\Shv^{\on{Weil,wt}}(X)\otimes \QCoh(\BA^1/\BG_m)\simeq \on{Filtr}(\Shv^{\on{Weil,wt}}(X))=
\on{Funct}(\BZ,\Shv^{\on{Weil,wt}}(X))$$
given by
$$n\mapsto \CF_n:=i_{\leq n}\circ (i_{\leq n})^R(\CF).$$

However, the latter is the canonical weight filtration on $\CF$.

\end{proof} 

\sssec{}

We now claim:

\begin{prop} \label{p:psi t-exact}
The contraction functor 
$$\psi:\Shv^{\on{relev}}(X)\to \Shv^{\on{relev},0}(X)$$
is t-exact.
\end{prop}

\begin{proof}

It suffices to show that the functor
$$\psi_0:\Shv^{\on{relev,gr}}(X)\to \Shv^{\on{relev},0}(X)$$
is t-exact.

\medskip

Using \lemref{l:good t}, this is equivalent to showing that the corresponding functor
$$\psi_0:\Shv^{\on{Weil,wt}}(X)\to \Shv^{\on{Weil},0}(X)$$
is t-exact.

\medskip

To prove this, it suffices to show that the functor $\psi_0$ sends irreducible perverse sheaves in $\Shv^{\on{Weil,wt}}(X)$
to objects in the heart. 

\medskip

Let $\CF\in \Shv^{\on{Weil,wt}}(X)$ be pure of weight $n$. If $n>0$, then it is annihilated by the functor $(i_{\leq 0})^R$
and hence by $\psi_0$. 

\medskip

If $n<0$, then the functor $(i_{\leq 0})^R$ sends it to itself, but it is annihilated by the projection
$$\Shv^{\on{Weil,wt}}(X)_{\leq 0}\to \Shv^{\on{Weil,wt}}(X)_{\leq 0}/\Shv^{\on{Weil,wt}}(X)_{\leq -1}.$$

\medskip

Finally, when $n=0$, then $\psi_0(\CF)=\CF$.

\end{proof}

\section{A contraction on the stack of local systems} \label{s:LS}

In this section we will apply Theorem-Construction \ref{t:contraction on category} to obtain a contraction of the 
stack of arithmetic local systems onto the locus, where the underlying geometric local system is semi-simple. 

\medskip

This contraction will allow us to prove our main result about the excursion algebra, \thmref{t:exc}. 

\ssec{Weil sheaves with prescribed singular support}

\sssec{}

Let $\CN\subset T^*(X)$ be a conical Zariski-closed subset. Let
$$\Shv_\CN(X)\subset \Shv(X)$$
be the full subcategory defined as in \cite[Sect. E.4]{AGKRRV}.

\medskip

Let 
\begin{equation} \label{e:Weil sing supp 1}
\Shv^{\on{Weil}}_\CN(X)\subset \Shv^{\on{Weil}}(X),\,\, 
\Shv^{\on{Weil,loc.fin}}_\CN(X)\subset \Shv^{\on{Weil,loc.fin}}(X),
\end{equation} 

\begin{equation} \label{e:Weil sing supp 2}
\Shv^{\on{Weil,wt}}_\CN(X)\subset \Shv^{\on{Weil,wt}}(X),\,\, 
\Shv^{\on{Weil,0}}_\CN(X)\subset \Shv^{\on{Weil,0}}(X)
\end{equation} 
be the full subcategories equal to the preimage(s) of $\Shv_\CN(X)$ under the respective forgetful functors to $\Shv(X)$.

\medskip

Note that the above subcategories are submodule categories for the actions of
$$\BZ\mod,\,\, \BZ\mod_{\on{loc.fin}}=\Rep(\BZ^{\on{alg}}),\,\, \BZ\mod_{\on{wt}}=\Rep(\BZ^{\on{alg,wt}}) \text{ and }
\BZ\mod_{0}=\Rep(\BZ^{\on{alg},0}),$$
respectively. 

\sssec{}

Set also
$$\Shv^{\on{relev}}_\CN(X):=\Shv^{\on{relev}}(X)\underset{\Shv(X)}\times \Shv_\CN(X)$$
and
$$\Shv^{\on{relev},0}_\CN(X):=\Shv^{\on{relev},0}(X)\underset{\Shv(X)}\times \Shv_\CN(X).$$

\medskip

Recall that we have
$$\Shv^{\on{Weil,loc.fin}}(X)\simeq (\Shv^{\on{relev}}(X))^{\BZ^{\on{alg}}},\,\,
\Shv^{\on{Weil,wt}}(X)\simeq (\Shv^{\on{relev}}(X))^{\BZ^{\on{alg,wt}}}$$
and also
$$\Shv^{\on{Weil,0}}(X)\simeq (\Shv^{\on{relev},0}(X))^{\BZ^{\on{alg},0}}.$$

\sssec{}

We claim:

\begin{lem} \label{l:act on sing supp}
The subcategories
\begin{equation} \label{e:relev N}
\Shv^{\on{relev}}_\CN(X) \subset \Shv^{\on{relev}}(X) \text{ and } \Shv^{\on{relev},0}_\CN(X) \subset \Shv^{\on{relev},0}(X)
\end{equation} 
are stable under the coactions of 
$$\QCoh(\BZ^{\on{alg}}),\,\, \QCoh(\BZ^{\on{alg,wt}}) \text{ and } \QCoh(\BZ^{\on{alg},0}),$$
respectively. Moreover, we have the equalities:  
$$(\Shv^{\on{relev}}_\CN(X))^{\BZ^{\on{alg}}}=\Shv^{\on{Weil,loc.fin}}_\CN(X),\,\, (\Shv^{\on{relev}}_\CN(X))^{\BZ^{\on{alg,wt}}}=\Shv^{\on{Weil,wt}}_\CN(X)$$
and 
$$(\Shv^{\on{relev},0}_\CN(X))^{\BZ^{\on{alg},0}}=\Shv_\CN^{\on{Weil,0}}(X)$$
as full subcategories of
$$\Shv^{\on{Weil,loc.fin}}(X),\,\, \Shv^{\on{Weil,wt}}(X) \text{ and } \Shv^{\on{Weil,0}}(X),$$
respectively.
\end{lem} 

\begin{proof}

Consider the following general paradigm:

\medskip

Let $\BH$ be as in \corref{c:eq de-eq W}. Let $\bC$ be an object of $\QCoh(\BH)\commod$, and set
$$\bD:=\bC^\BH\in \Rep(\BH)\mmod.$$ Consider the adjoint pair
$$\oblv_\BH:\bD\rightleftarrows \bC:\on{Av}^\BH.$$

Let $\bC_1\subset \bC$ be a full subcategory, stable under the comonad $\oblv_\BH\circ \on{Av}^\BH$.
Set
$$\bD_1:=\bD\underset{\bC}\times \bC_1,$$

Then $\bC_1$ is a $\QCoh(\BH)$-comodule and 
$$\bC_1^\BH=\bD_1.$$

\medskip

We claim that the above paradigm is applicable in the situation of the lemma. I.e., we have to show that the
subcategories \eqref{e:relev N} are stable under the corresponding comonad. By definition, the question reduces
to one about irreducible objects in the heart of the t-structure. 

\medskip

Such an object is a direct summand of an object \eqref{e:power Frob}, which lies in the image of
$$\oblv_\BH:\bD\to \bC,$$
and an action of the comonad on it is given by tensoring with the algebra of regular functions on $\BH$.
The latter operation tautologically preserves the singular support condition.

\end{proof} 

\sssec{} \label{sss:shv N coact}

Thus, we obtain that the category $\Shv^{\on{relev}}_\CN(X)$ is a comodule over $\QCoh(\BZ^{\on{alg,wt}})$, and in 
particular over $\QCoh(\BG_m)$. Set
$$\Shv^{\on{relev,gr}}_\CN(X):=(\Shv^{\on{relev}}_\CN(X))^{\BG_m}\subset (\Shv^{\on{relev}}(X))^{\BG_m}=\Shv^{\on{relev,gr}}(X).$$

Note that under the identification
$$\Shv^{\on{relev,gr}}(X)\simeq \Rep(\BG_m)\underset{\Rep(\BZ^{\on{alg,wt}})}\otimes \Shv^{\on{Weil,wt}}(X)\simeq
\Vect_{\sfe}\underset{\Rep(\BZ^{\on{alg},0})}\otimes \Shv^{\on{Weil,wt}}(X),$$
the full subcategory 
$$\Shv^{\on{relev,gr}}_\CN(X)\subset \Shv^{\on{relev,gr}}(X)$$
corresponds to
$$\Rep(\BG_m)\underset{\Rep(\BZ^{\on{alg,wt}})}\otimes \Shv_\CN^{\on{Weil,wt}}(X)\simeq
\Vect_{\sfe}\underset{\Rep(\BZ^{\on{alg},0})}\otimes \Shv_\CN^{\on{Weil,wt}}(X).$$

\sssec{}

We now claim: 

\begin{prop}
The full subcategory $\Shv^{\on{relev}}_\CN(X)\subset \Shv^{\on{relev}}(X)$ is stable under the coaction of
$\QCoh(\BA^1)$. 
\end{prop}

\begin{proof}

The assertion of the proposition is equivalent to the statement that the endofunctors
$i_{\leq n}\circ (i_{\leq n})^R$ of $\Shv^{\on{Weil,wt}}(X)$ preserve the subcategory $\Shv^{\on{Weil,wt}}_\CN(X)$. 

\medskip

Note that the functor $i_{\leq n}\circ (i_{\leq n})^R$ is t-exact (it is enough to check this on the heart, where it is obvious).
Hence, it is sufficient to show that for an irreducible object $\CF\in \Shv^{\on{Weil,wt}}_\CN(X)^\heartsuit$, the object 
$i_{\leq n}\circ (i_{\leq n})^R(\CF)$ belongs to $\Shv^{\on{Weil,wt}}_\CN(X)$.

\medskip

However, this is obvious: 

\medskip

If $\on{wt}(\CF)\leq n$, then $i_{\leq n}\circ (i_{\leq n})^R(\CF)\simeq \CF$, and if
 $\on{wt}(\CF)> n$, then $i_{\leq n}\circ (i_{\leq n})^R(\CF)=0$.
 
 \end{proof}

\begin{rem}

Note that we do not know in general whether the category $\Shv_\CN(X)$ is compactly generated (or even dualizable).
Hence, we do not know this either for the categories $\Shv_\CN^?(X)$ introduced above, with the exception of 
$\Shv^{\on{relev},0}_\CN(X)$, which
is semi-simple and hence compactly generated, and hence so is $\Shv_\CN^{\on{Weil,0}}(X)\simeq (\Shv^{\on{relev},0}_\CN(X))^{\BZ^{\on{alg},0}}$. 

\end{rem} 

\ssec{Weil local systems}

From now on we will assume that $X$ is smooth and connected. 

\sssec{}

We will apply the above discussion to $\CN=\{0\}$. In this case will adopt a special notation
$$\Shv^?_{\{0\}}(X)=:\qLisse^?(X).$$

\sssec{}

Let us choose a base point $x\in X(\ol{k})$. This choice gives rise to a fiber functor 
$$\qLisse(X)^\heartsuit\to \Vect_\sfe.$$

Denote by 
$$\on{Gal}^{\on{alg}}(X)$$ 
the corresponding pro-algebraic group. 

\medskip

Restricting to $\qLisse^{\on{relev}}(X)^\heartsuit\subset \qLisse(X)^\heartsuit$ we obtain a pro-algebraic group, to be denoted
$$\on{Gal}^{\on{alg,relev}}(X),$$ 
which is naturally a quotient of $\on{Gal}^{\on{alg}}(X)$.

\medskip

Denote by 
$$\on{Gal}^{\on{red}}(X) \text{ and } \on{Gal}^{\on{red,relev}}(X)$$
the maximal reductive quotients of $\on{Gal}^{\on{alg}}(X)$ and $\on{Gal}^{\on{alg,relev}}(X)$, respectively. 

\sssec{}

We have tautologically defined t-exact functors
\begin{equation} \label{e:Tannakian}
\Rep(\on{Gal}^{\on{alg}}(X)) \to \qLisse(X) \text{ and } \Rep(\on{Gal}^{\on{alg,relev}}(X)) \to \qLisse^{\on{relev}}(X),
\end{equation} 
which induce equivalence on the hearts of the t-structures. 

\sssec{}

Consider now the category $\qLisse^{\on{relev},0}(X)^\heartsuit$ with its induced fiber functor. Denote the corresponding pro-algebraic group by 
$$\on{Gal}^{\on{alg,relev},0}(X).$$ 

Note, however, that since the abelian category $\qLisse^{\on{relev},0}(X)^\heartsuit$ is semi-simple, the group $\on{Gal}^{\on{alg,relev},0}(X)$ is pro-reductive. Moreover,
since the entire category $\qLisse^{\on{relev},0}(X)$ is semi-simple, the functor
$$\Rep(\on{Gal}^{\on{alg,relev},0}(X))\to \qLisse^{\on{relev},0}(X)$$
is an equivalence.

\sssec{}

The functor
$$\qLisse^{\on{relev},0}(X)\to \qLisse^{\on{relev}}(X)$$
induces a homomorphism:

\begin{equation} \label{e:Gal red}
\on{Gal}^{\on{alg,relev}}(X)\to \on{Gal}^{\on{alg,relev},0}(X).
\end{equation}

We claim:

\begin{prop} \label{p:Gal red}
The map \eqref{e:Gal red} factors through an isomorphism
$$\on{Gal}^{\on{red,relev}}(X)\overset{\sim}\to  \on{Gal}^{\on{alg,relev},0}(X).$$
\end{prop}

In other words, \propref{p:Gal red} says that the map  \eqref{e:Gal red}
identifies $\on{Gal}^{\on{alg,relev},0}(X)$ with the maximal reductive quotient of 
$\on{Gal}^{\on{alg,relev}}(X)$.

\begin{proof}

The assertion of the proposition is equivalent to the following statement: 

\medskip

An object
$V\in \Rep(\on{Gal}^{\on{alg,relev}}(X))^\heartsuit$ is the restriction of an object in 
$\Rep(\on{Gal}^{\on{alg,relev,0}}(X))^\heartsuit$ if and only if it is semi-simple. 

\medskip

In other words, an object $\CF\in \qLisse^{\on{relev}}(X)^\heartsuit$ belongs to 
$\qLisse^{\on{relev},0}(X)^\heartsuit$ if and only if it is semi-simple. 

\medskip

However, this is the content of Remark \ref{r:semi-simple 0}.

\end{proof} 

\sssec{} \label{sss:contr qLisse}

According to \secref{sss:shv N coact}, the category $\qLisse^{\on{relev}}(X)$ carries a canonical coaction 
of the comonoidal category $\QCoh(\BZ^{\on{alg,wt}})$, compatible with the symmetric monoidal structure. Moreover, the coaction functor
$$\qLisse^{\on{relev}}(X)\to \qLisse^{\on{relev}}(X)\otimes \QCoh(\BZ^{\on{alg,wt}})$$
is right t-exact (in fact, it is t-exact). 

\medskip

In particular, $\qLisse^{\on{relev}}(X)$ carries a coaction of $\QCoh(\BG_m)$ with parallel properties. 

\sssec{}

The key input from the previous two sections is that the above coaction of $\QCoh(\BG_m)$ on  $\qLisse^{\on{relev}}(X)$
extends to a coaction
of $\QCoh(\BA^1)$. This coaction has the following properties:

\medskip

\begin{itemize}

\item It is compatible with the coaction of $\QCoh(\BZ^{\on{alg,wt}})$;

\medskip

\item It is compatible with the symmetric monoidal structure
on $\qLisse^{\on{relev}}(X)$;

\medskip

\item The coaction functor
$$\qLisse^{\on{relev}}(X)\to \qLisse^{\on{relev}}(X)\otimes \QCoh(\BA^1)$$
is right t-exact (in fact, it is t-exact). 

\end{itemize}

\sssec{}

In particular, we obtain a right t-exact (in fact, t-exact) symmetric monoidal endofunctor $[0]$ of $\qLisse^{\on{relev}}(X)$,
which factors as
$$\qLisse^{\on{relev}}(X)\overset{\psi}\to \qLisse^{\on{relev},0}(X) \overset{\phi}\to \qLisse^{\on{relev}}(X),$$
where $\phi$ is the tautological forgetful functor. 

\medskip

The composition $\psi\circ \phi$ is canonically isomorphic to the identity functor. 

\sssec{}

Let us see what the above structure gives for the abelian category 
$$(\qLisse^{\on{relev}}(X))^\heartsuit\simeq \Rep(\on{Gal}^{\on{alg,relev}}(X))^\heartsuit.$$

Since the fiber functor $\qLisse^{\on{relev}}(X)\to \Vect_\sfe$ is compatible with the $\BA^1$-action,
we obtain that the algebraic group $\on{Gal}^{\on{alg,relev}}(X)$ acquires an action of $\BA^1$, which
contracts it to $\on{Gal}^{\on{alg,relev},0}(X)$. This contraction gives rise to a \emph{preferred} Levi
splitting
$$\on{Gal}^{\on{alg,relev}}(X)\simeq \on{Gal}^{\on{alg,relev},0}(X)\ltimes \on{Gal}^{\on{alg,relev,unip}}(X).$$

\begin{rem} \label{r:can Levi}
The above action of $\BG_m\subset \BA^1$ on $\on{Gal}^{\on{alg,relev}}(X)$ can also be viewed as follows.
The algebraic version of the Weil group is an extension
$$1\to \on{Gal}^{\on{alg,relev}}(X) \to \on{Weil}^{\on{alg,loc.fin}}(X)\to \BZ^{\on{alg}}\to 1,$$
so that
$$(\qLisse^{\on{Weil,loc.fin}}(X))^\heartsuit  \simeq \Rep(\on{Weil}^{\on{alg,loc.fin}}(X))^\heartsuit.$$

The extension $\on{Weil}^{\on{alg,loc.fin}}(X)$ is induced from an extension
$$1\to \on{Gal}^{\on{alg,relev}}(X) \to \on{Weil}^{\on{alg,wt}}(X)\to \BZ^{\on{alg,wt}}\to 1,$$
$$(\qLisse^{\on{Weil,wt}}(X))^\heartsuit  \simeq \Rep(\on{Weil}^{\on{alg,wt}}(X))^\heartsuit.$$

The latter extension restricts to an extension
$$1\to \on{Gal}^{\on{alg,relev}}(X) \to \on{Gal}^{\on{alg,gr}}(X)\to \BG_m\to 1,$$
and we have 
$$\on{Gal}^{\on{alg,gr}}(X)\simeq (\on{Gal}^{\on{alg,relev},0}(X)\times \BG_m)\ltimes \on{Gal}^{\on{alg,relev,unip}}(X).$$

Unwinding, we obtain
$$(\qLisse^{\on{relev,gr}}(X))^\heartsuit \simeq \Rep(\on{Gal}^{\on{alg,gr}}(X))^\heartsuit.$$

\end{rem}

\ssec{The stack of local systems with restricted variation}

In this subsection, we review the material from \cite[Sects. 3.5-3.7]{AGKRRV}.

\sssec{} \label{sss:LS restr}

Let $\sG$ be a reductive group, and let $\LS^{\on{restr}}_\sG(X)$ be as in \cite[Sect. 1.4]{AGKRRV}. 

\medskip

By definition, for a $\sfe$-algebra $R$, 
$$\Maps(\Spec(R),\LS^{\on{restr}}_\sG(X))$$
is the groupoid of right t-exact symmetric monoidal functors
$$\Rep(\sG)\to R\mod(\qLisse(X)).$$

\sssec{} \label{sss:classical}

Note that the \emph{classical stack} underlying $\LS^{\on{restr}}_\sG(X)$ identifies with the \emph{classical stack} underlying 
$$\bMaps_{\on{groups}}(\on{Gal}^{\on{alg}}(X),\sG)/\on{Ad}(\sG),$$
where $\bMaps_{\on{groups}}(-,-)$ is the stack of maps of pro-algebraic groups. 

\sssec{} 

The projection 
$$\on{Gal}^{\on{alg}}(X)\to \on{Gal}^{\on{red}}(X)$$
gives rise to a closed embedding\footnote{The first isomorphism in the formula below is due to the fact that the mapping
space between reductive groups is classical.}
\begin{multline} \label{e:emb semi-simple locus}
\LS^{\on{restr},0}_\sG(X):=\bMaps_{\on{groups}}(\on{Gal}^{\on{red}}(X),\sG)/\on{Ad}(\sG) \simeq \\
\simeq {}^{\on{cl}}\bMaps_{\on{groups}}(\on{Gal}^{\on{red}}(X),\sG)/\on{Ad}(\sG) 
\hookrightarrow  {}^{\on{cl}}\bMaps_{\on{groups}}(\on{Gal}^{\on{alg}}(X),\sG)/\on{Ad}(\sG)\simeq \\
\simeq {}^{\on{cl}}\!\LS^{\on{restr}}_\sG(X)\hookrightarrow \LS^{\on{restr}}_\sG(X).
\end{multline} 

At the level of $\sfe$-points, the image of \eqref{e:emb semi-simple locus} corresponds to \emph{semi-simple} local systems, see \cite[Sect. 3.6]{AGKRRV}. 
Moreover, the map \eqref{e:emb semi-simple locus} induces a bijection between connected components.

\sssec{} \label{sss:semi-simple bis}

For future use we also note:

\medskip

\begin{enumerate}

\item The automorphism group of a semi-simple local system is reductive;

\medskip

\item For a connected component $\LS^{\on{restr}}_\sG(X)_\alpha$ of $\LS^{\on{restr}}_\sG(X)$, the intersection
$$\LS^{\on{restr},0}_\sG(X)_\alpha:=\LS^{\on{restr},0}_\sG(X)\cap \LS^{\on{restr}}_\sG(X)_\alpha$$
is the unique closed substack of ${}^{\on{cl}}\!\LS^{\on{restr}}_\sG(X)_\alpha$ that has a single isomorphism class of $\sfe$-points.

\end{enumerate} 

\sssec{}  

Note that the map
$$\LS^{\on{restr},0}_\sG(X)\to {}^{\on{cl}}\!\LS^{\on{restr}}_\sG(X)$$
admits a left inverse, 
\begin{multline} \label{e:emb semi-simple locus split}
{}^{\on{cl}}\!\LS^{\on{restr}}_\sG(X)\simeq {}^{\on{cl}}\bMaps_{\on{groups}}(\on{Gal}^{\on{alg}}(X),\sG)/\on{Ad}(\sG)\to  \\
\to {}^{\on{cl}}\bMaps_{\on{groups}}(\on{Gal}^{\on{red}}(X),\sG)/\on{Ad}(\sG) \simeq
\bMaps_{\on{groups}}(\on{Gal}^{\on{red}}(X),\sG)/\on{Ad}(\sG) = \\
= \LS^{\on{restr},0}_\sG(X)
\end{multline} 
 given by restriction along a choice of a Levi splitting
$$\on{Gal}^{\on{alg}}(X)\simeq \on{Gal}^{\on{red}}(X)\ltimes \on{Gal}^{\on{unip}}(X).$$

\medskip

For a geometric point of $\LS^{\on{restr}}_\sG(X)$, we will refer to its image under \eqref{e:emb semi-simple locus split}
as its \emph{semi-simplification}; it is well-defined up to isomorphism. 

\medskip

Thus, we have:

\begin{lem} \label{l:semi-simple}
Two geometric points of $\LS^{\on{restr}}_\sG(X)$ belong to the same connected component if and only if
they have isomorphic semi-simplifications.
\end{lem} 

\sssec{} \label{sss:LS arithm}

From now on we will take $k=\BF_q$. In this case, the action of $\Frob^*$ on $\qLisse(X)$ gives rise to an
automorphism of $\LS^{\on{restr}}_\sG(X)$, which we denote by $\Frob$. 

\medskip

We set
$$\LS^{\on{arithm}}_\sG(X):=(\LS^{\on{restr}}_\sG(X))^{\Frob}.$$

According to \cite[Theorem 24.1.4]{AGKRRV}, the prestack $\LS^{\on{arithm}}_\sG(X)$ is a disjoint union of 
quasi-compact algebraic stacks, each of which can be written as a quotient of an affine scheme by an action of a reductive 
group. 

\begin{rem}

By definition, for an $\sfe$-algebra $R$, 
$$\Maps(\Spec(R),\LS^{\on{arithm}}_\sG(X))$$
is the groupoid of right t-exact symmetric monoidal functors
$$\Rep(\sG)\to R\mod(\qLisse^{\on{Weil}}(X)).$$

Note that in the above formula, it is really $\qLisse^{\on{Weil}}(X)$ and not $\qLisse^{\on{Weil,loc.fin}}(X)$; the latter would produce a different object,
denoted $\LS^{\on{arithm-restr}}_\sG(X)$. For example, for $X=\on{pt}$, we have:

\medskip

\begin{itemize}

\item $\LS^{\on{restr}}_\sG(X)=\on{pt}/\sG$;

\medskip

\item $\LS^{\on{arithm}}_\sG(X)=\sG/\on{Ad}(\sG)$;

\medskip

\item $\LS^{\on{arithm-restr}}_\sG(X)$ is the union of formal neighborhoods of loci in $\sG/\on{Ad}(\sG)$ corresponding to elements in $\sG$
with a fixed semi-simple part.

\end{itemize} 

\end{rem}

\ssec{The substack of \emph{relevant} local systems}

\sssec{}

We now consider a variant of $\LS^{\on{restr}}_\sG(X)$, denoted $\LS^{\on{restr,relev}}_\sG(X)$, constructed in the
same way as $\LS^{\on{restr}}_\sG(X)$, with the difference that instead of the \emph{gentle Tannakian} 
category\footnote{See \cite[Sect. 1.7]{AGKRRV} for what this means.}
$\qLisse(X)$ we use $\qLisse^{\on{relev}}(X)$. 

\medskip

As in \secref{sss:classical}, the classical stack underlying $\LS^{\on{restr,relev}}_\sG(X)$ identifies with
$$^{\on{cl}}\bMaps_{\on{groups}}(\on{Gal}^{\on{alg,relev}}(X),\sG)/\on{Ad}(\sG).$$

\sssec{}

The (fully faithful) embedding
$$\qLisse^{\on{relev}}(X)\hookrightarrow \qLisse(X)$$ 
gives rise to a \emph{monomorphism}
\begin{equation} \label{e:LS relev to all}
\LS^{\on{restr,relev}}_\sG(X) \hookrightarrow \LS^{\on{restr}}_\sG(X).
\end{equation} 

\medskip

We claim:

\begin{prop} \label{p:LS relev to all}
The map \eqref{e:LS relev to all} is the inclusion of the union of some of the connected components.
\end{prop} 

\begin{proof}

It is clear that the map \eqref{e:LS relev to all} is formally \'etale 
(i.e., it induces an isomorphism on cotangent spaces). Hence, it is enough to prove the corresponding assertion at the level of the 
underlying classical substacks, i.e., 
$$^{\on{cl}}\bMaps_{\on{groups}}(\on{Gal}^{\on{alg}}(X),\sG)/\on{Ad}(\sG)\to 
{}^{\on{cl}}\bMaps_{\on{groups}}(\on{Gal}^{\on{alg,relev}}(X),\sG)/\on{Ad}(\sG).$$

\medskip

For a connected classical scheme $S=\Spec(R)$, let us be given an $S$-family of homomorphisms
\begin{equation} \label{e:R pt LS}
\on{Gal}^{\on{alg}}(X)\overset{\phi_S}\to \sG
\end{equation} 
such that for some $s\in S(\sfe)$, the corresponding homomorphism $\phi_s$ factors through 
\begin{equation} \label{e:red quot again}
\on{Gal}^{\on{alg}}(X)\twoheadrightarrow \on{Gal}^{\on{alg,relev}}(X).
\end{equation} 

We need to show that \eqref{e:R pt LS} factors through \eqref{e:red quot again}.

\medskip

The latter is equivalent to saying that for any $V\in \Rep(\sG)^\heartsuit$, the corresponding object 
$$\phi_S^*(V)\in R\mod(\qLisse(X))^\heartsuit$$
belongs to $R\mod(\qLisse^{\on{relev}}(X))^\heartsuit$. For that, it is enough to
show that for every geometric point 
$$\Spec(\sfe')\overset{s'}\to S,$$
the corresponding object $\phi_{s'}^*(V)\in \sfe'\mod(\qLisse(X))^\heartsuit$ belongs to 
\begin{equation} \label{e:base change s'}
\sfe'\mod(\qLisse^{\on{relev}}(X))\simeq \Vect_{\sfe'}\underset{\Vect_{\sfe}}\otimes \qLisse^{\on{relev}}(X).
\end{equation} 

\medskip

By \corref{c:char relev}\footnote{Note that here we are using an easy aspect of \corref{c:char relev}, i.e., when 
an object in question lies in the heart of the t-structure, i.e., we are not using the left-completeness property.},
this is equivalent to the fact that all irreducible subquotients of all perverse cohomologies 
of $\phi_{s'}^*(V)$ belong to \eqref{e:base change s'}. However, by \lemref{l:semi-simple}, these subquotients
are the same as the corresponding subquotients for $\phi_s(V)$ (base changed $\sfe\rightsquigarrow \sfe'$),
which belong to \eqref{e:base change s'}, by assumption. 

\end{proof} 

\sssec{}

Let $\LS^{\on{restr,relev},0}_\sG(X)$ be a variant of $\LS^{\on{restr,relev}}_\sG(X)$ constructed out of the gentle
Tannakian category $\qLisse^{\on{relev},0}(X)$. The symmetric monoidal functor 
$$\qLisse^{\on{relev},0}(X)\to \qLisse^{\on{relev}}(X)$$
gives rise to a map
\begin{equation} \label{e:semi-simple relev}
\LS^{\on{restr,relev},0}_\sG(X)\to \LS^{\on{restr,relev}}_\sG(X).
\end{equation}

We claim:

\begin{prop} \label{p:semi-simple relev}
The map \eqref{e:semi-simple relev} factors through an isomorphism
$$\LS^{\on{restr,relev},0}_\sG(X)\overset{\sim}\to \LS^{\on{restr},0}_\sG(X)\underset{\LS^{\on{restr}}_\sG(X)}\times \LS^{\on{restr,relev}}_\sG(X).$$
\end{prop}

\begin{proof}

By \propref{p:LS relev to all}, the right-hand side is the union of connected components of
$$\bMaps_{\on{groups}}(\on{Gal}^{\on{red}}(X),\sG)/\on{Ad}(\sG)$$
that correspond to homomorphisms that factor via
$$\on{Gal}^{\on{red}}(X)\twoheadrightarrow \on{Gal}^{\on{red,relev}}(X).$$

The required assertion follows now from \propref{p:Gal red}.

\end{proof} 

\sssec{}

Finally, we claim:

\begin{lem} \label{l:relev catches all}
The inclusion
$$(\LS^{\on{restr,relev}}_\sG(X))^{\Frob}\hookrightarrow (\LS^{\on{restr}}_\sG(X))^{\Frob}=:\LS^{\on{arithm}}_\sG(X)$$
is an isomorphism.
\end{lem}

\begin{proof} 

By \propref{p:LS relev to all}, a priori, the left-hand side is the union of some of the connected components
of the right hand side. Hence, in order to prove the lemma, we need to show that if an $\sfe$-point $\sigma$ of 
$\LS^{\on{restr}}_\sG(X)$ factors through
$\LS^{\on{arithm}}_\sG(X)$, then it belongs to $\LS^{\on{restr,relev}}_\sG(X)$. 

\medskip

I.e., we need to show that for $V\in \Rep(\sG)$, the corresponding object $V_\sigma\in \qLisse(X)$
belongs to $\qLisse^{\on{relev}}(X)$. Note that by assumption, $V_\sigma$ lies in the image 
of $$\oblv_{\on{Weil}}:\qLisse^{\on{Weil}}(X)\to \qLisse(X).$$

\medskip

With no restriction of generality, we can assume that $V$ is compact. Then $V_\sigma$ is constructible.
Hence, we obtain that $V_\sigma$ lies in the essential image of
$$\qLisse^{\on{Weil,loc.fin}}(X)\to \qLisse(X).$$

Hence, it belongs to $\qLisse^{\on{relev}}(X)$.

\end{proof}

\ssec{The contraction} \label{ss:contr LS}

\sssec{} \label{e:W on Lisse} 

Recall that according to \thmref{t:contraction on category}, the category $\qLisse^{\on{relev}}(X)$ carries a coaction of
of the comonoidal category $\QCoh(\BZ^{\on{alg,wt}})$, and in particular of $\QCoh(\BG_m)$. 

\medskip

By the functoriality of the construction in \cite[Sect. 1.8]{AGKRRV} with respect to the gentle Tannakian category, 
we obtain that the prestack $\LS^{\on{restr,relev}}_\sG(X)$
carries an action of the pro-algebraic group $\BZ^{\on{alg,wt}}$, which restricts to an action of $\BG_m$. 

\sssec{}

In addition, the above coaction of $\QCoh(\BG_m)$ on $\qLisse^{\on{relev}}(X)$
extends to a coaction of $\QCoh(\BA^1)$. 

\medskip

Applying again the functoriality of the construction in \cite[Sect. 1.8]{AGKRRV} with respect to the gentle Tannakian category, 
we obtain that $\LS^{\on{restr,relev}}_\sG(X)$ carries an action of the monoid $\BA^1$. 

\medskip

Moreover, this action is compatible
with the action of $\BZ^{\on{alg,wt}}$, and the resulting actions of
$$\BA^1 \hookleftarrow \BG_m\hookrightarrow  \BZ^{\on{alg,wt}}$$
coincide. 

\sssec{}

Consider the resulting endomorphism $[0]$ of $\LS^{\on{restr,relev}}_\sG(X)$. By construction, it factors as
$$\LS^{\on{restr,relev}}_\sG(X) \overset{\psi}\to \LS^{\on{restr,relev},0}_\sG(X)  \overset{\phi}\to \LS^{\on{restr,relev}}_\sG(X),$$
where:

\begin{itemize}

\item $\phi$ is the map \eqref{e:semi-simple relev};

\medskip

\item The map $\psi$ is induced by the (right t-exact symmetric monoidal) functor
$$\psi:\qLisse^{\on{relev}}(X)\to \qLisse^{\on{relev},0}(X);$$

\item The maps $\phi$ and $\psi$ are $\BZ^{\on{alg,wt}}$-equivariant, where the action on $\qLisse^{\on{relev},0}(X)$ is via
$$\BZ^{\on{alg,wt}}\twoheadrightarrow \BZ^{\on{alg},0}.$$

\item $\psi\circ \phi\simeq \on{id}$.

\end{itemize}

\sssec{} \label{sss:contraction!}

In particular, we can view $\LS^{\on{restr,relev}}_\sG(X)$ as a prestack \emph{over} $\LS^{\on{restr,relev},0}_\sG(X)$ by means
of $\psi$, and as such it carries a fiber-wise action of $\BA^1$.

\medskip

We think of this structure as a \emph{contraction} of $\LS^{\on{restr,relev}}_\sG(X)$ to $\LS^{\on{restr,relev},0}_\sG(X)$.

\ssec{Contraction on the stack of arithmetic local systems}

\sssec{} 

Since the action of the monoid $\BA^1$ on $\LS^{\on{restr,relev}}_\sG(X)$ commutes with the action of
$\BZ^{\on{alg,wt}}$, and in particular of the element $\Frob$, it induces an action of $\BA^1$ on the stack
$$(\LS^{\on{restr,relev}}_\sG(X))^{\Frob}\overset{\text{\lemref{l:relev catches all}}}\simeq \LS^{\on{arithm}}_\sG(X).$$

\sssec{}

In particular, we obtain the endomorphism $[0]$ of $\LS^{\on{arithm}}_\sG(X)$, which factors as 
\begin{equation} \label{e:contraction arithm}
\LS^{\on{arithm}}_\sG(X) \overset{\psi}\to (\LS^{\on{restr},0}_\sG(X))^{\on{Frob}} \overset{\phi}\to \LS^{\on{arithm}}_\sG(X),
\end{equation}
where we regard the stacks and maps in \eqref{e:contraction arithm} as compatible with an action of $\BZ^{\on{alg,wt}}$, and where we identify 
$$(\LS^{\on{restr,relev},0}_\sG(X))^{\on{Frob}}\simeq  (\LS^{\on{restr},0}_\sG(X))^{\on{Frob}}$$
thanks to \propref{p:semi-simple relev} and \lemref{l:relev catches all}.  

\sssec{} \label{sss:contraction! arithm}

As in \secref{sss:contraction!}, we can thus view $\LS^{\on{arithm}}_\sG(X)$ as a prestack over $(\LS^{\on{restr},0}_\sG(X))^{\on{Frob}}$,
equipped with a fiber-wise action of $\BA^1$. 

\sssec{} \label{sss:arithm 0}

Denote
$$\LS^{\on{arithm},0}_\sG(X):=(\LS^{\on{restr},0}_\sG(X))^{\on{Frob}}.$$

\medskip

Note that by construction, for an $\sfe$-algebra $R$, 
$$\Maps(\Spec(R),\LS^{\on{arithm},0}_\sG(X))$$
is the groupoid of right t-exact symmetric monoidal functors
$$\Rep(\sG)\to R\mod((\qLisse^{\on{relev},0}(X))^{\Frob}).$$

Note also that the symmetric monoidal category $(\qLisse^{\on{relev},0}(X))^{\Frob}$ appearing in the above formula
can be rewritten as 
\begin{equation} \label{e:strange category}
\BZ\mod \underset{\Rep(\BZ^{\on{alg},0})}\otimes \qLisse^{\on{Weil},0}(X).
\end{equation} 

\sssec{} \label{sss:conn comp LS 0}

Let us analyze the structure of the (prestack) $\LS^{\on{arithm},0}_\sG(X)$. 

\medskip

Let $\LS^{\on{restr}}_\sG(X)_\alpha$ be a connected component of $\LS^{\on{restr}}_\sG(X)$ that intersects 
$\LS^{\on{arithm}}_\sG(X)$ nontrivially; this is equivalent to saying that $\LS^{\on{restr}}_\sG(X)_\alpha$ 
is Frobenius-invariant. 

\medskip

Let $\sigma_\alpha$ be the corresponding semi-simple local system, which is defined up to an isomorphism. 
Denote $\sG_\alpha:=\on{Autom}(\sigma_\alpha)$. By
\secref{sss:semi-simple bis}(1), $\sG_\alpha$ is reductive. 

\medskip

We can think about $\sigma_\alpha$ as a homomorphism 
$$\phi_\alpha:\on{Gal}^{\on{alg}}(X)\to \sG;$$
then $\sG_\alpha:=Z_\sG(\on{Im}(\phi_\alpha))$. 

\medskip

The action of Frobenius on $\qLisse(X)$ gives rise to an extension\footnote{Cf. Remark \ref{r:can Levi}.}
\begin{equation} \label{e:alg Weil}
1\to \on{Gal}^{\on{alg}}(X)\to \on{Weil}^{\on{alg}}(X)\to \BZ\to 1.
\end{equation} 

Pick a lift of the element $1\in \BZ$ to $\on{Weil}^{\on{alg}}(X)$; denote it by $\Frob$. 

\medskip

By \secref{sss:semi-simple bis}(2), the fact that $\LS^{\on{restr}}_\sG(X)_\alpha$  is Frobenius-invariant means
that so is the substack $\LS^{\on{restr},0}_\sG(X)_\alpha$. I.e., $\sigma_\alpha$ admits a structure of Frobenius-equivariance. 
This means that there exists an element $\sg\in \sG$ such that the homomorphism $\phi_\alpha$ intertwines 
the action of $\Frob$ on $\on{Gal}^{\on{alg}}(X)$ by conjugation and the adjoint action of $\sg$. 

\medskip

The variety of such elements, to be denoted $\sF_\alpha$, is a torsor with respect to $\sG_\alpha$ with respect to
both left and right multiplication. In particular, we have a well-defined \emph{adjoint} action of $\sG_\alpha$ on $\sF_\alpha$.

\medskip

We identify
$$\LS^{\on{restr},0}_\sG(X)_\alpha\simeq \on{pt}/\sG_\alpha,$$
and we obtain that
\begin{equation} \label{e:structure stabilizer}
(\LS^{\on{restr},0}_\sG(X)_\alpha)^{\on{Frob}}\simeq \sF_\alpha/\on{Ad}(\sG_\alpha).
\end{equation} 

\sssec{}

We now claim:

\begin{prop} \label{p:contr affine}
The map $\psi$ in \eqref{e:contraction arithm} is affine.
\end{prop}

\begin{proof}

Let us base change the morphism in question by means of $\on{pt}\to \on{pt}/\sG$, where we map both sides to $\on{pt}/\sG$
by composing the forgetful map to $\LS^{\on{restr}}_\sG(X)$ with the map
$$\LS^{\on{restr}}_\sG(X)\to \LS^{\on{restr}}_\sG(\on{pt})\simeq \on{pt}/\sG,$$
corresponding to a point $x\in X$. 

\medskip

It suffices to show that for every connected component $\LS^{\on{arithm}}_\sG(X)_\alpha$ of $\LS^{\on{arithm}}_\sG(X)$,
the resulting map
\begin{equation} \label{e:contr affine}
\LS^{\on{arithm}}_\sG(X)_\alpha\underset{\on{pt}/\sG}\times \on{pt}\to \LS^{\on{arithm},0}_\sG(X)_\alpha \underset{\on{pt}/\sG}\times \on{pt}
\end{equation}
is affine. We claim that both prestacks appearing in \eqref{e:contr affine} are in fact affine schemes. 

\medskip

Indeed,
$$\LS^{\on{arithm}}_\sG(X)_\alpha\underset{\on{pt}/\sG}\times \on{pt}$$
is an affine scheme by \cite[Theorem 24.1.4]{AGKRRV}. 

\medskip

The fiber product $\LS^{\on{arithm},0}_\sG(X)_\alpha \underset{\on{pt}/\sG}\times \on{pt}$ is an affine scheme by \eqref{e:structure stabilizer}:
in fact, it is isomorphic to
$$(\sF_\alpha\times \sG)/\sG_\alpha,$$
which is affine since $\sG/\sG_\alpha$ is affine (being the quotient of a reductive group by a reductive subgroup). 

\end{proof} 

\ssec{The excursion algebra} 

\sssec{}

We are now ready to state and prove the main result of this paper:

\begin{thm} \label{t:exc}
The map
\begin{equation} \label{e:exc}
\Gamma(\LS^{\on{arithm}}_\sG(X),\CO_{\LS^{\on{arithm}}_\sG(X)})\to
\Gamma(\LS^{\on{arithm},0}_\sG(X),\CO_{\LS^{\on{arithm},0}_\sG(X)})
\end{equation} 
is an isomorphism.
\end{thm} 

Note that by combining \thmref{t:exc} and \eqref{e:structure stabilizer} we obtain:

\begin{cor} \label{c:exc}
For every connected component $\LS^{\on{arithm}}_\sG(X)_\alpha$ of $\LS^{\on{arithm}}_\sG(X)$, the 
algebra 
$$\on{Exc}(X,\sG)_\alpha:=\Gamma(\LS^{\on{arithm}}_\sG(X)_\alpha,\CO_{\LS^{\on{arithm}}_\sG(X)_\alpha})$$ is classical, integral and normal.
\end{cor}

\begin{proof}

Indeed, the algebras
$$\Gamma(\sF_\alpha/\on{Ad}(\sG_\alpha),\CO_{\sF_\alpha/\on{Ad}(\sG_\alpha)})\simeq
\Gamma(\sF_\alpha,\CO_{\sF_\alpha})^{\sG_\alpha}$$ have this property, since
since $\sF_\alpha$ is a torsor over $\sG_\alpha$, while $\sG_\alpha$ is reductive.

\end{proof}

\sssec{}

For a connected component $\LS^{\on{arithm}}_\sG(X)_\alpha$ of $\LS^{\on{arithm}}_\sG(X)$, the algebra 
$$\on{Exc}(X,\sG)_\alpha$$
appearing in \corref{c:exc} is called the 
\emph{Excursion Algebra} (corresponding to the given connected component). 

\medskip

Denote
$$\LS^{\on{arithm,coarse}}_\sG(X)_\alpha:=\Spec(\on{Exc}(X,\sG)_\alpha).$$

\medskip

Thus, \corref{c:exc} says that each connected component of $\LS^{\on{arithm,coarse}}_\sG(X)_\alpha$ is classical, integral and normal.

\sssec{}

We now begin the proof of \thmref{t:exc}. The first step of the proof is an assertion of independent interest.

\medskip

Since $\LS^{\on{arithm}}_\sG(X)_\alpha$ is a quotient of an affine scheme by an algebraic group, the functor
$$\Gamma(\LS^{\on{arithm}}_\sG(X)_\alpha,-)$$
preserves colimits. Hence, the 
action of $\BZ^{\on{alg,wt}}$ on $\LS^{\on{arithm}}_\sG(X)_\alpha$ allows us to upgrade $\on{Exc}(X,\sG)_\alpha$ to
an object of $\Rep(\BZ^{\on{alg,wt}})$. 

\medskip

Note, however, that the action of $\Frob\in \BZ^{\on{alg,wt}}$ on $\LS^{\on{arithm}}_\sG(X)_\alpha$ is tautologically trivial. Hence, the image
of $\on{Exc}(X,\sG)_\alpha$ under the functor
$$\Rep(\BZ^{\on{alg,wt}})\to \BZ\mod$$
is multiple of the trivial representation. 

\medskip

However, since the above functor is fully faithful, we obtain:

\begin{lem}  \label{l:Exc triv}
The object $\on{Exc}(X,\sG)_\alpha\in \Rep(\BZ^{\on{alg,wt}})$ is a multiple of the trivial representation.
\end{lem} 

\begin{cor}  \label{c:Exc triv}
The action of $\BG_m$ on $\on{Exc}(X,\sG)_\alpha$ obtained by restriction along $\BG_m\to \BZ^{\on{alg,wt}}$
is trivial.
\end{cor} 

\sssec{}

From \corref{c:Exc triv} we obtain that in order to prove \thmref{t:exc}, it suffices to show that the map 
\begin{equation} \label{e:exc Gm}
\left(\Gamma\left(\LS^{\on{arithm}}_\sG(X),\CO_{\LS^{\on{arithm}}_\sG(X)}\right)\right)^{\BG_m}\to
\left(\Gamma\left(\LS^{\on{arithm},0}_\sG(X),\CO_{\LS^{\on{arithm},0}_\sG(X)}\right)\right)^{\BG_m},
\end{equation} 
induced by \eqref{e:exc}, is an isomorphism. 

\sssec{}

Recall the map
$$\psi:\LS^{\on{arithm}}_\sG(X)\to \LS^{\on{arithm},0}_\sG(X).$$

It suffices to show that the map
\begin{equation} \label{e:exc Gm inverse}
\left(\Gamma\left(\LS^{\on{arithm},0}_\sG(X),\CO_{\LS^{\on{arithm},0}_\sG(X)}\right)\right)^{\BG_m} 
\to \left(\Gamma\left(\LS^{\on{arithm}}_\sG(X),\CO_{\LS^{\on{arithm}}_\sG(X)}\right)\right)^{\BG_m},
\end{equation} 
given by pullback along $\psi$, is an isomorphism.

\sssec{}

We rewrite 
$$\Gamma(\LS^{\on{arithm}}_\sG(X),\CO_{\LS^{\on{arithm}}_\sG(X)})\simeq 
\Gamma(\LS^{\on{arithm},0}_\sG(X),\psi_*(\CO_{\LS^{\on{arithm}}_\sG(X)})).$$

\medskip

The map \eqref{e:exc Gm inverse} is obtained by applying $\BG_m$-invariants to the map
$$\Gamma\left(\LS^{\on{arithm},0}_\sG(X),\CO_{\LS^{\on{arithm},0}_\sG(X)}\right)\to
\Gamma\left(\LS^{\on{arithm},0}_\sG(X),\psi_*(\CO_{\LS^{\on{arithm}}_\sG(X)})\right).$$

\sssec{}

Recall now (see \secref{sss:contraction! arithm}) that we can view the action of $\BG_m$ on $\LS^{\on{arithm}}_\sG(X)$
as \emph{fiberwise} for the projection $\psi$. Moreover, this action is \emph{contracting}, i.e., extends to a fiber-wise action of
$\BA^1$, and the action of $0\in \BA^1$ is the composition $\phi\circ \psi$. 

\medskip

The map \eqref{e:exc Gm inverse} is obtained by applying $\Gamma(\LS^{\on{arithm},0}_\sG(X),-)$
to the map
\begin{equation} \label{e:exc Gm inverse local}
\CO_{\LS^{\on{arithm}}_\sG(X)}\simeq (\CO_{\LS^{\on{arithm}}_\sG(X)})^{\BG_m}\to \psi_*(\CO_{\LS^{\on{arithm}}_\sG(X)})^{\BG_m}.
\end{equation} 

\medskip

It suffices to show that the map \eqref{e:exc Gm inverse local} is an isomorphism. However, this follows from 
\propref{p:contr affine}: indeed we are dealing with an affine map with a fiberwise contracting action of $\BG_m$.

\qed[\thmref{t:exc}]

\begin{rem}
Here is a different argument proving that \eqref{e:exc Gm} is an isomorphism (one that avoids \propref{p:contr affine}):

\medskip

Note that the composition
\eqref{e:exc Gm}$\circ$\eqref{e:exc Gm inverse} is the identity map. So it is enough to show that
the composition \eqref{e:exc Gm inverse}$\circ$\eqref{e:exc Gm} is the identity map. 

\medskip

Now, the composition
\begin{multline} \label{e:psiphi}
\Gamma\left(\LS^{\on{arithm}}_\sG(X),\CO_{\LS^{\on{arithm}}_\sG(X)}\right) \overset{\phi^*}\to \\
\to \Gamma\left(\LS^{\on{arithm},0}_\sG(X),\CO_{\LS^{\on{arithm},0}_\sG(X)}\right) \overset{\psi^*}\to 
\Gamma\left(\LS^{\on{arithm}}_\sG(X),\CO_{\LS^{\on{arithm}}_\sG(X)}\right)
\end{multline}
 equals $[0]^*$.

\medskip

We can regard $\Gamma\left(\LS^{\on{arithm}}_\sG(X),\CO_{\LS^{\on{arithm}}_\sG(X)}\right)$
as a representation of the monoid $\BA^1$, so that the map $[0]^*$ above is the action of point $0\in \BA^1$.

\medskip

Now, for any $W\in \Rep(\BA^1)$, the action of $0$ on $W^{\BG_m}$ is the identity, while 
the composition \eqref{e:exc Gm inverse}$\circ$\eqref{e:exc Gm} is obtained from \eqref{e:psiphi}
by taking $\BG_m$-invariants. 

\end{rem} 

\section{Properties of the excursion algebra} \label{s:fin}

In this section we will connect our definition of the excursion algebra to (a variant of) V.~Lafforgue's definition.

\medskip

In addition, we will study its finiteness properties over the \emph{Hecke} algebra. 

\ssec{A change of conventions}

\sssec{} 

The contents of the previous section admit the following variant: 

\medskip

We replace the category $\qLisse(X)$ by its variant, where we bound the degree of ramification 
(at the generic point of each boundary divisor for a given compactification of $X$). 

\sssec{} \label{sss:ram}

That said, we will \emph{not} change the notation, and denote the resulting category 
by the same symbol $\qLisse(X)$.

\medskip

Correspondingly, all the variants (e.g., $\qLisse^{\on{relev}}(X)$, $\qLisse^{\on{Weil}}(X)$), as well as the
moduli spaces that arise from them (e.g., $\LS^{\on{retsr}}_\sG(X)$, $\LS^{\on{arithm}}_\sG(X)$),
should be understood in this new sense.

\sssec{}

The impact of the restriction on ramification is the following: the stack $\LS^{\on{arithm}}_\sG(X)$
has now \emph{finitely many} connected components, see \cite{EK}. 

\sssec{}

Denote 
$$\on{Exc}(X,\sG):=\underset{\alpha}\Pi\, \on{Exc}(X,\sG)_\alpha;$$
this is the ``total" excursion algebra, whose properties we will study
in this section. 

\medskip

Denote also 
$$\LS^{\on{arithm,coarse}}_\sG(X):=\Spec(\on{Exc}(X,\sG)) \simeq \underset{\alpha}\sqcup\, \LS^{\on{arithm,coarse}}_\sG(X)_\alpha.$$

\ssec{Relation to V.~Lafforgue's definition}

\sssec{}

Let $\on{Weil}(X)$ denote the Weil group of $X$:
$$\on{Weil}(X):=\on{Gal}^{\on{arithm}}(X)\underset{\wh\BZ}\times \BZ,$$
equipped with its natural topology.

\medskip

Let $\on{Weil}^{\on{discr}}(X)$ (resp., $\on{Gal}^{\on{discr}}(X)$) be the group $\on{Weil}(X)$ (resp., $\on{Gal}(X)$), viewed as a discrete group. 
Recall the group $\on{Weil}^{\on{alg}}(X)$, see \eqref{e:alg Weil}. By construction, we have
$$\Rep(\on{Weil}^{\on{alg}}(X))^\heartsuit\simeq \qLisse^{\on{Weil}}(X)^\heartsuit.$$

\medskip

We have tautological homomorphisms
\begin{equation} \label{e:discr to usual}
\on{Weil}^{\on{discr}}(X)\to \on{Weil}^{\on{alg}}(X) \text{ and } \on{Gal}^{\on{discr}}(X)\to \on{Gal}^{\on{alg}}(X).
\end{equation} 

The first of the above homomorphisms gives rise to a functor 
\begin{equation} \label{e:discr to usual rep}
\qLisse^{\on{Weil}}(X)^\heartsuit \to \Rep(\on{Weil}^{\on{discr}}(X))^\heartsuit.
\end{equation}

The functor \eqref{e:discr to usual rep} is fully faithful, and its essential image
consists of representations, on which the action of the subgroup
$$\on{Gal}^{\on{discr}}(X)\subset \on{Weil}^{\on{discr}}(X)$$
is locally finite and as such corresponds to a continuous action of $\on{Gal}(X)$;
in other words, the resulting object of $\Rep(\on{Gal}^{\on{discr}}(X))^\heartsuit$ corresponds to 
an object of $\Rep(\on{Gal}^{\on{alg}}(X))^\heartsuit$. 

\sssec{}

Let $\LS^{\on{arithm,discr}}_\sG(X)$ be the stack
$$\bMaps_{\on{groups}}(\on{Weil}^{\on{discr}}(X),\sG)/\on{Ad}(\sG).$$

\medskip

The functor \eqref{e:discr to usual rep} gives rise to a map 
\begin{equation} \label{e:discr to usual LS}
^{\on{cl}}\LS_\sG^{\on{arithm}}(X)\to {}^{\on{cl}}\!\LS^{\on{arithm,discr}}_\sG(X).
\end{equation} 

The following is straightforward:

\begin{lem} \label{l:discr closed emb} \hfill

\smallskip

\noindent{\em(a)}
The map \eqref{e:discr to usual LS} is a closed embedding. 

\smallskip

\noindent{\em(b)} An $\sfe$-point of $\LS^{\on{arithm,discr}}_\sG(X)$ is the image of an $\sfe$-point of $\LS_\sG^{\on{arithm}}(X)$
if and only if the corresponding homomorphism $\on{Weil}^{\on{discr}}(X)\to \sG$ extends to a homomorphism
$\on{Weil}^{\on{alg}}(X)\to \sG$.

\end{lem}

\sssec{}

Denote
$$\on{Exc}(X,\sG)^{\on{discr}}:=\Gamma(\LS^{\on{arithm,discr}}_\sG(X),\CO_{\LS^{\on{arithm,discr}}_\sG(X)}).$$

\medskip

Note that $\LS^{\on{arithm,discr}}_\sG(X)$ is the quotient of an affine scheme (of infinite type)
by the action of $\sG$. In particular, $\on{Exc}(X,\sG)^{\on{discr}}$ is connective. 

\medskip

Denote
$$\LS^{\on{arithm,discr,coarse}}_\sG(X):=\Spec(\on{Exc}(X,\sG)^{\on{discr}}).$$

\sssec{}

The map \eqref{e:discr to usual LS} gives rise to a map
\begin{equation}  \label{e:map of exc algebras}
H^0(\on{Exc}(X,\sG)^{\on{discr}})\to H^0(\on{Exc}(X,\sG)).
\end{equation} 

Note, however, that since $\on{Exc}(X,\sG)^{\on{discr}}$ is connective and 
$\on{Exc}(X,\sG)$ is classical, the datum of a map as in \eqref{e:map of exc algebras} is 
equivalent to that of a map
$$\on{Exc}(X,\sG)^{\on{discr}}\to \on{Exc}(X,\sG).$$

Equivalently, we obtain a map
\begin{equation} \label{e:coarse map}
\LS^{\on{arithm,coarse}}_\sG(X)\to \LS^{\on{arithm,discr,coarse}}_\sG(X).
\end{equation} 

From \lemref{l:discr closed emb}(a), we obtain: 

\begin{cor} \label{l:exc surj}
The map \eqref{e:coarse map} is a closed embedding. 
\end{cor}

\begin{proof}

Follows from the fact that the map in \eqref{e:discr to usual LS} arises from a map
between affine schemes, by passing to stack-theoretic quotients by the reductive group $\sG$.

\end{proof}

\sssec{} \label{sss:pts coarse}

Recall now (see \cite[Lemma 11.9]{VLaf}) that the tautological projection
$$\LS^{\on{arithm,discr}}_\sG(X)\to \LS^{\on{arithm,discr,coarse}}_\sG(X)$$
gives rise to a bijection between:

\begin{itemize}

\item Isomorphism classes of \emph{semi-simple} homomorphisms $\on{Weil}^{\on{discr}}(X)\to \sG$, and 

\item $\sfe$-points of $\LS^{\on{arithm,discr,coarse}}_\sG(X)$.

\end{itemize}

\medskip

Similarly, the projection
$$\LS^{\on{arithm}}_\sG(X)\to \LS^{\on{arithm,coarse}}_\sG(X)$$
gives rise to a bijection between:

\begin{itemize}

\item Isomorphism classes of \emph{semi-simple} homomorphisms $\on{Weil}^{\on{alg}}(X)\to \sG$, and 

\item $\sfe$-points of $\LS^{\on{arithm,coarse}}_\sG(X)$.

\end{itemize}

\sssec{}

Let
$$'\!\LS^{\on{arithm,coarse}}_\sG(X)\subset \LS^{\on{arithm,discr,coarse}}_\sG(X)$$
be the Zariski closure of the subset of $\sfe$-points that correspond to \emph{continuous}
semi-simple homomorphisms $\on{Weil}^{\on{alg}}(X)\to \sG$.

\medskip

We claim:

\begin{prop} \label{p:coarse via discr}
The map \eqref{e:coarse map} factors via an isomorphism
$$\LS^{\on{arithm,coarse}}_\sG(X)\overset{\sim}\to {}'\!\LS^{\on{arithm,coarse}}_\sG(X).$$
\end{prop}

\begin{proof}

First we show that the map \eqref{e:coarse map} factors via $'\!\LS^{\on{arithm,discr}}_\sG(X)$.
Since $\LS^{\on{arithm,coarse}}_\sG(X)$ is reduced, it is enough to show that the image
of the set of $\sfe$-points of $\LS^{\on{arithm,coarse}}_\sG(X)$ under \eqref{e:coarse map}  
is contained in the set of $\sfe$-points of $'\!\LS^{\on{arithm,coarse}}_\sG(X)$. However, 
this follows from the description of the former set in \secref{sss:pts coarse} and
\lemref{l:discr closed emb}(b).

\medskip

In order to show that the resulting map
$$\LS^{\on{arithm,coarse}}_\sG(X)\to {}'\!\LS^{\on{arithm,coarse}}_\sG(X)$$
is an isomorphism, it is enough to show that it is surjective at the level of $\sfe$-points.
However, this again follows from \secref{sss:pts coarse} and
\lemref{l:discr closed emb}(b).

\end{proof} 

\ssec{Excursion vs Hecke operators}

\sssec{}

Fix a representation $\sG\to GL_n$, i.e., an object $V\in \Rep(\sG)^{\heartsuit,c}$. 

\medskip

To a conjugacy class 
$\gamma\in \on{Weil}(X)$ we attach an element
$$\on{exc}(V,\gamma)^{\on{discr}}\in H^0(\on{Exc}(X,\sG)^{\on{discr}})$$
whose value on $\sigma\in \bMaps_{\on{groups}}(\on{Weil}^{\on{discr}}(X),\sG)$ equals $\Tr(\sigma(\gamma),V)$.

\medskip

Let 
$$\on{exc}(V,\gamma)\in \on{Exc}(X,\sG)\simeq \Gamma( \LS^{\on{arithm}}_\sG(X),\CO_{\LS^{\on{arithm}}_\sG(X)})$$ 
denote the image of $\on{exc}(V,\gamma)^{\on{discr}}$ under \eqref{e:map of exc algebras}. 

\sssec{}

Let $x'$ be a closed point of $X$, and let $\Frob_{x'}\in \on{Weil}(X)$ be the conjugacy class
of the corresponding Frobenius element. We will use the notation
$$\on{H}(V,x')^{\on{discr}}:=\on{exc}(V,\Frob_{x'})^{\on{discr}} \text{ and }
\on{H}(V,x'):=\on{exc}(V,\Frob_{x'}).$$

The elements 
$$\on{H}(V,x')\in \on{Exc}(X,\sG)$$
are called the \emph{Hecke operators}. 

\sssec{}

The assignment
$$V\rightsquigarrow \on{H}(V,x')$$
defines a homomorphism
\begin{equation} \label{e:Hecke homomorphism}
\CH_{x'}(\sG)\to \on{Exc}(X,\sG),
\end{equation} 
where
$$\CH_{x'}(\sG):=\sfe\underset{\BZ}\otimes K^0(\Rep(\sG)^{\heartsuit,c})\simeq \CO_{\sG/\!/\on{Ad}(\sG)}.$$

\medskip

The homomorphisms \eqref{e:Hecke homomorphism} combine to a homomorphism
\begin{equation} \label{e:Hecke Homomorphism}
\CH_X(\sG):=\underset{x'\in |X|}\bigotimes\, \CH_{x'}(\sG)\to \on{Exc}(X,\sG).
\end{equation} 

\sssec{}

Assume for a moment that $\sG=GL_n$ and that $V$ is the tautological representation. In this case
we denote
$$\on{exc}(V,\gamma)^{\on{discr}}=:\on{exc}(\gamma)^{\on{discr}}_{\on{taut}} \text{ and } \on{exc}(V,\gamma)=:\on{exc}(\gamma)_{\on{taut}}.$$

The following is well-known (see, e.g., \cite{Pro}):

\begin{lem} \label{l:Tr for GL_n}
The elements $\on{exc}(\gamma)^{\on{discr}}_{\on{taut}}$, $\gamma\in \on{Weil}(X)$ generate $H^0(\on{Exc}(X,\sG)^{\on{discr}})$.
\end{lem} 

\begin{cor} \label{c:Tr for GL_n}
The elements $\on{exc}(\gamma)_{\on{taut}}$, $\gamma\in \on{Weil}(X)$ generate $\on{Exc}(X,\sG)$.
\end{cor} 

\sssec{} \label{sss:exc op}

Let $\sG$ again be arbitrary. Let us be given a data of $(I,V_I,\gamma_I,v_I,\xi_I)$, where:

\begin{itemize}

\item $I$ is a finite set;

\item $V_I$ is a finite-dimensional representation of $\sG^I$;

\item $\gamma_I$ is an $I$-tuple of elements of $\on{Weil}(X)$ (defined up to simultaneous conjugation); 

\item $v_I$ is a vector in $V^I$, invariant with respect to the diagonal copy of $\sG$;

\item $\xi_I$ is a covector on $V^I$, invariant with respect to the diagonal copy of $\sG$.

\end{itemize}

To this data we assign an element
$$\on{exc}(V_I,\gamma_I,v_I,\xi_I)^{\on{discr}}\in H^0(\on{Exc}(X,\sG)^{\on{discr}})$$ 
as follows:

\medskip

The value of $\on{exc}(V_I,\gamma_I,v_I,\xi_I)^{\on{discr}}$ on $\sigma\in \bMaps_{\on{groups}}(\on{Weil}^{\on{discr}}(X),\sG)$
the scalar $\xi(\gamma_I(v_I))$.

\medskip

Let $\on{exc}(V_I,\gamma_I,v_I,\xi_I)$ denote the image of $\on{exc}(V_I,\gamma_I,v_I,\xi_I)^{\on{discr}}$ under the homomorphism
\eqref{e:map of exc algebras}. 

\medskip

The elements $\on{exc}(V_I,\gamma_I,v_I,\xi_I)$ are called the \emph{excursion operators}. 

\begin{rem}
Note that the elements $\on{exc}(V,\gamma)^{\on{discr}}$ (resp., $\on{exc}(V,\gamma)$) are a particular case of 
$\on{exc}(V_I,\gamma_I,v_I,\xi_I)^{\on{discr}}$ (resp., $\on{exc}(V_I,\gamma_I,v_I,\xi_I)$), where:

\begin{itemize}

\item $I=\{1,2\}$;

\item $V_I=V\otimes V^\vee$;

\item $\gamma_I=(\gamma,1)$;

\item $v_I$ is the unit vector in $V\otimes V^\vee$;

\item $\xi_I$ is the counit covector.

\end{itemize}

\end{rem} 

\sssec{}

The following is established in \cite[Sect. 11]{VLaf} (see also \cite[Sect. 2]{GKRV}):

\begin{prop} \label{p:exc generate}
The elements $\on{exc}(V_I,\gamma_I,v_I,\xi_I)^{\on{discr}}$ span $H^0(\on{Exc}(X,\sG)^{\on{discr}})$.
\end{prop}

\begin{cor}
The elements $\on{exc}(V_I,\gamma_I,v_I,\xi_I)$ span $\on{Exc}(X,\sG)$.
\end{cor}

\ssec{Finiteness and generation properties}

\sssec{}

The goal of the rest of this section is to prove the following:

\begin{thm} \label{t:Hecke} \hfill

\medskip

\noindent{\em(a)} For any $x'\in X$, the homomorphism 
$$\CH_{x'}(\sG)\overset{\text{\eqref{e:Hecke homomorphism}}}\longrightarrow \on{Exc}(X,\sG)$$
makes $\on{Exc}(X,\sG)$ into a \emph{finitely generated} $\CH_{x'}(\sG)$-module.

\medskip

\noindent{\em(b)} Let $\sG=\GL_n$. Then the homomorphism 
$$\CH_X(\sG)\overset{\text{\eqref{e:Hecke Homomorphism}}}\longrightarrow  \on{Exc}(X,\sG)$$
is surjective. 

\end{thm}

\sssec{}

Let us first show how point (a) of \thmref{t:Hecke} implies point (b). Indeed, given point (a), it suffices to prove the following:

\medskip

Given a homomorphism $\CH_X(\sG)\overset{\ul\lambda}\longrightarrow\sfe$, the map
$$\sfe\to \sfe\underset{\CH_X(\sG)}\otimes \on{Exc}(X,\sG)$$
is surjective. 

\medskip

By \corref{c:Tr for GL_n}, it suffices to show that for every $\gamma\in \on{Weil}(X)$, the image of 
$\on{exc}(\gamma)_{\on{taut}}$ in 
\begin{equation} \label{e:coarse fiber}
\sfe\underset{\CH_X(\sG)}\otimes \on{Exc}(X,\sG)
\end{equation}
is constant. 

\sssec{}

Since $\LS^{\on{arithm}}_\sG(X)$ is the quotient of an affine scheme by an action of a reductive group, the algebra 
\eqref{e:coarse fiber} is the algebra of global functions on the fiber product
\begin{equation} \label{e:Hecke fiber}
\on{pt}\underset{\Spec(\CH_X(\sG))}\times \LS^{\on{arithm}}_\sG(X).
\end{equation}

Hence, it suffices to show that for every $\gamma\in \on{Weil}(X)$, the function on 
\eqref{e:Hecke fiber} obtained by restriction from 
$$\on{exc}(\gamma)_{\on{taut}}\in  \on{Exc}(X,\sG)\simeq \Gamma( \LS^{\on{arithm}}_\sG(X),\CO_{\LS^{\on{arithm}}_\sG(X)})$$
is constant.

\sssec{}

Since $\LS^{\on{arithm}}_\sG(X)$ is an algebraic stack locally of finite type, it suffices to establish the following two properties
for any $\gamma\in \on{Weil}(X)$: 

\medskip

\noindent{(i)} For a pair of $\sfe$-points $\sigma_1$ and $\sigma_2$ of \eqref{e:Hecke fiber}, the values of 
$\on{exc}(\gamma)_{\on{taut}}$ at these points are equal;

\medskip

\noindent{(ii)} For a local Artinian $\sfe$-algebra $R$ and an $R$-point $\sigma_R$ of \eqref{e:Hecke fiber}, 
the pullback of $\on{exc}(\gamma)_{\on{taut}}$ along $\sigma_R$ belongs to $\sfe\subset R$.

\sssec{}

We can think of $\sigma_i$ as a homomorphism
$$\phi_i:\on{Weil}(X)\to GL_n(\sfe),$$
whose restriction to $\on{Gal}(X)$ is \emph{continuous}.

\medskip

Similarly, we can think of $\sigma_R$ as a homomorphism
$$\phi_R:\on{Weil}(X)\to GL_n(R),$$
whose restriction to $\on{Gal}(X)$ is \emph{continuous} (for the natural topology on $GL_n(R)$). 

\medskip

The value of $\on{exc}(\gamma)_{\on{taut}}$ on $\sigma_i$ is $\Tr(\phi_i(\gamma))$, and the pullback of 
$\on{exc}(\gamma)_{\on{taut}}$ along $\sigma_R$ is $\Tr_{R}(\phi_R(\gamma))$. Thus, we need to show:

\medskip

\noindent{(i')} $\Tr(\phi_1(\gamma))=\Tr(\phi_2(\gamma))$, and 

\medskip

\noindent{(ii')} $\Tr_{R}(\phi_R(\gamma))\in \sfe\subset R$. 

\sssec{} \label{sss:Cheb}

As as a warm-up we will now prove assertions (i') and (ii') in the special case when the above homomorphisms
extend to \emph{continuous} homomorphisms from $\on{Gal}^{\on{arithm}}(X)$. 

\medskip

By Chebotarev's density, we can find a sequence of points $\gamma'_\alpha$ that converge to $\gamma$ in the topology
of $\on{Gal}^{\on{arithm}}(X)$, such that the conjugacy class of each $\gamma'_\alpha$ is of the form $\Frob_{x'_\alpha}$
for some $x'_\alpha\in |X|$.

\medskip

By assumption, the values of $\Tr(\phi_i(\gamma))\in \sfe$ ($i=1,2$) is the limit of $\Tr(\phi_i(\gamma'_\alpha))$ in the topology of $\sfe$.
Similarly, the value of $\Tr_{R}(\phi_R(\gamma))\in R$ is the limit of $\Tr_{R}(\phi_R(\gamma'_\alpha))$
in the topology of $R$.

\medskip

The fact that $\sigma_1$ and $\sigma_2$ both map to \eqref{e:Hecke fiber} implies that
$$\Tr(\phi_1(\gamma'_\alpha))=\Tr(\phi_2(\gamma'_\alpha)), \quad \forall \alpha.$$

Hence,
$$\Tr(\phi_1(\gamma))=\Tr(\phi_2(\gamma)),$$
as required. 

\medskip

The above establishes property (i'). For property (ii') we note that
$$\Tr_{R}(\phi_R(\gamma'_\alpha))\in \sfe\subset R$$
and hence the same is true for $\Tr_{R}(\phi_R(\gamma))$. 

\ssec{Proof of Property (i')} 

The argument below was explained to us by V.~Lafforgue.  It is an upgrade to Weil sheaves of the well-known argument
(see \cite[Theorem 1.1.2]{Laum}) that the map from the Grothendieck group of \'etale sheaves to $\underset{n}\Pi\, \sFunct(X(\BF_{q^n}),\sfe)$
is injective. 

\sssec{} \label{sss:mu i}

We will think of $\phi_i$ as an $n$-dimensional representation $V_i$ of $\on{Weil}(X)$; let
$$\underset{a}\oplus\, V_{i,a}$$ 
be the direct sum of its irreducible subquotients. 

\medskip

For each $a$ we can find an element $\mu_{i,a}\in \sfe^\times$, such that the representation
$$V_{i,a}\otimes \chi_{\mu_{i,a}},$$
extends to a continuous representation of $\on{Gal}^{\on{arithm}}(X)$, where $\chi_{\mu_{i,a}}$ is the
character of $\on{Weil}(X)$ that factors as
$$\on{Weil}(X)\to \BZ \overset{1\mapsto \mu_{i,a}}\longrightarrow \sfe^\times.$$

\medskip

For each value $r$ of the norm, set
$$V_{i,r}:=\underset{a,\, |\mu_{i,a}|=r}\oplus\, V_{i,a}.$$

\medskip 

It suffices to show that for every $r$,
$$\Tr(\gamma,V_{1,r})=\Tr(\gamma,V_{2,r}).$$

\sssec{}

Choose some $\mu_r$ with $|\mu_r|=r$, and set
$$V^c_{i,r}:=V_{i,r}\otimes \chi_{\mu_r}.$$

It suffices to show that for every $r$, 
\begin{equation} \label{e:Tr gamma r}
\Tr(\gamma,V^c_{1,r})=\Tr(\gamma,V^c_{2,r}).
\end{equation} 

Note, however, that each $V^c_{i,r}$ extends to a continuous representation of $\on{Gal}^{\on{arithm}}(X)$. 
Hence, by the argument in \secref{sss:Cheb}, in order to prove \eqref{e:Tr gamma r}, it suffices to show that
$$\Tr(\Frob_{x'},V^c_{1,r})=\Tr(\Frob_{x'},V^c_{2,r}), \quad \forall \,x'\in |X|.$$

The latter is equivalent to 
\begin{equation} \label{e:Tr Frob r}
\Tr(\Frob_{x'},V_{1,r})=\Tr(\Frob_{x'},V_{2,r}), \quad \forall \,x'\in |X|.
\end{equation} 

\sssec{}

For $x'\in |X|$, let $\lambda_{i,a,x',k}$, $k=1,...,\dim(V_{i,a})$ denote the eigenvalues of $\Frob_{x'}$ on $V_{i,a}$.
Denote $$\lambda^c_{i,a,x',k}:=\lambda_{i,a,x',k}\cdot \mu^{\deg(x')}_{i,a}.$$

\medskip

The fact that $\sigma_1$ and $\sigma_2$ both map to \eqref{e:Hecke fiber} implies that 
\begin{equation} \label{e:sum eigenvalues}
\underset{a}\Sigma\, \underset{k}\Sigma\, (\lambda^c_{1,a,x',k})^m\cdot \mu_{1,a}^{-m\cdot \deg(x')}
=\underset{a}\Sigma\, \underset{k}\Sigma\, (\lambda^c_{2,a,x',k})^m\cdot \mu_{2,a}^{-m\cdot \deg(x')}, \quad \forall m\in \BZ^{\geq 1}.
\end{equation} 

This implies that the sets 
$$\{\lambda^c_{1,a,x',k}\cdot \mu^{-\deg(x')}_{1,a}\} \text{ and } \{\lambda^c_{2,a,x',k}\cdot \mu^{-\deg(x')}_{2,a}\},$$
counted with multiplicities, agree.

\medskip

Note that the elements $\lambda^c_{i,a,x',k}$ have norm $1$ in $\sfe$. Hence, we obtain 
that for each value $r$ of the norm, the sets
$$\{\lambda^c_{1,a,x',k}\cdot \mu^{-\deg(x')}_{1,a}\,|\, |\mu_{1,a}|=r\} \text{ and } \{\lambda^c_{2,a,x',k}\cdot \mu^{-\deg(x')}_{2,a}\,|\, |\mu_{2,a}|=r\}$$
also agree. And hence
$$\underset{a,\, |\mu_{i,a}|=r}\Sigma\, \underset{k}\Sigma\, (\lambda^c_{1,a,x',k})^m\cdot \mu_{1,a}^{-m\cdot \deg(x')}
=\underset{a,\, |\mu_{i,a}|=r}\Sigma\, \underset{k}\Sigma\, (\lambda^c_{2,a,x',k})^m\cdot \mu_{2,a}^{-m\cdot \deg(x')}, \quad \forall m\in \BZ^{\geq 1}.$$

The latter equality means that for every $r$, 
$$\Tr(\Frob_{x'}^m,V_{1,r})=\Tr(\Frob_{x'}^m,V_{2,r}), \quad  \forall m\in \BZ^{\geq 1}.$$

Taking $m=1$, the latter equality gives the desired equality \eqref{e:Tr Frob r}. 

\qed[Property (i')]

\ssec{Proof of Property (ii)}

\sssec{}

Let us regard $\phi_R$ as a representation $V_R$ of $\on{Weil}(X)$ over the ring $R$; let $V$ be its reduction modulo $m_R$,
the maximal ideal of $R$.

\medskip

By the argument in \cite[Sect. 25.3]{AGKRRV}, we can equip $V$ with a filtration 
$$0=V_0\subset V_1\subset ....\subset V_{k-1}\subset V_k=V,$$
such that the deformation
$$V\rightsquigarrow V_R$$
extends to a filtered deformation
$$V_i \rightsquigarrow V_{i,R}.$$

This allows us to replace $V$ by 
$$\on{gr}(V):=\underset{i}\oplus\, V_i/V_{i-1}=:\underset{i}\oplus\, V'_i,$$
and $V_R$ by 
$$\underset{i}\oplus\, V_{i,R}/V_{i-1,R}=:\underset{i}\oplus\, V'_{i,R}.$$

\medskip

Denote the corresponding $R$-point of $\LS^{\on{arithm}}_\sG(X)$ by $\sigma'_R$. This point factors through 
\eqref{e:Hecke fiber}, and we can replace the original $\sigma_R$ by $\sigma'_R$. 


\sssec{}

Let $\mu_i$ be as in \secref{sss:mu i}, i.e., 
$$V'_i\otimes \chi_{\mu_i}$$
extends to a continuous representation of $\on{Gal}^{\on{arithm}}(X)$.

\medskip

For every value $r$ of the norm, set 
$$V'_{r,R}:=\underset{i,\,  |\mu_i|=r}\oplus\, V'_{i,R}.$$

Denote
$$n_r:=\underset{i,\,  |\mu_i|=r}\Sigma\, \dim(V'_i).$$

We obtain that $\sigma'_R$ factors through an $R$-point of $\LS^{\on{arithm}}_{\sM}(X)$, where $\sM$ is the Levi subgroup of $\sG=GL_n$ 
$$\sM:=\underset{r}\Pi\, GL(n_r).$$

\medskip

Note now that the map
$$\Spec(\CH_X(\sM))\to \Spec(\CH_X(\sG))$$ 
is formally \'etale at the image of the point
$$\sigma':=\sigma'_R|_{\on{pt}}$$
under $\LS^{\on{arithm}}_{\sM}(X)\to \Spec(\CH_X(\sM))$. 

\medskip

This implies that $\sigma'_R$ lifts to a point of 
$$\on{pt}\underset{\Spec(\CH_X(\sM))}\times \LS^{\on{arithm}}_\sM(X).$$

This allows us to reduce the proof to the case when only one value of $r$ appears. 

\sssec{}

Let $\mu_r$ be such that $|\mu_r|=r$. For $i$ with $|\mu_i|=r$, 
denote 
$$V'_i{}^c:=V'_i\otimes \chi_{\mu_r}.$$

Note that $V'_i{}^c$ also extends to a continuous representation of $\on{Gal}^{\on{arithm}}(X)$. 
Hence, so does
$$V'_{r,R}\otimes \chi_{\mu_r}.$$

\medskip

Thus, translating by $\mu_r$ on $\LS^{\on{arithm}}_\sG(X)$ and $\Spec(\CH_X(\sG))$, we can assume that 
$V'_{r,R}$ is continuous. This reduces us to the case considered in \secref{sss:Cheb}. 

\qed[Property (i'')]

\ssec{Proof of \thmref{t:Hecke}(a)}

\sssec{}

Since $\LS^{\on{arithm}}_\sG(X)$ has finitely many connected components, it suffices to show that for every connected
component $\LS^{\on{arithm}}_\sG(X)_\alpha$, the algebra $\on{Exc}(X,\sG)_\alpha$ is finite as a module over $\CH_{x'}(\sG)$. 

\medskip

By \thmref{t:exc} and \eqref{e:structure stabilizer}, it suffices to show that the map
$$\sF_\alpha/\!/\on{Ad}(\sG_\alpha)\to \sG/\!/\on{Ad}(\sG)$$ 
is finite.

\sssec{}

Choose a point $\sg\in \sF_\alpha$. This point normalizes $\sG_\alpha$, and we can identify
$$\sF_\alpha\simeq \sG_\alpha,$$
so that the adjoint action of $\sG_\alpha$ on $\sF_\alpha$ corresponds to the $\on{Ad}(\sg)$-twisted
conjugation action of $\sG_\alpha$ on itself.

\medskip

This, we have to show that the map
\begin{equation} \label{e:Vinberg map}
\sG_\alpha/\on{Ad}_\sg(\sG_\alpha)\to \sG/\!/\on{Ad}(\sG)
\end{equation} 
is finite.

\medskip

In fact, we will prove a more general group-theoretic statement, which we will now formulate.

\newcommand{\vphi}{\phi}
\newcommand{\Ad}{\mathrm{Ad}}
\newcommand{\sslash}{/\!/}

\sssec{} Let $\vphi$ be an automorphism of a (possibly disconnected) reductive group $G$ which preserves a maximal torus $T \subseteq G$. Then we obtain an action of $\vphi$ on the weight lattice $X^*(T)$ of $T$. Suppose that $\vphi$ acts on $X^*(T)$ with finite order. Then we claim that for any $n \in N(T)$, the automorphism $\on{Ad}_n \circ \vphi$ acts on $X^*(T)$ with finite order as well. To see this, let $W$ denote $N(T)/T$, let $W_0$ denote the image of $W$ in $\Aut(T)$, and let $\langle\phi\rangle$ denote the subgroup of $\Aut(T)$ generated by $\phi$. If $\vphi$ preserves $T$, then it automatically preserves $N(T)$, so the action of $W_0$ on $X^*(T)$ extends to an action of 
\[W_0 \rtimes \left\langle\vphi\right\rangle\]
on $X^*(T)$, and since both $W_0$ and $\left\langle\vphi\right\rangle$ are finite, any element of this group acts on $X^*(T)$ with finite order. 

\medskip

Thus, if $\vphi$ is an automorphism of a reductive group $G$, we will say that $\vphi$ acts with finite order on $X^*(T)$ if for some/any $g \in G$ such that $\on{Ad}_g \circ \vphi$ preserves $T$, the resulting automorphism of $X^*(T)$ has finite order. 

\begin{prop}\label{prop:finiteness}
Let $\vphi$ be an automorphism of a reductive group $G$ such that $\vphi$ acts with finite order on $X^*(T)$. 
Then if $H$ is any reductive subgroup of $G$ preserved by $\vphi$, the map 
\[H \sslash \Ad_\vphi(H) \to G \sslash \Ad_\vphi(G)\]
is finite. 
\end{prop}

\sssec{}

Clearly, \propref{prop:finiteness} contains the assertion that \eqref{e:Vinberg map} is finite as a particular case.

\qed[\thmref{t:Hecke}(a)]

\ssec{Proof of \propref{prop:finiteness}}

\sssec{} Note that when $\vphi$ is the identity automorphism of $G$, \propref{prop:finiteness} recovers \cite[Theorem 1]{Vin} in the case when $p = 1$. 

\medskip
 
After some preliminary remarks, we will prove \propref{prop:finiteness} in five steps, starting with the assumption that $G$ is connected and $H = T$ is a maximal torus of $G$, and then gradually relaxing these assumptions.   

\sssec{}

Suppose that $G$ is a connected reductive group with an automorphism $\vphi$. Recall the Peter-Weyl theorem  
\[\CO_G \simeq \bigoplus_\sigma V_\sigma \otimes V_\sigma^\vee \in \Rep(G \times G).\]
If $G$ acts on itself by $\vphi$-twisted conjugacy, then we obtain instead 
\[\CO_G \simeq \bigoplus_\sigma V_\sigma \otimes V_{\vphi(\sigma)}^\vee.\]
It follows from Schur's lemma that 
\[\CO_{G\sslash \Ad_\vphi(G)} \simeq \bigoplus_{\sigma = \vphi(\sigma)} k\]
where for each $\sigma$ that is isomorphic to its image under $\vphi$, the corresponding copy of $k$ is spanned by the function 
\[g \mapsto \on{tr}_V(\alpha \cdot g)\]
for any $\phi$-equivariance data $\alpha$ for $V$. 

\sssec{}

Suppose that $T$ is a subtorus of $G$ which is preserved by $\vphi$. Applying the discussion of the previous paragraph, we have
\begin{equation}\label{eq:torus-quot-funcs}
\CO_{T\sslash \Ad_\vphi(T)} \simeq \bigoplus_{\substack{\lambda \in X^*(T) \\
 \lambda = \vphi(\lambda)}} k    
\end{equation}
If $V$ is any irreducible representation of $G$ with $V \simeq V^\vphi$, we observe that the function 
\[t \mapsto \on{tr}_V(\alpha \cdot t)\]
only receives contributions from the weights in $V$ which are stable under the action of $\vphi$. 

\medskip 

We now begin the proof of \propref{prop:finiteness}. 

\sssec{Step 1}

$G$ is connected and $H = T$ is a maximal torus of $G$. We want to show that the map $T \sslash \Ad_\vphi(T) \to G \sslash \Ad_\vphi(G)$ is finite. Because it is affine, it suffices to show that it is proper. We will prove this using the valuative criterion. Suppose that $R$ is a valuation ring with fraction field $K$, and suppose that we have a commutative diagram:
\begin{equation}\label{eq:val-criterion-diag}
\CD
\CO_{G \sslash \Ad_\vphi(G)} @>>> R \\
@VVV @VVV \\
\CO_{T \sslash \Ad_\vphi(T)} @>>> K
\endCD   
\end{equation}
We need to show that the image of every element of $\CO_{T\sslash \Ad_\vphi(T)}$ has positive valuation in $K$. By (\ref{eq:torus-quot-funcs}) it suffices to show this for the $\vphi$-invariant characters of $T$. Let $\lambda$ be any $\vphi$-invariant character of $T$ and let $V$ be a $\phi$-equivariant representation of $G$ with a nonzero $\lambda$ weight-space.\footnote{For example, if $w \cdot \lambda$ is dominant then we can take $V$ to be the irreducible representation with highest weight $w\cdot \lambda$} To each weight $\mu$ appearing in $V$, we assign the following numbers:
\begin{itemize}
    \item Let $n_\mu$ denote the size of the orbit of $\mu$ under the action of $\vphi$ on the weights of $V$.
    \item Let $v_\mu$ denote $\frac{1}{n_\mu}$ times the valuation of $\prod_{i = 1}^{n_\mu} \vphi^i(\mu)$ in $K$.
\end{itemize}   
Let $v_0$ denote the minimum of $v_\mu$ as $\mu$ ranges over all of the weights in $V$. To show that $v_\lambda \ge 0$, it suffices to show that $v_0 \ge 0$.  

\medskip

To see that $v_0 \ge 0$, let $r$ denote the dimension of the subspace of $V$ spanned by the weight spaces $V_\mu$ with $v_\mu = v_0$. Because $V$ is $\phi$-equivariant, the representation $\Lambda^r V$ also admits a natural $\phi$-equivariance structure, so we may consider its character $\chi_{\Lambda^r V} \in \CO_{G \sslash \Ad_\phi(G)}$. As a $G$-representation, $\Lambda^r V$ contains a unique $\vphi$-invariant weight $\lambda_0$ with valuation $v_0r$, the $\lambda_0$-weight space has multiplicity 1, and $v_0r$ is \textit{strictly} smaller than the valuation of any other $\vphi$-invariant weight. By the non-archimedean property of the valuation on $K$, it follows that 
\[\on{val}_K(\chi_{\Lambda^r V}) = \on{val}_K(\lambda_0).\]

\medskip

Because $\chi_{\Lambda^r V}$ is an element of $\CO_{G \sslash \Ad_\phi(G)}$, the commutativity of \eqref{eq:val-criterion-diag} implies that $\on{val}_K(\chi_{\Lambda^r V}) \ge 0$. As a result, 
\[v_0r = \on{val}_K(\lambda_0) = \on{val}_K(\chi_{\Lambda^r V}) \ge 0,\]
so $v_0 \ge 0$, which is what we needed to show. 

\sssec{Step 2}

$G$ is connected and $H = T$ is an arbitrary subtorus of $G$. Let $M = Z_G(T)$ denote the Levi subgroup of $G$ associated to $T$. It is also preserved by $\vphi$. Let $T'$ denote a maximal torus inside of $M$ (which necessarily contains $T$). By modifying $\vphi$ by conjugation by an element of $M$, we may assume that $\vphi$ preserves $T'$ as well. We have a short exact sequence 
\[0 \to \Lambda \to X^*(T') \to X^*(T) \to 0\]
where $\Lambda$ is a finitely generated free abelian group. The snake lemma identifies the cokernel of $X^*(T')^\vphi \to X^*(T)^\vphi$ with a subgroup of $\Lambda_\vphi$, so the cokernel must be finitely generated. Because the action of $\vphi$ on $X^*(T')$ has finite order, it follows that the rationalization of the cokernel of $X^*(T')^\vphi \to X^*(T)^\vphi$ vanishes, so the cokernel itself must be a finite abelian group. This shows that in the diagram 
\[T\sslash \Ad_\vphi(T) \to T' \sslash \Ad_\vphi(T') \to G \sslash \Ad_\vphi(G)\]
the first arrow is finite. Since the second arrow is finite by Step 1, this case is finished. 

\sssec{Step 3}

Now suppose that $G$ is arbitrary and $H = T$ is an arbitrary subtorus of $G$. Then $T \sslash \Ad_\vphi(T) \to G \sslash \Ad_\vphi(G)$ factors as 
\[T \sslash \Ad_\vphi(T) \to G^\circ \sslash \Ad_\vphi(G^\circ) \to G^\circ \sslash \Ad_\vphi(G) \to G \sslash \Ad_\vphi(G).\]

\sssec{Step 4}

Now suppose that $G$ is arbitrary and $H$ is connected. By modifying $\vphi$ by conjugation by an element of $H$, we may assume that $\vphi$ preserves a maximal torus $T$ of $H$. By (\ref{eq:torus-quot-funcs}) and highest weight theory for $H$, we see that the map $T \sslash \Ad_\vphi(T) \to H \sslash \Ad_\vphi(H)$ is dominant. Since 
\[T \sslash \Ad_\vphi(T) \to G \sslash \Ad_\vphi(G)\]
is finite, it follows formally that 
\[H \sslash \Ad_\vphi(H) \to G \sslash \Ad_\vphi(G)\]
is, too.

\sssec{Step 5}

Now suppose that both $G$ and $H$ are arbitrary. For any $x \in H$, the component $H^\circ \cdot  x$ of $H$ is preserved by the 
$\vphi$-twisted adjoint action of $H^\circ$, since 
\[h \cdot H^\circ \cdot  x \cdot  \vphi(h^{-1}) = h \cdot  H^\circ \cdot  \Ad_x(\vphi(h^{-1}))\cdot  x.\]
This also shows the map 
\begin{equation}\label{eq:component}
H^\circ \cdot x \sslash \Ad_\vphi(H^\circ) \to G \sslash \Ad_\vphi(G)    
\end{equation}
identifies with 
\[H^\circ\sslash \Ad_{\Ad_x \circ \vphi}(H^\circ) \to G \sslash \Ad_{\Ad_x \circ \vphi}(G)\] 
so by Step 4 we know that (\ref{eq:component}) is finite (note that if $\phi$ acts on $X^*(T)$ with finite order then the same is true of $\Ad_x \circ \phi$). Since $H \sslash \Ad_\vphi(H^\circ)$ is a finite disjoint union of components 
$H^\circ \cdot x \sslash \Ad_\vphi(H^\circ)$, it follows that 
\[H \sslash \Ad_\vphi(H^\circ) \to G \sslash \Ad_\vphi(G)\] 
is finite. Since 
\[H \sslash \Ad_\vphi(H^\circ) \to H\sslash \Ad_\vphi(H)\] 
is dominant, the result follows. 

\qed[\propref{prop:finiteness}]

\section{Rationality and independence of \texorpdfstring{$\ell$}{l} of the excursion algebra} \label{s:rat}

In this section we show that the excursion algebra $\on{Exc}(X,\sG)$ is defined over $\BQ$ and, in particular, 
is independent of $\ell$. This will be based on the ``independence of $\ell$" results from the paper \cite{Dr}. 

\ssec{Statement of the result} 


\sssec{}

We will prove:

\begin{thm} \label{t:rat} 
There exists a canonically defined $\BQ$-algebra $\on{Exc}(X,\sG)_\BQ$ that 
provides a rational model for $\on{Exc}(X,\sG)$ for any $\ell$. Moreover, 
the elements $\on{H}(V,x')$ for $V\in \Rep(\sG)^{\heartsuit,c}$, $x'\in |X|$ are rational. 
\end{thm}

Note that the second assertion in \thmref{t:rat} means that the homomorphism
$$\CH_X(\sG)\overset{\text{\eqref{e:Hecke Homomorphism}}}\to \on{Exc}(X,\sG)$$
is defined over $\BQ$.

\medskip

From here we obtain:

\begin{cor} \label{c:rat}
The algebra $\CH_{X,\BQ}(\sG)$ admits a quotient $\CH^{\on{glob}}_{X,\BQ}$ such that for any $\sfe$, 
the homomorphism \eqref{e:Hecke Homomorphism} factors as 
$$\CH_X(\sG)=\sfe\underset{\BQ}\otimes \CH_{X,\BQ}(\sG)\to \sfe\underset{\BQ}\otimes \CH^{\on{glob}}_{X,\BQ}\hookrightarrow \on{Exc}(X,\sG).$$
\end{cor} 

\sssec{}

Assume for a moment that $\sG=\GL_n$. Combining \corref{c:rat} and \thmref{t:Hecke}(b)
we obtain:

\begin{cor} \label{c:rat GLn}
The algebra $\CH_{X,\BQ}(\sG)$ admits a quotient $\CH^{\on{glob}}_{X,\BQ}$, such that the homomorphism \eqref{e:Hecke Homomorphism} factors as 
$$\CH_X(\sG)=\sfe\underset{\BQ}\otimes \CH_{X,\BQ}(\sG)\to \sfe\underset{\BQ}\otimes \CH^{\on{glob}}_{X,\BQ}(\sG)\simeq \on{Exc}(X,\sG).$$
\end{cor}

\begin{rem} \label{r:RINO}

Note that \corref{c:rat GLn} gives a rather robust characterization of the rational structure on $\on{Exc}(X,\sG)$ for $\sG=GL_n$; indeed, $\on{Exc}(X,\sG)_\BQ$
is the quotient $\CH^{\on{glob}}_{X,\BQ}$ of $\CH_{X,\BQ}(\sG)$. 

\medskip

However, for a general $\sG$, since an assertion parallel to \thmref{t:Hecke}(b) fails, \thmref{t:rat}
is a rather weak statement\footnote{It can be called ``rationality in name only".}. 

\end{rem} 

\sssec{} 

Let $X$ be a curve, and let $\ol{X}$ be its compactification.  One can show that V.~Lafforgue's construction 
(with an extension by \cite{Xue} from cuspidal to all automorphic functions plus an input from \cite{AGKRRV})
can be upgraded to
an action of $\on{Exc}(X,\sG)$ on the space
$$\on{Autom}:=\on{Funct}_c(\Bun_G^{\on{level}}(\BF_q),\sfe),$$
where:

\begin{itemize}

\item $G$ is the Langlands dual of $\sG$;

\medskip

\item $\Bun_G^{\on{level}}$ is the moduli space of $G$-bundles on $\ol{X}$ with level structure on $\ol{X}-X$, whose
depth depends on the amount of ramification we allow for our choice of $\qLisse(X)$ (see \secref{sss:ram}).

\end{itemize} 

\medskip

Following \cite[Conjecture 12.12]{VLaf} we propose:

\begin{conj} \label{conj:rat}
The action of $\on{Exc}(X,\sG)$ on $\on{Autom}$ is compatible with the rational structures. 
\end{conj}

We do not know how to approach this conjecture for an arbitrary $\sG$ (see Remark \ref{r:RINO}),  
since, on the one hand, \thmref{t:rat} does not give a unique characterization of the rational 
structure, and on the other hand, its proof is rather inexplicit (yet, at the end of the day, it follows from 
Drinfeld's theorem in \cite{Dr}, which is deduced by comparison with the automorphic side).

\medskip

However, we claim:

\begin{thm}
The statement of \conjref{conj:rat} holds for $\sG=GL_n$.
\end{thm}

\begin{proof}

Follows from the fact that the action of $\CH_X(\sG)$ on $\on{Autom}$ is compatible with the rational structures.

\end{proof} 

\begin{rem}
An assertion parallel to \conjref{conj:rat} holds in the context of geometrization of the local Langlands correspondence
(\`a la Fargues-Scholze). This has been established in \cite{Sch} using motivic methods. 

\medskip

One of the key points of this
paper is that when studying motivic local systems in the local situation, it is sufficient to consider \emph{mixed Tate} 
motives, for which one knows the existence of the t-structure, etc. We do not know how to adopt these methods to
the context of a global curve.
\end{rem}

\ssec{The pro-reductive quotient of the Weil group} \label{ss:Drinf}

The rest of this section is devoted to the proof of \thmref{t:rat}. 

\sssec{} \label{sss:Weil red}

Recall the group $\on{Weil}^{\on{alg}}(X)$, see \eqref{e:alg Weil}. Note that the subgroup
$$\on{Gal}^{\on{unip}}(X)\subset  \on{Gal}^{\on{alg}}(X)\subset \on{Weil}^{\on{alg}}(X)$$
is normal. Consider the short exact sequence
\begin{equation} \label{e:Weil red almost}
1\to \on{Gal}^{\on{red}}(X)\to \on{Weil}^{\on{alg}}(X)/\on{Gal}^{\on{unip}}(X)\to \BZ\to 1.
\end{equation} 

Let us denote by $\on{Weil}^{\on{red}}(X)$ the push-out of \eqref{e:Weil red almost} with respect to
$$\on{Gal}^{\on{red}}(X)\to \on{Gal}^{\on{red,relev}}(X),$$
i.e.,
\begin{equation} \label{e:Weil red}
1\to \on{Gal}^{\on{red,relev}}(X)\to \on{Weil}^{\on{red}}(X)\to \BZ\to 1.
\end{equation} 

\medskip

The group $\on{Weil}^{\on{red}}(X)$ comes equipped with a homomorphism
\begin{equation} \label{e:discr to Drinf'}
\on{Weil}^{\on{discr}}(X)\to \on{Weil}^{\on{alg}}(X)\to \on{Weil}^{\on{red}}(X).
\end{equation} 

\sssec{}

By \thmref{t:exc}, for the proof of \thmref{t:rat}, we can replace
$$\on{Exc}(X,\sG) \rightsquigarrow \on{Exc}(X,\sG)^0:=
\Gamma(\LS^{\on{arithm},0}_\sG(X),\CO_{\LS^{\on{arithm},0}_\sG(X)}).$$

\medskip

Note that the stack $\LS^{\on{arithm},0}_\sG(X)$ can be identified with 
$$\bMaps_{\on{groups}}(\on{Weil}^{\on{red}}(X),\sG)/\on{Ad}(\sG),$$
see \secref{sss:arithm 0}. 

\medskip

Hence, we can interpret $\on{Exc}(X,\sG)^0$ as
$$\Gamma(\bMaps_{\on{groups}}(\on{Weil}^{\on{red}}(X),\sG)/\on{Ad}(\sG),
\CO_{\bMaps_{\on{groups}}(\on{Weil}^{\on{red}}(X),\sG)/\on{Ad}(\sG)}).$$

\sssec{}

Let $\on{Gal-Drinf}^{\on{arithm}}(X)$ be the pro-semi-simple group over $\sfe$ defined in \cite[Sect. 1.1.2]{Dr} (in {\it loc. cit.} it is denoted 
$\hat\Pi_\lambda$). 
It comes equipped with a map
$$\on{Gal-Drinf}^{\on{arithm}}(X)\to \wh\BZ,$$
where the pro-finite group $\wh\BZ$ is viewed as a pro-semi-simple group.

\medskip

Denote
$$\on{Weil-Drinf}(X):=\on{Gal-Drinf}^{\on{arithm}}(X)\underset{\wh\BZ}\times \BZ.$$

By construction, the group $\on{Weil-Drinf}(X)$ comes equipped with a homomorphism
\begin{equation} \label{e:discr to Drinf}
\on{Weil}^{\on{discr}}(X)\to \on{Weil-Drinf}(X).
\end{equation} 

\sssec{}

We claim:

\begin{propconstr} \label{p:Drinf}
There is a canonical isomorphism
$$\on{Weil}^{\on{red}}(X)\simeq \on{Weil-Drinf}(X),$$
which makes the diagram
$$
\CD
\on{Weil}^{\on{discr}}(X) @>{\on{id}}>> \on{Weil}^{\on{discr}}(X)  \\
@V{\text{\eqref{e:discr to Drinf'}}}VV @VV{\text{\eqref{e:discr to Drinf}}}V \\
\on{Weil}^{\on{red}}(X) @>{\sim}>>  \on{Weil-Drinf}(X)
\endCD
$$
commute.
\end{propconstr}

The proposition will be proved in \secref{ss:proof of Drinf}. 

\sssec{}

By \propref{p:Drinf}, we have:
\begin{equation} \label{e:LS arithm 0 again}
\LS^{\on{arithm},0}_\sG(X)\simeq \bMaps_{\on{groups}}(\on{Weil-Drinf}(X),\sG)/\on{Ad}(\sG),
\end{equation}
compatibly with maps to 
$$\bMaps_{\on{groups}}(\on{Weil}^{\on{discr}}(X),\sG)/\on{Ad}(\sG).$$

\medskip

Therefore, for the proof of \thmref{t:rat}, we can further replace
\begin{multline*} 
\on{Exc}(X,\sG)^0 \rightsquigarrow \on{Exc-Drinf}(X,\sG)^0:= \\
=\Gamma(\bMaps_{\on{groups}}(\on{Weil-Drinf}(X),\sG)/\on{Ad}(\sG),\CO_{\bMaps_{\on{groups}}(\on{Weil-Drinf}(X),\sG)/\on{Ad}(\sG)}).
\end{multline*} 

\ssec{Proof of \propref{p:Drinf}} \label{ss:proof of Drinf}

\sssec{}

We will construct an equivalence of symmetric monoidal categories
$$\Rep(\on{Weil-Drinf}(X))\simeq \Rep(\on{Weil}^{\on{red}}(X)) ,$$
compatible with the forgetful functors to $\Rep(\on{Weil}^{\on{discr}}(X))$.

\sssec{}

Let 
$$\qLisse^{\on{Weil,ss},0}(X)^\heartsuit\subset \qLisse^{\on{Weil,0}}(X)^\heartsuit$$
be the full \emph{abelian} subcategory consisting of semi-simple objects, and let
$$\qLisse^{\on{Weil,ss},0}(X):=D(\qLisse^{\on{Weil,ss},0}(X)^\heartsuit).$$

\sssec{}

Let
$$\qLisse^{\on{Weil,ss,fin.det}}(X)\subset \qLisse^{\on{Weil,ss},0}(X)$$
be the full subcategory, we impose the further condition that the irreducible constituents are 
local systems with a finite order determinant. This is a Tannakian category thanks to \cite[Proposition 3.6.1]{Dr}.

\medskip

Unwinding, we obtain:
\begin{equation} \label{e:Weil drinf fin det} 
\Rep(\on{Weil-Drinf}(X))\simeq \BZ\mod\underset{\wh\BZ\mod}\otimes \qLisse^{\on{Weil,ss,fin.det}}(X),
\end{equation} 
where we note that
$$\wh\BZ\mod\simeq \qLisse^{\on{Weil,ss,fin.det}}(\on{pt}).$$

\sssec{}

Note, however, that the naturally defined functor
$$\BZ\mod_{\on{ss},0} \underset{\hat\BZ\mod}\otimes \qLisse^{\on{Weil,ss,fin.det}}(X)\to \qLisse^{\on{Weil,ss},0}(X)$$
is an equivalence,
where
$$\BZ\mod_{\on{ss},0}:=\qLisse^{\on{Weil,ss},0}(\on{pt}).$$

Hence, we can rewrite
$$\Rep(\on{Weil-Drinf}(X))\simeq  \BZ\mod\underset{\BZ\mod_{\on{ss},0}}\otimes \qLisse^{\on{Weil,ss},0}(X).$$

\sssec{}

Note that we have:
\begin{multline} \label{e:Weil red mod} 
\Rep(\on{Weil}^{\on{red}}(X)) \simeq (\Rep(\on{Gal}^{\on{red,relev}}(X)))^{\Frob}\overset{\text{\propref{p:Gal red}}}\simeq \\
\simeq (\Rep(\on{Gal}^{\on{red,relev},0}(X)))^{\Frob}\simeq 
(\qLisse^{\on{relev},0}(X))^{\Frob} \overset{\text{\eqref{e:strange category}}}\simeq 
\BZ\mod \underset{\BZ\mod_0}\otimes \qLisse^{\on{Weil},0}(X). 
\end{multline} 

\medskip

Hence, in order to prove \propref{p:Drinf}, it remains to establish an equivalence
\begin{equation} \label{e:qLisse 0 ss}
\BZ\mod_0\underset{\BZ\mod_{\on{ss},0}}\otimes \qLisse^{\on{Weil,ss},0}(X)\simeq \qLisse^{\on{Weil},0}(X),
\end{equation}
as symmetric monoidal categories, compatibly with the forgetful functors to $\Rep(\on{Weil}^{\on{discr}}(X))$.

\sssec{}

Note that we have a canonically defined functor\footnote{Note that the functor \eqref{e:Dr to all} is no longer fully faithful.}
\begin{equation} \label{e:Dr to all}
\qLisse^{\on{Weil,ss},0}(X)\to \qLisse^{\on{Weil},0}(X).
\end{equation}

The functor \eqref{e:Dr to all} gives rise to a functor $\to$ in \eqref{e:qLisse 0 ss}. Let us show that the latter 
is an equivalence.

\medskip

By \corref{c:eq de-eq W}, it suffices to show that the induced functor
\begin{multline} \label{e:qLisse 0 ss bis}
\Vect_\sfe\underset{\BZ\mod_{\on{ss},0}}\otimes \qLisse^{\on{Weil,ss},0}(X)\simeq 
\Vect_\sfe\underset{\BZ\mod_0}\otimes 
\BZ\mod_0\underset{\BZ\mod_{\on{ss},0}}\otimes \qLisse^{\on{Weil,ss},0}(X)\to \\
\to \Vect_\sfe\underset{\BZ\mod_0}\otimes \qLisse^{\on{Weil},0}(X)\simeq \qLisse^{\on{relev},0}(X)
\end{multline} 
is an equivalence. 

\medskip

However, the latter is obvious: both sides of \eqref{e:qLisse 0 ss bis} are semi-simple, and the functor
\eqref{e:qLisse 0 ss bis} induces a bijection on the set of isomorphism classes of objects. 

\qed[\propref{p:Drinf}]

\begin{rem}

Note that combining the equivalence in \eqref{e:Weil drinf fin det} and that of \propref{p:Drinf}, we 
obtain an equivalence 
\begin{equation} 
\Rep(\on{Weil}^{\on{red}}(X)) \simeq \BZ\mod\underset{\wh\BZ\mod}\otimes \qLisse^{\on{Weil,ss,fin.det}}(X).
\end{equation}

In particular, we obtain that the action of the group $\BZ^{\on{alg,wt}}$ on $\Rep(\on{Weil}^{\on{red}}(X))$ canonically 
factors via
\begin{equation} \label{e:factor Z hat}
\BZ^{\on{alg,wt}}\twoheadrightarrow \pi_0(\BZ^{\on{alg,wt}})\simeq \wh\BZ.
\end{equation} 

Hence, we obtain that the action of $\BZ^{\on{alg,wt}}$ on
$$\LS_\sG^{\on{restr}} \text{ and } \LS_\sG^{\on{arithm}}$$
(which the proof of \thmref{t:exc} reiled on), factors through \eqref{e:factor Z hat}.

\medskip

We note the following:

\medskip

\begin{enumerate}

\item The action of $\wh\BZ$ on $\Gamma(\LS_\sG^{\on{arithm}},\CO_{\LS_\sG^{\on{arithm}}})$ is trivial;

\medskip

\item The action of $\BZ$ on $\LS_\sG^{\on{arithm}}$ is trivial;

\medskip

\item One can show that the action of $\wh\BZ$ on $\LS_\sG^{\on{arithm}}$ is \emph{non}-trivial.

\end{enumerate}

\end{rem}

\ssec{Interlude: the ``motivic" version}

In this subsection we will comment on the relation between the Galois groups appearing above and another object from \cite{Dr}.

\sssec{}

Let $\on{Gal}^{\on{red,arithm}}(X)$ be the semi-simple envelope of the arithmetic Galois group of $X$, i.e., 
$$\on{Gal}^{\on{red,arithm}}(X):= \on{Gal-Drinf}^{\on{arithm}}(X).$$

We have a short exact sequence
$$1\to \on{Gal}^{\on{red,relev}}(X) \to \on{Gal}^{\on{red,arithm}}(X)\to \wh\BZ\to 1.$$

\sssec{}

Consider the pullback
\begin{equation} \label{e:Gal mot}
\on{Gal}^{\on{mot}}(X):=
\on{Gal}^{\on{red,arithm}}(X)\underset{\wh\BZ}\times \BZ^{\on{red,wt}},
\end{equation} 
where $\BZ^{\on{red,wt}}$ is the maximal reductive quotient of $\BZ^{\on{alg,wt}}$. We have
$$\on{Weil}^{\on{red}}(X)\simeq \on{Gal}^{\on{mot}}(X)\underset{\BZ^{\on{red,wt}}}\times \BZ.$$

\medskip

Then $\on{Gal}^{\on{mot}}(X)$ is isomorphic to the group denoted by $\wh\Pi^{\on{mot}}_\lambda$ in \cite[Sect. 6.1]{Dr}.
 
\medskip

Note that we have:
$$\Rep(\on{Gal}^{\on{mot}}(X)) \simeq \Rep(\BZ^{\on{red,wt}})\underset{\Rep(\BZ^{\on{red,0}})}\otimes \qLisse^{\on{Weil,ss},0}(X),$$
where 
$$\Rep(\BZ^{\on{red,0}})\simeq \BZ\mod_{\on{ss},0}.$$

\sssec{}

Consider the stack
$$\LS^{\on{mot}}_\sG(X):=\bMaps_{\on{groups}}(\on{Gal}^{\on{mot}}(X),\sG)/\on{Ad}(\sG).$$

Note, however, that since $\on{Gal}^{\on{mot}}(X)$ is pro-reductive, the stack $\LS^{\on{mot}}_\sG(X)$ is the union
of (infinitely many) connected components of the form $\on{pt}/H$, where $H$ is reductive group. 

\medskip

For every finite sub-union
$$'\!\LS^{\on{mot}}_\sG(X)\subset \LS^{\on{mot}}_\sG(X),$$
set
$$'\!\on{Exc}(X,\sG)^{\on{mot}}:=\Gamma({}'\!\LS^{\on{mot}}_\sG(X),\CO_{'\!\LS^{\on{mot}}_\sG(X)})$$
and 
$$'\!\LS_\sG^{\on{mot,coarse}}(X):=\Spec({}'\!\on{Exc}(X,\sG)^{\on{mot}}).$$

Set
$$\LS_\sG^{\on{mot,coarse}}(X):=\cup\, '\!\LS_\sG^{\on{mot,coarse}}(X)$$

Thus, $\LS_\sG^{\on{mot,coarse}}(X)$ is (an infinite) union of point-schemes.

\begin{rem}

The above definition of $\LS_\sG^{\on{mot,coarse}}(X)$ coincides with one in \cite[Sect. 6.6]{Dr}.

\end{rem}

\sssec{}

The homomorphism
$$\on{Weil}^{\on{red}}(X)\to \on{Gal}^{\on{mot}}(X)$$
gives rise to a map 
\begin{equation} \label{e:mot to arithm}
\LS^{\on{mot}}_\sG(X)\to \LS^{\on{arithm},0}_\sG(X).
\end{equation} 

\medskip

This map is a closed embedding when restricted to every finite union $'\!\LS^{\on{mot}}_\sG(X)$. 

\medskip

This map in turns gives rise to a surjection
\begin{multline*} 
\on{Exc}(X,\sG)\simeq \on{Exc}(X,\sG)^0=\Gamma(\LS^{\on{arithm},0}_\sG(X),\CO_{\LS^{\on{arithm},0}_\sG(X)})\to \\
\to \Gamma({}'\!\LS^{\on{mot}}_\sG(X),\CO_{'\!\LS^{\on{mot}}_\sG(X)})={}'\!\on{Exc}(X,\sG)^{\on{mot}},
\end{multline*}
and hence to a closed embedding
$$'\!\LS_\sG^{\on{mot,coarse}}(X)\hookrightarrow \LS_\sG^{\on{arithm,coarse}}(X).$$

\medskip

These embeddings combine to a map
\begin{equation} \label{e:emb coarse mot}
\LS_\sG^{\on{mot,coarse}}(X)\hookrightarrow \LS_\sG^{\on{arithm,coarse}}(X).
\end{equation}

The image of \eqref{e:emb coarse mot} is the (disjoint) union of $\sfe$-points of $\LS_\sG^{\on{arithm,coarse}}(X)$
that correspond to those semi-simple Weil $\sG$-local systems on $X$ that are \emph{weakly motivic} in sense of \cite[Proposition 6.2.1]{Dr}, i.e. 
for every representation $\sG\to GL_n$, the corresponding object of $\qLisse^{\on{Weil,loc.fin}}(X)^\heartsuit$
belongs to $\qLisse^{\on{Weil,wt}}(X)^\heartsuit$. 

\begin{rem}

The image for the terminology ``weakly motivic" is the following: 

\medskip

By \cite{LLaf}, when $X$ is a curve, every irreducible object in $\qLisse^{\on{Weil,ss,fin.det}}(X)^\heartsuit$, 
up to a half-integral Tate twist, is a direct summand of a local system that appears in the cohomology of a smooth and proper 
direct image $Y\to X$ of the constant sheaf on $Y$. 

\medskip

A further modification of the action of the Frobenius by a Weil number stays motivic, by Honda's theorem. 

\end{rem} 

\sssec{}

Once \thmref{t:rat} is proved, the same argument shows that the (ind)-affine scheme $\LS_\sG^{\on{mot,coarse}}(X)$ is also defined
over $\BQ$, as well as the map \eqref{e:emb coarse mot}. 

\medskip

The image of the map \eqref{e:emb coarse mot} is called the locus of \emph{motivic} local systems.

\begin{rem}

If in \eqref{e:Gal mot} instead of $\BZ^{\on{red,wt}}$ we took $\BZ^{\on{red}}$ 
(the reductive quotient of $\BZ^{\on{alg}}$), we would obtain a moduli stack, whose 
coarse moduli space is still a disjoint union of points, and its image in $\LS_\sG^{\on{arithm,coarse}}(X)$
is now \emph{all} $\sfe$-points of $\LS_\sG^{\on{arithm,coarse}}(X)$.

\medskip

Note, however, that an analog of \thmref{t:rat} does \emph{not} apply here, since $\BZ^{\on{red}}$, unlike $\BZ^{\on{red,wt}}$,
is \emph{not} defined over $\BQ$.

\end{rem}

\begin{rem}

An intermediate option is, instead of $\BZ^{\on{red,wt}}$ or $\BZ^{\on{red}}$, to take the diagonalizable group, 
whose characters are \emph{integral units} in $\sfe$. In this way, we obtain the pro-reductive group that is denoted $\wh\Pi^{\on{red}}_\lambda$ in 
\cite[Sect. 3.5]{Dr}.

\medskip

Note the resulting coarse moduli space is still \emph{not} defined over $\BQ$, for the same reason as for $\BZ^{\on{red}}$. 

\end{rem}

\ssec{Drinfeld's category}

\sssec{} \label{sss:ext pro Red}

For a field $\bk$ of characteristic $0$, let $\on{Pro-AlgGrp}(\bk)$ be the category of affine group-schemes
(a.k.a., pro-algebraic groups) over $\bk$. 

\medskip

Let $\on{Pro-AlgGrp}^\coarse(\bk)$ denote the (ordinary) category of group schemes over $\bk$ ``up to conjugation by their neutral components''. 
So the objects of $\on{Pro-AlgGrp}^\coarse(k)$ are the same as those of $\on{Pro-AlgGrp}(\bk)$, and 
\[\Hom_{\on{Pro-AlgGrp}^\coarse(\bk)}(G, H) = \Hom_{\on{Pro-AlgGrp}^\coarse(\bk)}(G, H)/H^\circ(\bk)\]
where an element $h \in H^\circ(\bk)$ acts on $\Hom_{\on{Pro-AlgGrp}(\bk)}(G, H)$ by  
\[f \mapsto \Ad_h \circ f.\]

\sssec{}

Let 
$$\on{Pro-ss}(\bk)\subset \on{Pro-red}(\bk)\subset \on{Pro-AlgGrp}(\bk)$$ and 
$$\on{Pro-ss}^\coarse(\bk)\subset \on{Pro-red}^\coarse(\bk)\subset \on{Pro-AlgGrp}^\coarse(\bk)$$
be the full subcategories, where we restrict objects to be pro-semi simple (resp., pro-reductive) 
algebraic groups. 

\medskip

Note that $\on{Pro-ss}^\coarse(\bk)$ and $\on{Pro-red}^\coarse(\bk)$ are the same categories that 
have been introduced in  \cite[Sect. 1.2.3]{Dr} (NB:  in {\it loc. cit.} they are denoted $\on{Pro-ss}(\bk)$ 
and $\on{Pro-red}(\bk)$, respectively). 

\medskip

In this subsection we will only need the underlying groupoids of $\on{Pro-ss}^\coarse(\bk)$ and $\on{Pro-red}^\coarse(\bk)$,
which we denote by the same symbol, by a slight abuse of notation. 

\sssec{}

Let $\on{Pro-red-ext}^{\on{coarse}}(\bk)$ be a similarly defined groupoid, where we consider extensions 
\begin{equation} \label{e:ext pro Red}
1\to H\to H'\to \BZ\to 1,
\end{equation} 
where $H$ is a pro-reductive group.

\sssec{}

As is explained in \cite{Dr}, to an object $H\in \on{Pro-red}^\coarse(\bk)$, we can attach a well-defined scheme
$$H/\!/\on{Ad}(H^\circ),$$
where $H^\circ$ is the connected component of $H$. 

\medskip

Similarly, to an object $H'\in \on{Pro-red-ext}^{\on{coarse}}(\bk)$, we can attach a well-defined scheme
$$H'/\!/\on{Ad}(H^\circ).$$

\sssec{}

Consider the functor from $\on{Pro-red-ext}(\bk)$ to the category of (possibly infinite unions of) affine schemes
\begin{equation}  \label{e:coarse}
H'\mapsto  \bMaps_{\on{groups}}(H',\sG)/\!/\on{Ad}(\sG),
\end{equation} 

\begin{lem} \label{l:coarse}
The functor \eqref{e:coarse} factors via a functor
$$\on{Pro-red-ext}^{\on{coarse}}(\bk)\to \affSch_{/\bk}.$$
\end{lem} 

\begin{cor} \label{c:coarse}
The functor 
\begin{equation}  \label{e:coarse alf}
H'\mapsto \on{Exc}(H',\sG):=
\Gamma(\bMaps_{\on{groups}}(H',\sG)/\on{Ad}(\sG),\CO_{\bMaps_{\on{groups}}(H',\sG)/\on{Ad}(\sG)})
\end{equation} 
from the category of extensions \eqref{e:ext pro Red} to $\on{ComAlg}(\Vect_{\bk})$ factors via a functor
\begin{equation}  \label{e:coarse fact}
\on{Pro-red-ext}^{\on{coarse}}(\bk)\to \on{ComAlg}(\Vect_{\bk}).
\end{equation} 
\end{cor} 

\begin{rem}

Note that the functor 
$$H'\mapsto \bMaps_{\on{groups}}(H',\sG)/\on{Ad}(\sG)$$
with values in stacks does \emph{not} factor via $\on{Pro-red-ext}^{\on{coarse}}(\bk)$.

\end{rem} 

\sssec{}

Recall now that according to \cite[Sect. 1.2.3]{Dr}, for a map of algebraically closed fields $\bk_1\to \bk_2$,
the functor
$$\on{Pro-red-ext}^{\on{coarse}}(\bk_1)\to \on{Pro-red-ext}^{\on{coarse}}(\bk_2)$$
is an equivalence.

\medskip

In particular, the groupoid $\on{Pro-red-ext}^{\on{coarse}}(\ol\BQ)$ carries a canonical action of $\on{Gal}(\ol\BQ/\BQ)$. Hence,
for an object of $\on{Pro-red-ext}^{\on{coarse}}(\ol\BQ)$, it makes sense to talk about structure of $\on{Gal}(\ol\BQ/\BQ)$-equivariance
on it. 

\sssec{}

We now recall the main result of the paper \cite{Dr}, namely, Theorem 1.4.1 in {\it loc. cit.}:

\begin{thm} \label{t:Dr}  
There exists a canonically defined object $\on{Gal-Drinf}^{\on{arithm}}(X)_{\ol\BQ}\in \on{Pro-red}^{\on{coarse}}(\ol\BQ)$, equipped with a structure of
$\on{Gal}(\ol\BQ/\BQ)$-equivariance, such that for any $\ell\neq \on{char}(k)$ and an embedding
$\ol\BQ\to \sfe$, 
the image of $\on{Gal-Drinf}^{\on{arithm}}(X)_{\ol\BQ}$ in $\on{Pro-red}^{\on{coarse}}(\sfe)$ identifies with the class of the extension 
$\on{Gal-Drinf}^{\on{arithm}}(X)$.\footnote{Strictly speaking, the result in \cite{Dr} concerns lisse sheaves on $X$ with no bound on their ramification at infinity. However, using \cite[Theorem 9.8]{De1half} as in \cite[Section 4]{EK}, one can show that the result also holds for lisse sheaves with bounded ramification.}
\end{thm} 

As a formal corollary, we obtain: 

\begin{cor} \label{c:Dr}
There exists an object $\on{Weil-Drinf}(X)_{\ol\BQ}\in \on{Pro-red-ext}^{\on{coarse}}(\ol\BQ)$, equipped with a structure of
$\on{Gal}(\ol\BQ/\BQ)$-equivariance, such that for any $\ell\neq \on{char}(k)$ and an embedding $\ol\BQ\to \sfe$, 
the image of $\on{Weil-Drinf}(X)_{\ol\BQ}$ in $\on{Pro-red-ext}^{\on{coarse}}(\sfe)$ identifies with the class of the extension 
$\on{Weil-Drinf}(X)$.
\end{cor}

\sssec{}

Let $x'$ be a closed point of $X$, and let $\Frob_{x'}\in \on{Weil}^{\on{discr}}(X)$ be the conjugacy class
of the corresponding Frobenius element. Let $\ol\Frob_{x',\sfe}$ denote the image of $\Frob_{x'}$ under the composition
\begin{multline*} 
\on{Weil}^{\on{discr}}(X)/\on{Ad}(\on{Weil}^{\on{discr}}(X))\to 
\on{Weil-Drinf}(X)/\on{Ad}(\on{Weil-Drinf}(X))\to \\
\to \on{Weil-Drinf}(X)/\!/\on{Ad}(\on{Weil-Drinf}(X)).
\end{multline*} 

\sssec{}

Another assertion of \cite[Theorem 1.4.1]{Dr} reads:

\begin{thm} \label{t:Dr bis}
For any $x'$ as above, there exists a $\on{Gal}(\ol\BQ/\BQ)$-invariant element 
$$\ol\Frob_{x',\ol\BQ}\in \on{Weil-Drinf}(X)_{\ol\BQ}/\!/\on{Ad}(\on{Weil-Drinf}(X)_{\ol\BQ})$$
such that for any embedding $\ol\BQ\to \sfe$ its image under
$$\on{Weil-Drinf}(X)_{\ol\BQ}/\!/\on{Ad}(\on{Weil-Drinf}(X)_{\ol\BQ})\to \on{Weil-Drinf}(X)/\!/\on{Ad}(\on{Weil-Drinf}(X))$$
equals $\ol\Frob_{x',\sfe}$.
\end{thm} 

\ssec{Proof of \thmref{t:rat}}

\sssec{}

Note that for a map of algebraically closed fields $\bk_1\to \bk_2$, the diagram
$$
\CD
\on{Pro-red-ext}^{\on{coarse}}(\bk_1) @>{\text{\eqref{e:coarse fact}}}>>  \on{ComAlg}(\Vect_{\bk_1}) \\
@VVV @VVV \\
\on{Pro-red-ext}^{\on{coarse}}(\bk_2) @>{\text{\eqref{e:coarse fact}}}>>  \on{ComAlg}(\Vect_{\bk_2})
\endCD
$$
commutes.

\medskip

Combined with \corref{c:Dr}, we obtain that 
$$\on{Exc-Drinf}(X,\sG)^0_{\ol\BQ}\in \on{ComAlg}(\Vect_{\ol\BQ}^\heartsuit)$$
carries a structure of $\on{Gal}(\ol\BQ/\BQ)$-equivariance.

\sssec{}

Furthermore, for a representation $V$ of $\sG$ defined over $\BQ$, the natural transformation of functors 
$$(H\mapsto H'/\!/\on{Ad}(H^\circ)) \Rightarrow (H\mapsto \on{Exc}(H',\sG)),$$
given by
$$H'/\!/\on{Ad}(H^\circ) \times \bMaps_{\on{groups}}(H',\sG)/\on{Ad}(\sG)\to \sG/\!/\on{Ad}(\sG) \overset{\on{Tr(-,V)}}\to \BA^1$$
is also compatible with maps between algebraically closed fields. 

\medskip

Combined with \thmref{t:Dr bis}, this implies that for any $x'\in X$, and $V\in \Rep(\sG)^{\heartsuit,c}$ defined over $\BQ$,
the image of the element $\on{H}(V,x')$ under
$$\on{Exc}(X,\sG)\to  \on{Exc}(X,\sG)^0 \simeq \on{Exc-Drinf}(X,\sG)^0$$
comes from a $\on{Gal}(\ol\BQ/\BQ)$-invariant element in $\on{Exc-Drinf}(X,\sG)^0_{\ol\BQ}$.

\sssec{}

Thus, we obtain that in order to prove \thmref{t:rat}, it remains to establish the following:

\begin{thm} \label{t:cont}
The action of $\on{Gal}(\ol\BQ/\BQ)$ on $\on{Exc-Drinf}(X,\sG)^0_{\ol\BQ}$ is continuous.
\end{thm}

The theorem will be proved in \secref{sss:proof-of-cont}. 

\sssec{}

Assume for a moment that $\sG=GL_n$. Then we claim that the statement of \thmref{t:cont} is automatic:

\medskip

Indeed, this follows from \thmref{t:Hecke}(b) and the fact that the elements $\on{H}(V,x')$ are 
$\on{Gal}(\ol\BQ/\BQ)$-invariant.

\subsection{The continuity theorem--preparations} \label{ss:cont}

\subsubsection{} A map of fields $\bk_1 \to \bk_2$ induces a base-change functor $\on{Pro-AlgGrp}(\bk_1) \to \on{Pro-AlgGrp}(\bk_2)$, which in turn 
descends to a functor $\on{Pro-AlgGrp}^\coarse(\bk_1) \to \on{Pro-AlgGrp}^\coarse(\bk_2)$. 
In particular, taking $\bk_1 = \bk_2 = \bk$, we see that $\Aut(\bk)$ (viewed as an abstract group) acts on $\on{Pro-AlgGrp}^\coarse(\bk)$. 

\medskip

If $\sigma$ is an element 
of $\Aut(\bk)$ we will denote the corresponding automorphism of $\on{Pro-AlgGrp}^\coarse(\bk)$ by $F_\sigma$. Note that any $\Aut(\bk)$-equivariant object of 
$\on{Pro-AlgGrp}(\bk)$ defines an $\Aut(\bk)$-equivariant object of $\on{Pro-AlgGrp}^\coarse(\bk)$. 

\sssec{}

Note that when $\bk$ is algebraically closed, \cite[Sect. 2.2.2]{Dr} gives an explicit description of (the underlying groupoid of) $\pss(\bk)$ in terms of (pro-)combinatorial data. 

\medskip

In terms of the notation of \textit{loc. cit.}, suppose we are given some object $G$ of $\pss(\bk)$ and some field $\bk_0 \subseteq \bk$ such that the map 
\[Z_\Delta(\bk) \to Z\]
is $\Aut(\bk/\bk_0)$-equivariant (where $Z$ is given the trivial $\Aut(\bk/\bk_0)$-action). Then the construction of \cite[Sect. 2.2.4]{Dr} 
produces a pro-semisimple group $G_0$ defined over $\bk_0$ whose base-change to $\bk$ is isomorphic to $G$. 
This endows $G$ with $\Aut(\bk/\bk_0)$-equivariance data which we call ``split'' equivariance data for $G$. 

\subsubsection{}\label{sss:actions-on-equiv} Since $GL_n$ is defined over $\BQ$, it acquires equivariance data for $\Aut(\bk)$ as an object of 
$\on{Pro-AlgGrp}^\coarse(\bk)$. As a result, if $G$ is any $\Aut(\bk)$-equivariant object of $\on{Pro-AlgGrp}^\coarse(\bk)$, we obtain an action of $\Aut(\bk)$ on 
$K^+(G)$, where the latter is the Grothendieck semi-ring of the small abelian category $\Rep(G)^{\heartsuit,c}$.

\medskip

Namely, if $\sigma$ is an element of $\Aut(\bk)$ and $\rho \in K^+(G)$ then $\sigma \cdot \rho$ is the isomorphism class of the representation 
\[G \totext{\sim} F_\sigma(G) \totext{F_\sigma(\rho)} F_\sigma(GL_n) \totext{\sim} GL_n.\]

\sssec{}

There is another piece of data that we can extract from $\Aut(\bk)$-equivariance for $G$: 

\medskip

If $\sigma$ is an element of $\Aut(\bk)$, then (for any $G$) there is a natural map 
\[(G \sslash \on{Ad}(G^\circ))(\bk) \to (F_\sigma(G) \sslash \on{Ad}(F_\sigma(G)^\circ))(\bk).\]

\medskip

Hence, if $G$ is $\Aut(\bk)$-equivariant then we obtain an action of $\Aut(\bk)$ on $(G \sslash \on{Ad}(G^\circ))(\bk)$. 

\sssec{}

These two pieces of data are compatible in the following way: if $G$ is any object of $\on{Pro-AlgGrp}^\coarse(\bk)$, there is a well-defined trace map 
\[\on{tr}: (G \sslash \on{Ad}(G^\circ))(\bk) \times K^+(G) \to \bk\]
and if $G$ is equipped with the data of $\Aut(\bk)$-equivariance, then the trace map intertwines the induced action of $\Aut(\bk)$ on $(G \sslash \on{Ad}(G^\circ))(\bk) \times K^+(G)$ with the canonical action of $\Aut(\bk)$ on $\bk$. 

\newcommand{\BQbar}{\overline{\mathbb{Q}}}
\newcommand{\Gal}{\on{Gal}}
\newcommand{\Fr}{\on{Frob}}
\newcommand{\GD}{\on{Gal-Drinf}^{\on{arithm}}(X)_{\ol\BQ}}
\newcommand{\tr}{\on{Tr}}

\subsubsection{}\label{sss:gal-action-on-gd}
According to \cite[Sect. 1.4.2]{Dr}, the object $G = \on{Gal-Drinf}^{\on{arithm}}(X)_{\ol\BQ}$ of $\pss(\BQbar)$ admits an equivariance datum for $\Gal(\BQ)$, which is uniquely characterized by the requirement that the Frobenius elements in $(G\sslash \on{Ad}(G^\circ))(\BQbar)$ are fixed by the action of $\Gal(\BQ)$. 

\medskip

This has the following representation-theoretic consequence: if $\Fr_x$ is any Frobenius element in $G$ and $\rho: G \to GL(V)$ is any element of $K^+(G)$, then 
\[\tr(\Fr_x, V_\sigma) = \sigma \cdot \tr(\Fr_x, V)\]
where $V_\sigma$ is the element of $K^+(G)$ obtained by applying the action from \secref{sss:actions-on-equiv}.

\begin{prop}\label{prop:gal-continuity}
Suppose that $H$ is a finite-type quotient of $\GD$. Then there is an open subgroup $U \subseteq \Gal(\BQ)$ such that 
\begin{itemize}
    \item $H$ has a split model over $\BQbar^U$.
    \item The map $\GD \to H$ is $U$-equivariant, where $H$ is equipped with the split equivariance structure.
\end{itemize} 
\end{prop}

\newcommand{\sfG}{\mathsf{G}}
\newcommand{\geom}{\mathrm{geom}}
\newcommand{\Z}{\mathbb{Z}}
\newcommand{\et}{\mathrm{\acute{e}t}}
\newcommand{\red}{\on{red}}
\newcommand{\relev}{\on{relev}}
\newcommand{\Ggeom}{\Gal^{\red, \relev}(X)}
\newcommand{\Garith}{\on{Gal-Drinf}^{\on{arithm}}(X)}

\subsubsection{}\label{sss:proof-of-cont} The proof of \propref{prop:gal-continuity} will occupy the next section of this paper. Taking the proposition for granted, we will now prove \thmref{t:cont}. 

\medskip

Let $\sigma: \Ggeom \to \sfG$ be a semisimple $\sfG$ local system on $X$, and assume that $\sigma$ admits a Weil structure. 
Pick a faithful representation $\sfG \to GL_n$. It is well known that there is an etale local system $\sigma': \Garith \to GL_n$ whose restriction to $\Ggeom$ agrees with our original composite 
\[\Ggeom \to \sfG \to GL_n.\]
Let $H$ denote the image of $\sigma'$. By \propref{prop:gal-continuity}, we know that there is an open subgroup $U$ of $\Gal(\BQ)$ such that $\Garith \to H$ is $U$-equivariant (where $H$ has the split equivariance structure). 

\medskip

Let $K$ denote the intersection of $\Ggeom$ with the kernel of $\Garith \to H$. This group is evidently preserved by the action of $U$. Additionally, because $\sfG \to GL_n$ was chosen to be injective, we know that $K = \ker(\sigma)$. 

\medskip

Let $H^\geom$ denote the quotient $\Ggeom/K$. We obtain a commutative diagram 
\[
\begin{tikzcd}
1 \arrow[r] &\Gal^{\red, \relev}(X) \arrow[d] \arrow[r] & \Garith \arrow[d] \arrow[r] & \hat{\Z} \arrow[d] \arrow[r] & 1\\
1 \arrow[r] & H^\geom \arrow[r] & H \arrow[r] & \Z/n \arrow[r] & 1
\end{tikzcd}
\]
where every map is $U$-equivariant. In particular, the group $H \underset{\Z/n}\times \Z$ acquires a natural (split) $U$-equivariance structure, for which the map 
$$\WD(X) \to H \underset{\Z/n}\times \Z$$ is $U$-equivariant. 

\medskip 

Since $K$ is the kernel of $\sigma$, any map $\WD(X) \to \sfG$ extending $\sigma$ factors through $$\WD(X)/K \simeq H \underset{\Z/n}\times \Z.$$
This implies that the restriction map 
\begin{equation}\label{eq:WD-fin-type-map}
\Maps(\WD(X)/K, \sfG)/\sfG \to \Maps(\WD(X), \sfG)/\sfG    
\end{equation}
induces an isomorphism on the component of $\on{LS}^{\on{arithm}, 0}$ corresponding to $\sigma$. 

\medskip

Since the map \eqref{eq:WD-fin-type-map} is $U$-equivariant and the $U$-equivariance data for $\WD(X)/K$ is split, it follows that the action of $U$ on the ring of functions of this component is continuous. 

\medskip

By \cite{EK}, $\on{LS}^{\on{arithm}, 0}$ only has finitely many connected components, so there is some open subgroup of $\Gal(\BQ)$ which acts continuously on the whole ring of functions of $\on{LS}^{\on{arithm}, 0}$. But this implies that the $\Gal(\BQ)$ action is itself continuous. 

\qed{[\thmref{t:cont}]}

\subsection{Proof of \propref{prop:gal-continuity}} \label{ss:cont-bis}

\sssec{}

We start with the following observation:

\begin{lem}\label{lem:surj-maps-equal-coarse}
Let $\bk$ be algebraically closed, and let $f_1, f_2: G \to H$ be surjective maps of pro-semisimple groups over $\bk$. Assume that:

\medskip

\begin{itemize}

\item The neutral component of $H$ is a product of almost simple groups;

\medskip

\item $f_1$ and $f_2$ induce the same map 
\[\pi_0(G) \to \pi_0(H);\]

\medskip

\item For every open subgroup $\Gamma_0 \subseteq \pi_0(H)$, they induce the same map 
\[K^+(H \underset{\pi_0(H)}\times \Gamma_0) \to K^+(G \underset{\pi_0(H)}\times \Gamma_0).\]

\end{itemize}

\medskip

Then $f_1 = f_2$ as morphisms in $\pss(\bk)$.     
\end{lem}

\begin{proof}
Because $f_1$ and $f_2$ induce the same map $K^+(H) \to K^+(G)$, we know by \cite[Lemma 4.1.1(i)]{Dr} that $\ker(f_1) = \ker(f_2)$; denote this group by $K$. 
Then $f_1$ and $f_2$ descend to two isomorphisms $G/K \to H$ which induce the same map on $\pi_0$. 

\medskip

Since the image of $K$ in $\pi_0(H)$ is trivial, we know that for every open subgroup $\Gamma_0 \subseteq \pi_0(H)$, 
\[K \subseteq G \underset{\pi_0(H)}\times \Gamma_0\] 
and so 
\[G/K \underset{\pi_0(H)}\times \Gamma_0 \simeq (G \underset{\pi_0(H)}\times \Gamma_0)/K.\]
This implies that $K^+(G/K \underset{\pi_0(H)}\times \Gamma_0)$ is a subring of $K^+(G \underset{\pi_0(H)}\times \Gamma_0)$, 
so by assumption $f_1$ and $f_2$ induce the same map 
\[K^+(H \underset{\pi_0(H)}\times \Gamma_0) \to K^+(G/K \underset{\pi_0(H)}\times \Gamma_0).\]

\medskip
 
Now the result follows by \cite[Theorem 4.3.7(i)]{Dr}.
\end{proof}

\subsubsection{Step 1} In this step we will prove \propref{prop:gal-continuity} under the additional assumption that $H^\circ$ is simply connected. 

\medskip

Since $H$ is finite-type over $\BQbar$, there is an open subgroup $U$ of $\Gal(\BQ)$ such that $H$ admits a split model over the fixed field of $U$. We claim that after passing to a further open subgroup $V$ of $U$ the map $\pi: G \to H$ becomes $V$-equivariant (where $H$ is equipped with the split equivariance datum).

\medskip

To see this, we will invoke \lemref{lem:surj-maps-equal-coarse}. For an element $\sigma \in U$, let $\pi': G \to H$ denote the composite 
\[G \totext{\sim} F_\sigma(G) \totext{F_\sigma(\pi)} F_\sigma(H) \totext{\sim} H.\]
It follows from Chebotarev density and the defining property of the $\Gal(\BQ)$-equivariance data on $G$ that $\pi$ and $\pi'$ induce the same map $\pi_0(G) \to \pi_0(H)$. 

\medskip

For a subgroup $\Gamma_0 \subseteq \pi_0(H)$, we will denote by the same symbols $\pi$ and $\pi'$ the corresponding maps 
\[G \underset{\pi_0(H)}\times  \Gamma_0 \to H \underset{\pi_0(H)}\times  \Gamma_0.\]

\medskip

By \lemref{lem:surj-maps-equal-coarse} it remains to show that there is an open subgroup $V$ of $U$ such that for all $\Gamma_0 \subseteq \pi_0(H)$, the maps 
\[K^+(H \underset{\pi_0(H)}\times  \Gamma_0) \to K^+(G \underset{\pi_0(H)}\times  \Gamma_0)\]
induced by $\pi$ and $\pi'$ agree for all $\sigma \in V$. Because $\pi_0(H)$ is finite, it suffices to work with a single $\Gamma_0$. For ease of notation, let $H_0$ denote $H \underset{\pi_0(H)}\times  \Gamma_0$ and let $G_0$ denote $G \underset{\pi_0(H)}\times  \Gamma_0$. 

By our choice of equivariance data for $H$, for any $\rho: H_0 \to GL_n$, the diagram 
\[
\begin{tikzcd}
F_\sigma(H) \arrow[r, "\sim"] \arrow[d, "F_\sigma(\rho)"'] & H \arrow[d, "\rho"]\\
F_\sigma(GL_n) \arrow[r, "\sim"] & GL_n
\end{tikzcd}
\]
commutes for any $\sigma$ in an open subgroup of $U$\footnote{Because $H_0$ and $GL_n$ are both finite type groups over the fixed field of $U$, any map $\rho$ between them is defined over a finite extension of this fixed field.}. It follows that $\rho \circ \pi' = \sigma \cdot (\rho \circ \pi)$ in $K^+(G_0)$. In other words, if we write $V$ for the class of $\rho \circ \pi$ in $K^+(G_0)$, we know that the class of $\rho \circ \pi'$ is equal to $V_\sigma$ (using the notation of \secref{sss:gal-action-on-gd}). 

\medskip

Now we note that since $\pi_0(G) \underset{\pi_0(H)}\times  \Gamma_0$ is an open subgroup of $\pi_0(G) = \pi_1^\et(X)$, there is a finite cover $\wt{X}$ of $X$ such that 
\[\pi_0(G) \underset{\pi_0(H)}\times  \Gamma_0 \simeq \pi_1^\et(\wt{X}).\]
By \cite[Lemma 3.1.2]{Dr}, $G_0$ identifies $\Gal(\BQ)$-equivariantly with the pro-semisimple completion of $\pi_1^\et(\widetilde{X})$. Since the trace map $K^+(G_0) \to \on{Fun}(\pi_1^\et(\widetilde{X})_{\Fr}; \BQbar)$ is injective, to prove that $V = V_\sigma$ it suffices to check that 
\[\tr(\Fr_x, V) = \tr(\Fr_x, V_\sigma) = \sigma \cdot \tr(\Fr_x, V)\]
for all $x \in |\widetilde{X}|$. 

\medskip

For any fixed $\rho$, \cite[Theorem 3.1]{De3} shows that this equality holds for any $\sigma$ in an open subgroup of $U$. 

\medskip

Since $H_0$ is finite-type, $K(H_0)$ is finitely generated, so to check that $\pi$ and $\pi'$ induce the same map $K^+(H_0) \to K^+(G_0)$ it suffices to check this condition for finitely many $\rho$.

\subsubsection{} We now wish to relax the assumption that $H^\circ$ is simply connected. To this end, we will prove the following lemma: 

\begin{lem}\label{lem:fin-type-quot-has-sc-cover}
Let $G$ be a pro-semisimple group whose neutral component is simply connected. Let $H$ be a finite-type quotient of $G$. Then there is a 
finite type quotient $H'$ of $G$ whose neutral component is simply connected such that the map $G \to H$ factors through $H'$.     
\end{lem}

\begin{proof}
Let $K$ denote the neutral component of the kernel of $G \to H$. Note that $K$ is a normal subgroup of $G$. Because $G^\circ$ was assumed to be simply connected, $G^\circ/K$ is the simply-connected cover of $H^\circ$. For every nonzero element $x \in G/K$, there is some intermediate finite-type quotient $G/K \to H_x \to H$ such that the image of $x$ is nonzero in $H_x$. Because the kernel of $G^\circ/K \to H^\circ$ is finite, there is some finite-type quotient $H'$ of $G/K$ such that $(H')^\circ = G^\circ/K$ is the simply-connected cover of $H$, and such that $G \to H$ factors through $H'$. 
\end{proof}

\subsubsection{Step 2} We now prove \propref{prop:gal-continuity} for any $H$.

\medskip

By \cite[Proposition 3.1.4]{Dr} and \lemref{lem:fin-type-quot-has-sc-cover}, we can find a quotient $H'$ of $G$ such that $(H')^\circ$ is simply connected and the map $G \to H$ factors through $H'$. 

\medskip

By Step 1, we know that there is an open subgroup $U$ of $\Gal(\BQ)$ such that the map $G \to H'$ is $U$-equivariant, where $H'$ has the split equivariance structure. 

\medskip

Since both $H'$ and $H$ have the split equivariance structure, after replacing $U$ by a finite index open subgroup the map $H' \to H$ is $U$-equivariant as well.

\qed{[\propref{prop:gal-continuity}]}

\end{document}